\documentclass[a4paper,11pt]{amsart}
\usepackage[top=2.2cm,left=2.8cm,right=2.8cm,bottom=2.8cm]{geometry}
\usepackage[foot]{amsaddr}

\usepackage{hyperref}
\hypersetup{
    colorlinks,
    citecolor=green,
    filecolor=black,
    linkcolor=blue,
    urlcolor=blue
}
\hypersetup{linktocpage}
\usepackage[table]{xcolor}
\usepackage{booktabs}
\usepackage{tabularx}
\usepackage{tabu}
\usepackage{cite}
\usepackage{graphicx}

\usepackage{amsbsy,bbm}
\usepackage{latexsym}
\usepackage{amsfonts}
\usepackage{amssymb}
\usepackage{amsmath,amsthm,stackengine}
\usepackage{enumerate}

\usepackage{tikz}
 \usetikzlibrary{arrows.meta}
\usetikzlibrary{arrows}
\usetikzlibrary{matrix}
\usepackage{mathdots}
\usepackage{mathtools}
\usepackage{mathdots}

\usepackage[ruled]{algorithm2e}

\usepackage{scalerel}

\usepackage{stmaryrd}
\usepackage{dsfont}
\usepackage{kbordermatrix}
\usepackage{bm}

\usepackage{pgfplots}

\DeclareMathOperator{\Span}{span}

\newcommand\diag{\mathop{\rm diag}}
\newcommand\bidiag{\mathop{\rm bidiag}}

\newcommand\supp{\mathop{\rm supp}}

\providecommand{\abs}[1]{\lvert#1\rvert}

\providecommand{\norm}[1]{\lVert#1\rVert}

\newtheorem{theorem}{Theorem}
\newtheorem{lemma}[theorem]{Lemma}
\newtheorem{prop}[theorem]{Proposition}
\newtheorem{cor}[theorem]{Corollary}

\theoremstyle{definition}

\newtheorem{example}[theorem]{Example}

\theoremstyle{remark}
\newtheorem{remark}[theorem]{Remark}

\newcommand{\cK}{{\mathcal{K}}}

\newcommand{\cN}{{\mathcal{N}}}
\newcommand{\cF}{{\mathcal{F}}}

\newcommand{\cS}{\mathcal{S}}

\newcommand{\Chi}{\raise .3ex
\hbox{\large $\chi$}}

\newcommand{\R}{\mathbb{R}}
\newcommand{\N}{\mathbb{N}}

\newcommand{\C}{\mathbb{C}}

\newcommand{\ten}[1]{\boldsymbol{#1}}
\newcommand{\rep}[1]{\mathsf{#1}}
\newcommand{\rk}[1]{\mathsf{#1}}

\newcommand{\sst}[1]{{\scaleto{#1}{4.5pt}}}

\newcommand{\rmap}[1]{\tau(#1)}

\newcommand{\rmapless}[3]{\tau^\sst{<}_{#1,#2} (#3)}
\newcommand{\rmapleq}[3]{\tau^\sst{\leq}_{#1,#2} (#3)}

\newcommand{\rmapgtr}[3]{\tau^\sst{>}_{#1,#2} (#3)}
\newcommand{\rmapgeq}[3]{\tau^\sst{\geq}_{#1,#2} (#3)}

\newcommand{\opleq}[2]{#1^\sst{\leq}_{#2}}
\newcommand{\opgtr}[2]{#1^\sst{>}_{#2}}

\DeclareMathOperator*{\SKP}{\Join}

\newcommand{\unocc}[1]{\colorbox{gray!30}{\makebox(8,8){$\mathsf{#1}$}}}
\newcommand{\occ}[1]{\colorbox{black}{\textcolor{white}{\makebox(8,8){$\mathsf{#1}$}}}}

\newcommand{\vac}{\ten{e}_\mathrm{vac}}

\SetArgSty{textup}

\numberwithin{equation}{section}
\numberwithin{theorem}{section}

\title[Block Structures in Matrix Product States]{Particle Number Conservation and Block Structures in Matrix Product States}
\author{Markus Bachmayr$^1$}
\address{\rm $^1$ Institut f\"ur Mathematik, Johannes Gutenberg-Universit\"at Mainz, Staudingerweg 9, 55128 Mainz, Germany}
\email[Markus Bachmayr]{bachmayr@uni-mainz.de}
\author{Michael G\"otte$^2$}
\address{\rm $^2$ Institut f\"ur Mathematik, Technische Universit\"at Berlin, Stra\ss e des 17. Juni 136, 10623 Berlin, Germany}
\email[Michael G\"otte]{goette@math.tu-berlin.de}
\author{Max Pfeffer$^3$}
\address{\rm $^3$ Fakult\"at f\"ur Mathematik, Technische Universit\"at Chemnitz, Reichenhainer Str.~41, 09107 Chemnitz, Germany}
\email[Max Pfeffer]{max.pfeffer@math.tu-chemnitz.de}
\thanks{M.B.\ acknowledges funding by
 Deutsche Forschungsgemeinschaft (DFG, German Research Foundation) -- Projektnummern 233630050; 211504053 -- TRR 146; SFB 1060. 
M.P. was funded by DFG -- Projektnummern 448293816; 211504053 -- SFB 1060. M.G.\ was funded by DFG (SCHN530/15-1).}

\date{\today}

\begin{document}

\maketitle

\vspace{-18pt}
\begin{abstract}
The eigenvectors of the particle number operator in second quantization are characterized by the block sparsity of their matrix product state representations. This is shown to generalize to other classes of operators. Imposing block sparsity yields a scheme for conserving the particle number that is commonly used in applications in physics. Operations on such block structures, their rank truncation, and implications for numerical algorithms are discussed. Explicit and rank-reduced matrix product operator representations of one- and two-particle operators are constructed that operate only on the non-zero blocks of matrix product states.

\smallskip

\end{abstract}

\section{Introduction}

In wavefunction methods of quantum chemistry, one aims to directly approximate the wavefunction of the electrons in a given molecular system.
In the many-electron case, these are defined on extremely high-dimensional spaces. In addition, due to the fermionic nature of electrons, these wavefunctions need to respect certain antisymmetry requirements. Post-Hartree-Fock methods are an established class of wavefunction methods based on approximations of wavefunctions by \emph{Slater determinants}, that is, by antisymmetrized tensor products of single-electron basis functions on $\R^3$. With a judicious choice of such lower-dimensional basis functions, called \emph{orbitals}, these methods can achieve high-accuracy approximations of the wavefunctions corresponding to the lowest-energy states of the system. This is achieved essentially by exploiting near-sparsity of wavefunctions in the basis of all Slater determinants formed from the orbitals.

However, for certain types of problems, for instance \emph{strongly correlated} systems with several competing states of lowest energy, these classical methods typically fail to yield good approximations. Thus more flexible data-sparse parametrizations of the linear combinations of Slater determinants that can be formed from a given finite set of orbitals are of interest. An elegant way of representing such linear combinations of antisymmetric functions is the formalism of \emph{second quantization}, where wavefunctions are represented in terms of the occupation of each orbital by a particle. With respect to a sequence of orthonormal orbitals $\{ \phi_k\}_{k\in\N}$, this leads to a representation of the wavefunction by \emph{occupation numbers} $(\C^2)^\infty$, corresponding to an occupied and an unoccupied state for each orbital. The corresponding space of functions is called \emph{Fock space}.

For electrons with Coulomb interaction in an external potential $V$, one has the corresponding representation of the Hamiltonian acting on occupation number tensors,
\begin{equation}\label{eq:hamil}
    \ten{H} = \sum_{i,j} t_{ij} \ten{a}_{i}^* \ten{a}_{j}  + \sum_{i,j,k,l} v_{ijkl} \ten{a}^*_{i} \ten{a}^*_{j} \ten{a}_{k} \ten{a}_{l}
\end{equation}
 with coefficient tensors $(t_{ij})$ and $(v_{ijkl})$ depending on the orbitals,
 in terms of the \emph{creation operators} $\ten{a}^*_i$ and \emph{annihilation operators} $\ten{a}_i$. These can be thought of as switching particles from the unoccupied to the occupied state of orbital $i$ or back, respectively. The antisymmetry of wavefunctions corresponds to the anticommutation relations
 \begin{equation}\label{eq:anticomm}
     \ten{a}_{i} \ten{a}^*_{j} + \ten{a}^*_{j} \ten{a}_{i} = \delta_{ij}, \quad   \ten{a}^*_{i} \ten{a}^*_{j} + \ten{a}^*_{j} \ten{a}^*_{i} =   \ten{a}_{i} \ten{a}_{j} + \ten{a}_{j} \ten{a}_{i} = 0\,.
 \end{equation}

The second-quantized representation is particularly suitable for the application of low-rank tensor formats such as \emph{matrix product states} (abbreviated MPS), also known as \emph{tensor trains} (abbreviated TT) in the numerical analysis context, or the more general \emph{tree tensor networks} (or \emph{hierarchical tensors}); see \cite{Oseledets:2011:TT,Hackbusch:09,BSU:16}.
Whereas the implementation of wavefunction antisymmetry in such tensor formats is problematic in the real-space representation of wavefunctions, this does not present a problem in the second-quantized representation: the antisymmetry properties are encoded in the representation of operators, and the corresponding occupation numbers describing the wavefunctions can be directly approximated in low-rank tensor formats.
However, in contrast to real-space approximations of wavefunctions \cite{Hackbusch:18}, where the number of electrons is tied to the spatial dimensionality of the problem, this particle number is not fixed in the second-quantized formulation and thus needs to be prescribed explicitly.
 
Prescribing a number of $N$ particles amounts to restricting the eigenvalue problem for $\ten{H}$ to the subspace of those occupation numbers that are also eigenvectors of the
\emph{particle number operator}
 \[
  \ten{P} = \sum_{i} \ten{a}_i^* \ten{a}_i
\]
with eigenvalue $N$. The particle number constraint does not need to be implemented explicitly: as we investigate in detail in this work, every particle number eigenspace corresponds to a certain \emph{block-sparse structure} in the cores of MPS. This fact has a long history in the physical literature (see, e.g., \cite{OR95,Daley:04,McCulloch:07,Schollwoeck:11,SPV11,bauer2011implementing}), where such block sparsity is usually derived from gauge symmetries, such as $U(1)$ symmetry corresponding to particle number conservation. 
Here we use elementary linear algebra to arrive at this block structure, which to the best of our knowledge has not received any attention thus far in a mathematical context. 

Block sparsity can not only be used to build the particle number constraint into the low-rank tensor representations, but it can also be exploited to reduce the costs of operations on MPS. We also consider the implications for analogous representations of linear operators acting on MPS, which are called \emph{matrix product operators} (MPOs). As we show, these can be applied in a form that preserves their low-rank structure and at the same time maintains the block structure of MPS.
 
The existence of such block structures is commonly exploited in applications of MPS in physics for solving eigenvalue problems for general Hamiltonians with one- and two-particle interactions as in \eqref{eq:hamil}; see, for instance, 
\cite{itensor,tensornetwork,tenpy,pytenet}. In these applications, the focus is mainly on \emph{density matrix renormalization group} (DMRG) algorithms \cite{white,Schollwoeck:11}, which operate locally on components of the MPS. The block structures of the MPS in this case appear as block diagonal structures of density matrices computed from the MPS. However, DMRG schemes are known to fail in certain circumstances \cite{Dolfietal:12}, a fact that is related to their local mode of operation.

For designing eigenvalue solvers for MPS that can be guaranteed to converge, an important building block is the eigenvalue residual $\ten{H}\ten{x} -\langle \ten{H}\ten{x},\ten{x}\rangle \ten{x}$. For its efficient evaluation for an MPO representation of $\ten{H}$ and an MPS representation of $\ten{x}$, the respective global block structures of these quantities that we consider here become important. In addition to describing the block structure-preserving action of $\ten{H}$, we also show that the representation rank of $\ten{H}$ can be substantially reduced if the tensor $(v_{ijkl})$ satisfies certain sparsity conditions that can be satisfied with a suitable choice of orbitals.
 
 As one main contribution of this work, we thus consider from the point of view of numerical linear algebra the use of block-sparse MPS as a means of enforcing particle number constraints in solving eigenvalue problems as they arise in quantum chemistry. In particular, we consider the realization of basic operations on block-structured MPS, how Hamiltonians acting on MPS in a compatible block-structured form can be implemented, and their respective computational complexity. We also consider some basic effects of the block structure of MPS on the convergence of eigensolvers.
 
 More generally, we show that a similar block structure is present whenever MPS (or tensor trains) are restricted to an eigenspace of a diagonal operator with a certain Laplacian-type structure, with the particle number operator as a particular example. 
Without explicitly enforcing the block structure, in exact arithmetic such a constraint is also preserved by many operations on MPS; due to issues of numerical stability, however, this is in general no longer true in numerical computations: when working on full MPS without explicit block structure, the particle number will in general accumulate numerical errors over the course of iterative schemes.

The outline of this paper is as follows: In Section \ref{sec:prelim}, we introduce basic notions and notation of MPS. In Section \ref{sec:blockstructure}, we consider the block structure of MPS under particle number constraints and some of their consequences, and we consider the realization of standard operations on MPS exploiting this block structure in Section \ref{sec:blockops}. 
In Section \ref{sec:mpo}, we consider low-rank representations of one- and two-electron operators in Hamiltonians as well as their interaction with block-structured MPS.
Finally, in \ref{sec:numer}, we discuss basic implications for iterative eigensolvers and numerical illustrations.

\section{Preliminaries}\label{sec:prelim}

Since we are mainly interested in real-valued Hamiltonians as they arise in molecular systems, we restrict ourselves to real-valued occupation numbers in $(\R^2)^\infty$. However, the following considerations immediately generalize to the complex-valued case.
We consider Fock space restricted to a fixed number $K \in \N$ of orbitals, corresponding to occupation numbers in $\cF^K := (\R^2)^{K}$, which we regard as tensors of order $K$ with indices $\alpha \in \{ 0, 1\}^K$. This space is spanned by the unit vectors $e^\alpha = e^{\alpha_1} \otimes \cdots \otimes e^{\alpha_K}$, where $e^{\alpha_k}=(\delta_{\alpha_k,\beta})_{\beta =0,1}$ are Kronecker vectors for $k = 1,\ldots, K$.

\subsection{Matrix product states and operators}

In our notation, we follow \cite{kazeev_low-rank_2012,BK:20} with some adaptations.
The \emph{matrix product state} (or \emph{tensor train}) representation of $\ten{x} \in \cF^K$ with \emph{ranks} $r_1, \ldots, r_{K-1} \in \N_0$ reads
\begin{equation*} \label{eq:mps}
\ten{x}_\alpha = \ten{x}_{\alpha_1, \ldots, \alpha_K} =
		\sum_{j_1=1}^{r_1}
		\cdots
		\sum_{j_{K-1}=1}^{r_{K-1}}
		X_1(j_0, \alpha_1, j_1)
		X_2(j_1, \alpha_2, j_2)
		\, \cdots \,
		X_K(j_{K-1}, \alpha_K, j_K),
\end{equation*}
where for notational reasons we set $j_0 = j_K = 1$, $r_0 = r_K = 1$.
For the third-order component tensors $X_k$ in such a representation, called \emph{cores}, we write
\begin{equation}\label{eq:cores}
 \rep{X} = (X_1,\ldots, X_K),\quad    X_k = \bigl(  X_k(j_{k-1}, \alpha_k, j_k)  \bigr)_{\substack{j_{k-1}  = 1,\ldots, r_{k-1}, \\ \alpha_k = 0,1,  \\ j_k = 1,\ldots,r_k}} \,.
\end{equation}
For linear mappings on $\cF^K$, we have an analogous \emph{matrix product operator} (MPO) representation
\begin{equation}\label{eq:mpo}
  \ten{M}_{\alpha_1, \ldots, \alpha_K,\beta_1, \ldots, \beta_K}
   = \sum_{j_1=1}^{r_1}
		\cdots
		\sum_{j_{K-1}=1}^{r_{K-1}}
		M_1(j_0, \alpha_1, \beta_1, j_1)
		\, \cdots \,
		M_K(j_{K-1}, \alpha_K, \beta_K , j_K),
\end{equation}
where we similarly write $\rep{M} = (M_1,\ldots,M_K)$.

For specifying the $k$-th component of an MPS or MPO explicitly, we use the notation
\begin{equation*}\label{eq:corenotation1}
   X^{ [j_{k-1}, j_k] }_k = \bigl( X_k(j_{k-1}, \alpha_k, j_k) \bigr)_{\alpha_k=0,1},   \quad 
   M^{ [j_{k-1}, j_k] }_k = \bigl( M(j_{k-1}, \alpha_k, \beta_k, j_k) \bigr)^{\beta_k=0,1}_{\alpha_k=0,1}\,.
\end{equation*}
Note that here and in the following, we write row indices in subscript and column indices in superscript, indicating that 
$M_k^{ [ j_{k-1}, j_k ] }$ is a matrix. In terms of the vectors $X_k^{[j_{k-1}, j_k]} \in \R^{\{0,1\}}$, a core $X_k$ is then given by the rankwise block representation
\begin{equation}\label{eq:corenotation2}
   X_k = \begin{bmatrix} X_k^{[1,1]} & \cdots &  X_k^{[1,r_k]}  \\ \vdots & \ddots & \vdots \\ X_k^{[r_{k-1},1]} & \cdots & X_k^{[r_{k-1},r_k]}  \end{bmatrix},
\end{equation}
with the analogous notation for the matrices $ M_k^{ [j_{k-1}, j_k] }$. 

We define multiplication of a core $X_k$ by a matrix $G$ of appropriate size from the left or right on its indices $j_{k-1}$ or $j_k$, respectively, by
\[
\begin{aligned}
  (GX_k)(j_{k-1}, \alpha, j_k)  &=  \sum_{j' = 1}^{r_{k-1}} G_{j_{k-1},j'} X_k(j', \alpha_k, j_k) , \\
   (X_k G)(j_{k-1},\alpha,j_k)  &=  \sum_{j' = 1}^{r_{k}} X_k(j_{k-1}, \alpha_k, j') G_{j', j_k}, 
\end{aligned}
\]
with the analogous definition for the components $M_k$ in the representation of operators.

Complementing \eqref{eq:corenotation2}, we also introduce
\[
   X^{ \{ \alpha_k \} }_k  = \bigl( X_k(j_{k-1}, \alpha_k, j_k) \bigr)_{j_{k-1}=1,\ldots,r_{k-1}}^{j_k=1,\ldots,r_k} , \quad  M_k^{ \{ \alpha_k, \beta_k \} }  = \bigl( M_k(j_{k-1}, \alpha_k, \beta_k, j_k) \bigr)_{j_{k-1}=1,\ldots,r_{k-1}}^{j_k=1,\ldots,r_k} \,.
\]
Again, the subscripts and superscripts indicate that $X^{ \{ \alpha_k \} }_k$ and $M_k^{ \{ \alpha_k, \beta_k \} }$ are matrices.

For a compact way of writing \eqref{eq:mps} in terms of the cores \eqref{eq:cores}, we introduce the \emph{Strong Kronecker product}, 
\[
   ( X_1 \SKP X_2 )^{ \{ \alpha_1 \alpha_2 \} } = X_1^{ \{ \alpha_1 \}  }  X_2^{ \{ \alpha_2 \}  }, 
   \quad 
    ( M_1 \SKP M_2 )^{ \{ \alpha_1 \alpha_2, \beta_1 \beta_2 \} } = M_1^{ \{ \alpha_1 , \beta_1\}  }  M_2^{ \{ \alpha_2, \beta_2 \}  }.
\]
For example, for two cores $X, Y$ of ranks $2\times 2$, we obtain
\[
\begin{multlined}
  	\begin{bmatrix}
	X^{[1,1]} & X^{[1,2]}\\
	X^{[2,1]} & X^{[2,2]}\\
	\end{bmatrix}
	\SKP
	\begin{bmatrix}
	Y^{[1,1]} & Y^{[1,2]}\\
	Y^{[2,1]} & Y^{[2,2]}\\
	\end{bmatrix}  \qquad \\
 \qquad	=
	\begin{bmatrix}
	X^{[1,1]} \otimes Y^{[1,1]} + X^{[1,2]} \otimes Y^{[2,1]} & X^{[1,1]} \otimes Y^{[1,2]} + X^{[1,2]} \otimes Y^{[2,2]}\\
	X^{[2,1]} \otimes Y^{[1,1]} + X^{[2,2]} \otimes Y^{[2,1]} & X^{[2,1]} \otimes Y^{[1,2]} + X^{[2,2]} \otimes Y^{[2,2]}\\
	\end{bmatrix}
	\end{multlined}\,.
\]
With this notation, we have
\[
     [\ten{x}] = X_1 \SKP X_2 \SKP \cdots \SKP X_K\,,
\]
where the block $[\ten{x}] \in \R^{ 1 \times \{0,1\}^K \times 1}$ has leading and trailing dimensions of mode size 1.
For simplicity, we ignore such singleton dimensions, that is,
we identify $[\ten{x}]$ with the tensor $\ten{x} \in \R^{\{0,1\}^K}$ of order $K$. With this identification, for the \emph{representation mapping} $\tau$, we write \eqref{eq:mps} as
\[
   \rmap{\rep{X}} :=  X_1 \SKP X_2 \SKP \cdots \SKP X_K,
\]
where we have $\ten{x} = \rmap{\rep{X}}$. We use partial representation mappings that assemble the first or last cores of a matrix product state. We set
\[
   \rmapless{k}{j}{\rep{X}} = \Bigl(  ( X_1 \SKP \cdots \SKP X_{k-1} )_{1,\alpha,j} \Bigr)_{\alpha \in \{0,1\}^{k-1}}, 
   \quad    \rmapleq{k}{j}{\rep{X}} = \Bigl(  ( X_1 \SKP \cdots \SKP X_{k} )_{1,\alpha,j} \Bigr)_{\alpha \in \{0,1\}^{k}},
\]
and analogously
\[
 \rmapgtr{k}{j}{\rep{X}} = \Bigl(  ( X_{k+1} \SKP \cdots \SKP X_{K} )_{1,\alpha,j} \Bigr)_{\alpha \in \{0,1\}^{K-k-1}}, 
  \quad
  \rmapgeq{k}{j}{\rep{X}}  = \Bigl(  ( X_{k} \SKP \cdots \SKP X_{K} )_{1,\alpha,j} \Bigr)_{\alpha \in \{0,1\}^{K-k}} \,.
\]
The representation $\rep{X}$ is called \emph{left-orthogonal} if the vectors 
$\rmapless{K}{j}{\rep{X}}$ are orthonormal for $j=1,\ldots,r_{K-1}$, and \emph{right-orthogonal} if 
$\rmapgtr{1}{j}{\rep{X}}$ are orthonormal for $j=1,\ldots,r_{1}$. 

As a second product operation between an MPS core $X_k$ of ranks $r$ and an MPO core $M_k$ of ranks $r'$(or analogously between two MPO cores), with $j_k = 1,\ldots,r_k$, $j_k'=1,\ldots,r_k'$ for each $k$, we introduce the \emph{mode core product},
\[
\left( M_k \bullet X_k \right)^{ [ r_{k-1}'(j_{k-1}-1) + j'_{k-1}, \,r_k' (j_k-1) + j_k'] } = M_k^{[j_{k-1}', j_k']} X_k^{ [j_{k-1}, j_k] } \,.
\]
For matrix product operators a $\ten{M} = \rmap{\rep{M}}$ and matrix product states $\ten{x} = \rmap{\rep{X}}$, we have
\[
     \ten{M} \ten{x}  = \tau( M_1 \bullet X_1, \ldots, M_K \bullet X_K) \,.
\]

Finally, we introduce a \emph{lift product} of a matrix $W \in \R^{r \times r'}$ with a matrix $M \in \R^{2 \times 2}$ or vector $x \in \R^2$, that is to be understood as a Kronecker product with reordered indices,
\begin{align*}
&\uparrow: \mathbb{R}^{r \times r'} \times \mathbb{R}^{2\times 2} \rightarrow  \mathbb{R}^{r \times 2 \times 2 \times r'}, & (W \uparrow M)_{j \alpha \beta j'} &= W_{j j'}M_{\alpha \beta}, \\
&\uparrow: \mathbb{R}^{r \times r'} \times \mathbb{R}^{2} \rightarrow  \mathbb{R}^{r \times 2 \times r'}, & (W \uparrow x)_{j \alpha j'} &= W_{j j'}x_{\alpha},
\end{align*}  
thus resulting in an MPO core or an MPS core, respectively.

\subsection{Singular value decomposition}
For $k=1,\ldots,K$, using $(\alpha_1,\ldots,\alpha_k)$ as row index and $(\alpha_{k+1},\ldots, \alpha_K)$ as column index, one obtains corresponding \emph{matricizations} (or \emph{unfoldings}) of a tensor $\ten x \in \cF^K$.
Any representation $\rep X$ of $\ten x$ can be transformed by operations on its cores such that these matricizations are in SVD form. This representation is known as \emph{Vidal decomposition} \cite{vidal03} or as \emph{tensor train SVD (TT-SVD)} \cite{Oseledets:2011:TT} and also arises as a special case of the \emph{hierarchical SVD} \cite{Grasedyck:2010:HierarchicalSVD} of more general tree tensor networks.

Specifically, $\rep X$ with $\ten x=\rmap{\rep X}$ is in \emph{left-orthogonal TT-SVD form} if for $k = 1,\ldots, K-1$, $\{ \rmapleq{k}{j}{\rep{X}}  \}_{j = 1,\ldots, r_{k}}$ are orthonormal and $\{ \rmapgtr{k}{j}{\rep X} \}_{j = 1,\ldots, r_{k}}$ are orthogonal, where $\sigma_{k,j}(\ten x) := \norm{\rmapgtr{k}{j}{\rep X}}_2$ with $\sigma_{k,1}(\ten x) \geq \ldots \geq \sigma_{k,r_k}(\ten x)$ are the singular values of the $k$-th matricization. 
Analogously, $\rep X$ is in \emph{right-orthogonal TT-SVD form} if for $k = 1,\ldots, K-1$, $\{ \rmapleq{k}{j}{\rep{X}}  \}_{j = 1,\ldots, r_{k}}$ are orthogonal with $\norm{\rmapleq{k}{1}{\rep X}}_2 \geq \ldots \geq \norm{\rmapleq{k}{k}{\rep X}}_2$ and $\{ \rmapgtr{k}{j}{\rep X} \}_{j = 1,\ldots, r_{k}}$ are orthonormal.
These forms can be obtained by the scheme given in \cite[Algorithm~1]{Oseledets:2011:TT}.

The rank truncation of either SVD form yields quasi-optimal approximations of lower ranks \cite{Oseledets:2011:TT,Grasedyck:2010:HierarchicalSVD}: let $\rep X$ be given in TT-SVD form with ranks $r_1,\ldots,r_{K-1}$, and denote by $\operatorname{trunc}_{s_1,\ldots,s_K}(\rep X)$ its truncation to ranks $s_k\leq r_k$, $k=1,\ldots,K$, then
\begin{equation}\label{eq:ttsvdquasiopt}
\begin{multlined}
  \norm{ \rmap{\rep X} - \rmap{\rep{\operatorname{trunc}}_{s_1,\ldots,s_K}(\rep X)} }_2
   \leq \biggl(\sum_{k=1}^{K-1} \sum_{j=1}^{r_k} \sigma_{k,j}^2 \biggr)^{\frac12} \\
  \qquad\qquad\qquad \leq \sqrt{K-1} \min\bigl\{ \norm{\rmap{\rep X} - \rmap{\rep Y}}_2 \colon \rep Y \text{ of ranks $s_1,\ldots, s_{K-1}$} \bigr\}.
   \end{multlined}
\end{equation}
Using the above error bound in terms of the matricization singular values, one obtains an approximation $\operatorname{trunc}_\varepsilon(\rep X)$ with $\norm{\rmap{\rep X} - \rmap{\operatorname{trunc}_\varepsilon(\rep X)}}_2 \leq \varepsilon$ for any $\varepsilon>0$ by truncating ranks according to the smallest singular values.

\subsection{Tangent space projection}\label{ssc:tangentspace}
It is well known that MPS of fixed multilinear rank constitute an embedded smooth submanifold of the tensor space $\cF^K$ \cite{holtz_manifolds_2012}. The tangent space of this manifold can be explicitly characterized, and the ranks of the tangent vectors at $\ten x$ in MPS representation are at most twice the ranks of $\ten x$. 

Let $\ten x = \rmap{\rep{U}} = \rmap{\rep V}$ where $\rep{U} = (U_1,\cdots,U_K)$ is in left- and $\rep{V} = (V_1,\cdots,V_K)$ is in right-orthogonal form.
The projection operator onto the tangent space at $\ten x$ is given by
\begin{equation*}
\ten{Q}_{\ten{x}}= \sum_{k=1}^K \bigl( \ten Q_{\ten x}^{k,1} - \ten Q_{\ten x}^{k,2} \bigr),
\end{equation*}
where, identifying mappings with their representation matrices, for $k = 1,\ldots,K$,
\begin{equation*}
\ten Q_{\ten x}^{k,1} = \biggl(\sum_{j=1}^{r_{k-1}}\rmapless{k}{j}{\rep{U}}\, \langle \rmapless{k}{j}{\rep{U}},\, \cdot\, \rangle \biggr)
 \otimes I \otimes \biggl(\sum_{j=1}^{r_{k}}{\rmapgtr{k}{j}{\rep{V}}}\,\langle{\rmapgtr{k}{j}{\rep{V}}},\,\cdot\,\rangle \biggr),
\end{equation*}
for $k = 1,\ldots,K-1$,
\begin{equation*}
\ten Q_{\ten x}^{k,2} = \biggl(\sum_{j=1}^{r_{k}}\rmapleq{k}{j}{\rep{U}}\,\langle \rmapleq{k}{j}{\rep{U}},\,\cdot\,\rangle \biggr) \otimes \biggl(\sum_{j=1}^{r_{k}}{\rmapgtr{k}{j}{\rep{V}}}\,\langle {\rmapgtr{k}{j}{\rep{V}}},\,\cdot\,\rangle\biggr),
\end{equation*}
and $\ten Q_{\ten x}^{K,2} = 0$ for $k=K$. 

\subsection{Second quantization} The operators of second quantization can be represented as mappings on $\cF^K$ as follows: with the elementary components
\[
  S = \begin{pmatrix} 1 & 0 \\ 0 & -1 \end{pmatrix}, \quad
   A = \begin{pmatrix} 0 & 1 \\ 0 & 0 \end{pmatrix}, \quad 
   I = \begin{pmatrix} 1 & 0 \\ 0 & 1 \end{pmatrix},
\]
the \emph{annihilation operator} $\ten{a}_i$ on $\cF^K$ reads
\begin{equation}\label{def:anniloperator}
  \ten{a}_i =  \biggl( \bigotimes_{k=1}^{i-1} S \biggr) \otimes A \otimes \biggl( \bigotimes_{k=i+1}^{K} I \biggr),
\end{equation}
and the corresponding \emph{creation operator} is $\ten{a}_i^*$.
The particle number operator on $\cF^K$ is given by
\[
  \ten{P} = \sum_{i=1}^K \ten{a}_i^* \ten{a}_i.
\]
In addition, we introduce the truncated versions
\begin{align}\label{eq:partnumtrunc}
 \opleq{\ten{P}}{k} = \sum_{i = 1}^k  \biggl( \bigotimes_{\ell=1}^{i-1} I \biggr) \otimes A^*A \otimes \biggl( \bigotimes_{\ell=i+1}^{k} I \biggr),
  \quad
    \opgtr{\ten{P}}{k} = \sum_{i = k+1}^K  \biggl( \bigotimes_{\ell=k+1}^{i-1} I \biggr) \otimes A^*A \otimes \biggl( \bigotimes_{\ell=i+1}^{K} I \biggr),
\end{align}
which act only on the left and right sections, respectively, of a matrix product state.

\section{Block Structure of Matrix Products States}\label{sec:blockstructure}

In this section we characterize the block sparsity of an MPS $\rep{X}$ such that $\ten{x}=\rmap{\rep{X}}$ is an eigenvector of the particle number operator $\ten{P}$, or in fact of any operator that shares a certain structural feature of $\ten{P}$. 
We first formulate the result for general tensors $\ten x \in \mathbb{R}^{n_1\times\cdots \times n_K}$ and then obtain the corresponding result for eigenvectors of $\ten{P}$ in $\cF^K$ as a special case.
While the definitions of Section \ref{sec:prelim} are given for tensors in $\cF^K$ for simplicity, they immediately carry over to general tensors in $\mathbb{R}^{n_1\times\cdots \times n_K}$ with indices in 
$\mathcal{N} = \bigtimes_{k=1}^K \{0,1,\ldots,n_k-1\}$,
 where we abbreviate partial index sets for modes $k_1$ to $k_2$ by
$\mathcal{N}_{k_1}^{k_2} = \bigtimes_{k={k_1}}^{k_2} \{0,1,\ldots,n_k-1\}$.

For $\ten{P}$, the block sparsity of $\rep{X}$ that we obtain is of the following form:  For each $k$, the matrices $X^{\{0\}}_k$ and $X^{\{1\}}_k$ have block structure with nonzero blocks only on the main diagonal for $X^{\{0\}}_k$, and only on the first superdiagonal for $X^{\{1\}}_k$. Intuitively, this can be interpreted as follows: each block corresponds to a certain number of occupied orbitals to the left of $k$. For $X^{\{0\}}_k$, this number does not change. For the occupied state, in $X^{\{1\}}_k$ the positions of the blocks correspond to increasing the number of particles by one.

More generally, we obtain such block sparsity for so-called \emph{Laplace-like} operators \cite{kazeev_low-rank_2012} on $\mathbb{R}^{n_1\times\cdots \times n_K}$ of the form
\begin{subequations}\label{eq:Ldef}
\begin{equation}\label{eq:Ldef1}
\ten{L} = \sum_{k=1}^K  \biggl( \bigotimes_{\ell=1}^{k-1} I \biggr) \otimes L_k \otimes \biggl( \bigotimes_{\ell=k+1}^{K} I \biggr)
\end{equation}
with diagonal matrices 
\begin{equation}\label{eq:Ldef2}
L_k = \text{diag}(\lambda_{k,0},\lambda_{k,1},\ldots,\lambda_{k,n_k-1}),
\end{equation}
\end{subequations}
where in the case of $\ten{P}$, we have $L_k = A^*A$.
The unit vectors $e^{\alpha_1} \otimes \cdots \otimes e^{\alpha_K}$ are eigenvectors of such $\ten{L}$, with eigenvalues given by 
\[ \lambda_\alpha = \sum_{k=1}^K \lambda_{k,\alpha_k}, \quad{\alpha\in \mathcal{N}}.
\]

\begin{remark}
	Note that for any $\ten L$ of the form \eqref{eq:Ldef1} with general symmetric matrices $L_k$, there exists $\ten U = U_1 \otimes \cdots \otimes U_K$ with orthogonal matrices $U_1,\ldots, U_K$ such that $\ten{\tilde L} = \ten U \ten L \ten U^\top$ satisfies also \eqref{eq:Ldef2}, and thus the following considerations apply to $\ten{ \tilde L}$. 
\end{remark}

Let $\ten{x} = \rmap {\rep X}$ as above satisfy $\ten{L}\ten{x} = \lambda \ten{x}$, $\ten x \neq 0$. For such $\lambda$, we define the subset $I_\lambda \subset  \mathcal{N}$ of all $\alpha$ such that $\lambda = \lambda_\alpha$. Furthermore, for each $\alpha\in I_\lambda$ and $k=1,\ldots,K-1$ we can split the eigenvalue $\lambda$ in the form
\[
\lambda = \lambda_{k,\alpha}^\sst{\leq} + \lambda_{k,\alpha}^\sst{>} := \sum_{\ell=1}^k \lambda_{\ell,\alpha_\ell} +\sum_{\ell=k+1}^K \lambda_{\ell,\alpha_\ell}.
\]
Using the notation from \eqref{eq:partnumtrunc}, the summands $\lambda_{k,\alpha}^\sst{\leq }, \lambda_{k,\alpha}^\sst{>}$ are eigenvalues of the truncated versions $\ten{L}_k^\sst{\leq}$ and $\ten{L}_k^\sst{>}$ of $\ten{L}$. We write $\cK_{\lambda,k}$ for the set of all $\lambda_{k,\alpha}^\sst{\leq }$ for given $\lambda$ and $k$, that is,
\[
  \cK_{\lambda,k} = \left\{  \sum_{\ell=1}^k \lambda_{\ell,\alpha_\ell}  \colon \alpha \in I_\lambda   \right\}{,}
\]
{where we have $\mathcal{K}_{\lambda,0} = \{0\}$}.
In full representation, $\ten x$ necessarily has a certain sparsity pattern, since $\ten x_\alpha = 0$ if $\alpha \notin I_\lambda$. We can exploit the invariance 
\begin{equation}
\begin{aligned}\label{eq:rotmps}
\ten x &= X_1 \SKP X_2 \SKP \cdots \SKP X_{K-1} \SKP X_K \\
&= X_1 G_1 \SKP G_1^{-1} X_2 G_2 \SKP \cdots \SKP G_{K-2}^{-1} X_{K-1} G_{K-1} \SKP G_{K-1}^{-1} X_K 
\end{aligned}	
\end{equation}
for invertible $G_k \in \R^{r_k \times r_k}$ for $k = 1,\ldots,K-1$ in order to obtain a block structure for the component tensors $X_k$ , which is our following main result. {Here and in what follows, for $\cK \subset \R $ and $\lambda \in \R$, we write $\mathcal{K} - \lambda := \{\mu \in \R: \mu+\lambda \in \mathcal{K} \}$.}

\begin{theorem}\label{thm:generalblockstructure}
	Let $\ten{x} \in \mathbb{R}^{n_1\times\cdots\times n_K}$, $\ten{x}\neq 0$, have the representation $\ten{x} = \rmap{\rep{X}}$ with minimal ranks $\rk{r} = (r_1,\ldots, r_{K-1})$.
	Then one has $\ten{L} \ten{x} = \lambda \ten{x}$ with $\ten L$ as in \eqref{eq:Ldef} precisely when $\rep{X}$ can be chosen such that the following holds:
	for $k=1,\ldots, K$ and for all  
$\mu \in \cK_{\lambda,k}$,
	there exist $\cS_{k,\mu} \subseteq \{ 1,\ldots, r_k\}$ 
	such that
	\begin{equation}\label{evpi2}
	\opleq{\ten{L}}{k}  \rmapleq{k}{j}{\rep{X}}   =   \mu   \rmapleq{k}{j}{\rep{X}}, \quad    \opgtr{\ten{L}}{k} \rmapgtr{k}{j}{\rep{X}} = (\lambda-\mu) \rmapgtr{k}{j}{\rep{X}} ,  \quad j \in \cS_{k,\mu}     ,
	\end{equation}
	and the matrices $X^{\{\beta\}}_{k}$, $\beta = 0,1,\ldots,n_k-1$, have nonzero entries only in the blocks 
	\begin{equation}\label{blocki2}    
	\begin{aligned} 
	& X^{\{\beta\}}_{k}\big|_{\cS_{k-1,\mu} \times \cS_{k,\mu+\lambda_{k,\beta}}} & & \text{for $\mu \in \cK_{\lambda,k-1}\cap ( \cK_{\lambda,k}-\lambda_{k,\beta})$},  
	\end{aligned}
	\end{equation}
	where we set $\cS_{0,0} = \cS_{K,\lambda} = \{1\}$.
\end{theorem}

\begin{proof}
We first show that $\ten{L} \ten{x} = \lambda \ten{x}$ implies that \eqref{evpi2}, \eqref{blocki2} hold, proceeding by induction over $k$. Thus, let $k = 1$. For fixed $\beta \in\{ 0,\ldots, n_1-1 \}$ we define
\[
\ten{y}^\beta = \bigl(  \ten{x}_{\beta,\hat\alpha} \bigr)_{\hat \alpha \in \mathcal{N}_2^K}.
\]
Then, by the definition of $\ten{L}$ and by our assumption,
\[
\sum_{\beta=0}^{n_1-1} e^\beta\otimes \opgtr{\ten{L}}{1}\ten{y}^\beta + \lambda_{1,\beta}e^\beta\otimes \ten{y}^\beta = \ten{L} \ten{x} = \lambda\ten{x} =  \sum_{\beta=0}^{n_1-1} e^\beta \otimes \lambda\ten{y}^\beta.
\]
Consequently $\opgtr{\ten{L}}{1}\ten y^\beta = (\lambda - \lambda_{1,\beta})\ten y^\beta $ for each $\beta$, and thus either $\ten y^\beta = 0$ or $\ten{y}^\beta$ is an eigenvector of a self-adjoint linear mapping. Orthogonality of eigenvectors with distinct eigenvalues implies $\langle \ten{y}^{\beta}, \ten{y}^{\beta'}\rangle = 0$ if $ \lambda_{1,\beta} \neq  \lambda_{1,\beta'}$. Writing $\ten{y}^\beta = X_1^{\{\beta\}} X_2 \SKP  \cdots \SKP X_K$,  we obtain
\begin{equation*}\label{eq:orthk1}
0 = \langle \ten{y}^{\beta}, \ten{y}^{\beta'}\rangle = \langle X_1^{\{\beta\}} X_2 \SKP  \cdots \SKP X_K, X_1^{\{\beta'\}} X_2 \SKP  \cdots \SKP X_K \rangle = \langle  G_1 X^{\{\beta\}}_1, X^{\{\beta'\}}_1\rangle
\end{equation*}
with 
\[
{G}_1 =\bigl(  \langle \rmapgtr{1}{j}{\rep{X}},  \rmapgtr{1}{j'}{\rep{X}} \rangle \bigr)_{j, j' = 1,\ldots,r_1},
\]
which is invertible since the ranks of $\rep X$ are minimal.
This means that $X^{\{\beta\}}_1G_1^{1/2}, X^{\{\beta'\}}_1G_1^{1/2}$ are pairwise orthogonal. Thus using \eqref{eq:rotmps}, replacing $X_1$ by $X_1 G_1^{1/2}$ and $X_2$ by $G_1^{-1/2}X_2$, we can ensure that $\langle X^{\{\beta\}}_1, X^{\{\beta'\}}_1 \rangle = 0$ if $ \lambda_{1,\beta} \neq  \lambda_{1,\beta'}$. By minimality of ranks, there exist precisely $r_1$ different $\beta_1,\ldots,\beta_{r_1} \in \{ 0,\ldots,n_1-1\}$ such that $X^{\{\beta_1\}}_1,\ldots, X^{\{\beta_{r_1}\}}_1$ are linearly independent. 
For $\mu \in \cK_{\lambda, 1} = \{ \lambda_{1,\beta} \colon \beta = 0,\ldots,n_1-1\}$, we now define
\[
  \cS_{1,\mu} = \{  j \colon \lambda_{1,\beta_j} = \mu  \}.
\]
Again making use of \eqref{eq:rotmps}, by Householder reflectors we can construct an orthogonal transformation $Q_1 \in \mathrm{O}(r_1)$ such that replacing $X_1$ by $X_1 Q_1$ and $X_2$ by $Q_1^\top X_2$, we have
\[
X_1^{\{ \beta\}} \in \Span \bigl\{ e^{j}\in\mathbb{R}^{r_1}\colon  j\in \mathcal{S}_{1,\lambda_{1,\beta}}    \bigr\}\text{ for $\beta =0,\ldots,n_1-1$.}
\]
This means that
 $X_1^{[1,j]} \in \Span \{ e^{\beta} \colon  \lambda_{1,\beta} = \mu \}$ for $j \in \cS_{1,\mu}$.
Noting that the eigenspaces of $L_1$ are given by $\Span\{ e^{\beta'} \colon \text{$\beta' = 0,\ldots,n_1-1$ with $\mu  =  \lambda_{1,\beta'}$} \}$ for $\mu \in \cK_{\lambda, 1}$, as well as $X_1^{[1,j]} =  \rmapleq{1}{j}{\rep X}$, we thus have
\[
L_1  \rmapleq{1}{j}{\rep X}  = \mu  \rmapleq{1}{j}{\rep X} \;\text{ for $\mu \in \cK_{\lambda,1}$, $j \in \mathcal{S}_{1,\mu}$}.
\]
This shows \eqref{blocki2} and the first statement in \eqref{evpi2} for $k=1$. 
Moreover, combining
\[
 \ten L \ten x = 
  \sum_{\mu \in \cK_{\lambda, 1}} \sum_{j \in \mathcal{S}_{1,\mu}} \bigl(\mu \rmapleq{1}{j}{\rep X} \otimes \rmapgtr{1}{j}{\rep{X}} + \rmapleq{1}{j}{\rep X} \otimes \opgtr{\ten L}{1} \rmapgtr{1}{j}{\rep{X}} \bigr)
\]
and $\ten L\ten x = \lambda \ten x$ we obtain
\[
  \sum_{\mu \in \cK_{\lambda, 1}} \sum_{j \in \mathcal{S}_{1,\mu}} \rmapleq{1}{j}{\rep X}  \otimes \opgtr{\ten L}{1} \rmapgtr{1}{j}{\rep{X}}
   =  \sum_{\mu \in \cK_{\lambda, 1}}  (\lambda - \mu) \sum_{j \in \mathcal{S}_{1,\mu}} \rmapleq{1}{j}{\rep X}  \otimes  \rmapgtr{1}{j}{\rep{X}}.
\]
Since $\rmapleq{1}{j}{\rep X} $, $j = 1,\ldots, r_1$, are linearly independent by our assumption of minimal ranks, the second statement in \eqref{evpi2} for $k=1$ follows.

Suppose we have sets $\mathcal{S}_{k,\mu}$ with $\mu\in \cK_{\lambda,k}$ such that \eqref{evpi2} and \eqref{blocki2} hold for some $k$ with $1 \leq k < K-1$, where $0 \leq \mu \leq \lambda$ by construction. Then
\begin{equation}\label{eq:eigenK}
\opgtr{\ten{L}}{k} \rmapgtr{k}{j}{\rep{X}}  = (\lambda-\mu) \rmapgtr{k}{j}{\rep{X}}  \quad\text{for all $\mu\in \cK_{\lambda,k}$ and all $j \in \mathcal{S}_{k,\mu}$.}
\end{equation}
For $\beta = 0,\ldots, n_{k+1}-1$, let
\[
\ten{y}^\beta_{k,j} = \bigl(  \rmapgtr{k}{j}{\rep{X}}_{\beta, \hat \alpha} \bigr)_{\hat \alpha \in \mathcal{N}_{k+2}^{K}}, \quad j = 1,\ldots, r_{k}.
\]
For each $\mu\in \cK_{\lambda,k}$, we then have
\[
\opgtr{\ten{L}}{k+1} \ten{y}^\beta_{k,j} = (\lambda-\mu - \lambda_{k+1,\beta})  \ten{y}^\beta_{k,j},        
\qquad 
j \in \cS_{k,\mu}.
\]
By the orthogonality of eigenvectors corresponding to distinct eigenvalues, this implies
\begin{equation}\label{eq:orthcond2}
\langle  \ten{y}^\beta_{k,j}, \ten{y}^{\beta'}_{k,j'} \rangle = 0 \quad \text{for $ j \in \mathcal{S}_{k,\mu}, \, j' \in \mathcal{S}_{k,\mu'}$ with ${\mu + \lambda_{k+1,\beta}\neq \mu' + \lambda_{k+1,\beta'}}
$.}
\end{equation}
We write $(X_{k+1}^{\{\beta\}})_j$ for the $j$-th row of $X_{k+1}^{\{\beta\}}$, $1\leq k < K-1$. For $\mu \in \mathcal{K}_{\lambda,k+1}$, we define
\[
\mathcal{Z}_{k+1,\mu}:= \operatorname{span} \, \bigcup_{\beta=0}^{n_{k+1}-1} \Bigl\{ (X_{k+1}^{\{\beta\}})_j \colon j \in \mathcal{S}_{k,\mu-\lambda_{k{+1},\beta}}   \Bigr\}  ,
\]
where for each $k$, we set $\mathcal{S}_{k,\mu} = \emptyset$ for $\mu \notin \mathcal{K}_{\lambda,k}$.
Since $k < K-1$, we have $ \ten{y}^\beta_{k,j} =  (X_{k+1}^{\{\beta\}})_j X_{k+2} \SKP \cdots \SKP X_K $. Under the conditions in \eqref{eq:orthcond2}, we thus have
\begin{equation*}
\langle G_{k+1} (X^{\{\beta\}}_{k+1})_j , (X^{\{\beta'\}}_{k+1})_{j'} \rangle  = 0, \quad G_{k+1} = \Bigl(  \langle \rmapgtr{k+1}{\ell}{\rep{X}} ,  \rmapgtr{k+1}{\ell'}{\rep{X}} \rangle   \Bigr)_{\ell,\ell' = 1, \ldots, r_{k+1}} ,
\end{equation*}
where again, $G_{k+1}$ is invertible since the ranks of $\rep X$ are minimal.
Consequently, for $z\in \mathcal{Z}_{k+1,\mu}$, $z' \in \mathcal{Z}_{k+1,\mu'}$ and $\mu \neq \mu'$, \eqref{eq:orthcond2} means that $\langle G_{k+1} z, z'\rangle = 0$. Again, there exists an orthogonal transformation $Q_{k+1}\in \mathrm{O}(r_{k+1})$ such that by replacing $X_{k+1}$ by $X_{k+1}G_{k+1}^{1/2}Q_{k+1}$ and $X_{k+2}$ by $Q_{k+1}^TG_{k+1}^{-1/2}X_{k+2}$ as before, we can ensure pairwise orthogonality of the spaces $\mathcal{Z}_{k+1,\mu}$, $\mu \in \mathcal{K}_{\lambda,k+1}$, in the Euclidean inner product. Additionally, for $\mu \in \cK_{\lambda,k+1}$, we can define subsets $\cS_{k+1,\mu}$ that form a partition of {$\{1, \ldots, r_{k+1}\}$} with $\# \mathcal{S}_{k+1,\mu} = \dim \mathcal{Z}_{k+1,\mu}$, such that for all $z \in \mathcal{Z}_{k+1,\mu}$ we have $\supp(z) \subseteq  \mathcal{S}_{k+1,\mu}$. For $k < K-1$, this implies the block structure \eqref{blocki2} for $X_{k+1}$, and $\opleq{\ten{L}}{k+1}  \rmapleq{k+1}{j}{\rep{X}} =   \mu \rmapleq{k+1}{j}{\rep{X}}$ for $j \in \cS_{k+1,\mu}$, which is the first statement in \eqref{evpi2} for $k+1$, holds by construction. Thus we also have
\[
    \sum_{\mu \in \cK_{\lambda, k+1}} \sum_{j \in \cS_{k+1,\mu}} \rmapleq{k+1}{j}{\rep{X}} \otimes \opgtr{\ten{L}}{k+1} \rmapgtr{k+1}{j}{\rep{X}} =     \sum_{\mu \in \cK_{\lambda, k+1}}  (\lambda-\mu)\sum_{j \in \cS_{k+1,\mu}}  \rmapleq{k+1}{j}{\rep{X}} \otimes \rmapgtr{k+1}{j}{\rep{X}},
\]
which by minimality of ranks yields the second statement in \eqref{evpi2} for $k+1$.
By induction, the statement thus follows for all $k \leq K-1$. 

Finally, for $k=K-1$, if $\mu \in \cK_{\lambda,K-1}$, then $\lambda - \mu = \lambda_{K,\beta}$ for some $\beta \in \{ 0,\ldots,n_K-1\}$, and \eqref{eq:eigenK} becomes
\begin{equation}\label{eq:eigenLast}
  L_K X_K^{[j, 1]} = (\lambda - \mu) X_K^{[j, 1]} \quad\text{for $\mu\in \cK_{\lambda,K-1}$, $j \in \mathcal{S}_{K-1,\mu}$,}
\end{equation}
noting that $r_K = 1$. As $L_K$ is a diagonal matrix, the eigenspaces in \eqref{eq:eigenLast} are given by 
\begin{equation}\label{eq:eigenspacesLast}
   \operatorname{span}\{ e^{\beta'} \colon \text{$\beta' = 0,\ldots,n_K-1$ with $\mu  = \lambda -  \lambda_{K,\beta'}$} \}, \quad \mu \in \cK_{\lambda, K-1}.
 \end{equation}
For $\beta \in \{ 0,\ldots,n_K-1 \}$ with $\mu = \lambda - \lambda_{K,\beta}$ let $j' \in \{ 1, \ldots, r_{K-1} \}$. If $j' \notin \cS_{K-1,\mu}$, then $j' \in \cS_{K-1,\mu'}$ for some $\mu' \neq \lambda- \lambda_{K,\beta}$, and $X_K^{[j',1]}$ is orthogonal to $X_K^{[j,1]}$ for all $j \in \cS_{K-1,\mu}$, which by \eqref{eq:eigenspacesLast} implies $X_K(j,\beta,1) = 0$. Thus also $X_K$ satisfies \eqref{blocki2}.

Conversely, suppose now that $\ten{x} \in \mathbb{R}^{n_1\times\cdots\times n_K}$, $\ten{x}\neq 0$, has the block structure \eqref{evpi2}, \eqref{blocki2}. This means that expanding the representation $\rep{X}$ in terms of elementary tensors of order $K$, each of the resulting summands is an eigenfunction of $\ten{L}$ with eigenvalue $\lambda$, and thus the same holds for $\ten{x}$.	
\end{proof}

	\begin{remark}
		{The above proof uses the explicit tensor network structure of an MPS. Analogous block sparsity results can be derived for other tree tensor network representations of tensors that are eigenvectors of an operator with the structure \eqref{eq:Ldef}. 
		Similar use of block sparsity for enforcing physical symmetries has also been considered in \cite{bauer2011implementing,SPV11} for tensor networks without tree structure such as PEPS \cite{verstraete_renormalization_2004} and MERA 
		\cite{VidalMERA}. In such cases, however, whether any tensor with a given symmetry necessarily has a representation with a particular block-sparsity pattern is not clear in this case. In other words, in the absence of tree structure we cannot establish an equivalence between membership in certain eigenspaces and representability with block structure as in Theorem \ref{thm:generalblockstructure}. Block sparsity can then still be used for obtaining approximations, see \cite[Sec.\ III.C]{bauer2011implementing}.}
	\end{remark}

We state the specific result for the particle number operator $\ten{P}$ as a corollary. This corresponds to the special case $n_k = 2$ and $\lambda = N$. For $k=1,\ldots,K$, we have $\lambda_{k,0} = 0$ and $\lambda_{k,1} = 1$, and hence $\lambda_{k,\alpha}^\sst{\leq }, \lambda_{k,\alpha}^\sst{>}\in \{ 0, \ldots, N\}$  in Theorem~\ref{thm:generalblockstructure}.
\begin{cor}\label{cor:blockstructure}
	Let $\ten{x} \in \cF^K$, $\ten{x}\neq 0$, have the representation $\ten{x} = \rmap{\rep{X}}$ with minimal ranks $\rk{r} = (r_1,\ldots, r_{K-1})$.
	Then for $N=1,\ldots, K$, one has $\ten{P} \ten{x} = N \ten{x}$ precisely when $\rep{X}$ can be chosen such that the following holds:
	for $k=1,\ldots, K$ and for all  
	\begin{equation}\label{eq:Kkdef} n \in \cK_k := \bigl\{ \max\{ 0 , N - K + k \} ,\ldots,  \min\{N, k\} \bigr\} \end{equation}
	there exist $\cS_{k,n} \subseteq \{ 1,\ldots, r_k\}$ 
	such that
	\begin{equation*}%
	\opleq{\ten{P}}{k}  \rmapleq{k}{j}{\rep{X}}   =   n   \rmapleq{k}{j}{\rep{X}}, \quad    \opgtr{\ten{P}}{k} \rmapgtr{k}{j}{\rep{X}} = (N-n) \rmapgtr{k}{j}{\rep{X}} ,  \quad j \in \cS_{k,n}     ,
	\end{equation*}
	and the matrices $X^{\{\beta\}}_{k}$, $\beta = 0,1$, have nonzero entries only in the blocks 
	\begin{equation*}%
	\begin{aligned} 
	& X^{\{0\}}_{k}\big|_{\cS_{k-1,n} \times \cS_{k,n}} & & \text{for $n \in \cK_{k-1}\cap \cK_{k}$},  \\ 
	& X^{\{1\}}_{k}\big|_{\cS_{k-1,n} \times \cS_{k,n+1}} & &\text{for $n \in \cK_{k-1} \cap (\cK_{k}-1)$} ,
	\end{aligned}
	\end{equation*}
	where we set $\cS_{0,0} = \cS_{K,N} = \{1\}$.
\end{cor}

{The block structure described by Corollary \ref{cor:blockstructure} is equivalent to the one used in physics literature (see, e.g., \cite{McCulloch:07,Schollwoeck:11,SPV11}), where it is usually stated differently in terms of quantum numbers and derived via $U(1)$ symmetry of operators.
For other symmetries considered in physics, such as $SU(2)$, the linear algebraic structure is different and therefore cannot be described by Laplace-like operators \eqref{eq:Ldef} as done above.}

In addition to the the fermionic particle number operator $\ten{P}$, there is a variety of other settings where the structure described by Theorem~\ref{thm:generalblockstructure} can be used.

\begin{example}
  In quantum chemistry, not only the particle number is conserved, but also the numbers of spin-up and spin-down particles. So the MPS is an eigenvector of two  associated Laplace-like operators $\ten{P}_\mathrm{up}$ and $\ten{P}_\mathrm{down}$. For both cases we have $n_k = 2$ and $\lambda_{k,1} = 1$ if $k$ even/odd for the up/down operator and  $\lambda_{k,1} = 0$ otherwise.
  {We then have  partial eigenvalues $\mathcal K_k^{\mathrm{up}}$ and $K_k^{\mathrm{down}}$  and the blocks depend on two partial eigenvalues, i.e., we have sets $\mathcal S_{k,n_1,n_2}\subseteq \{1,\ldots,r_k\}$ for $n_1 \in \mathcal K_k^{\mathrm{up}}$ and $n_2 \in K_k^{\mathrm{down}}$.  The blocks can then be ordered such that the blocks from the first operator have block structure  themselves.}
 
  Alternatively, one can introduce {\em spatial orbitals} that can carry one spin-up and one spin-down electron \cite{Szalay2015}. In this case, all dimensions are $n_k = 4$ and a similar block structure can be derived that also takes into account the antisymmetry of the particles.
\end{example}	
	
\begin{example}
Another example is the bosonic particle number operator with $n_k = n$ and $\lambda_{k,\alpha_k} = \alpha_k$. Tensor trains have frequently been applied in the parametrization of elements of high-dimensional polynomial spaces such as
\[
V_n^K = \left \{\sum_{\alpha\in \{0, \ldots, n-1\}^K} c_\alpha \prod_{k =1}^{K}x_k^{\alpha_k}\right\}\,,
\]	
see, e.g., \cite{oster_approximating_2020,dolgov_tensor_nodate,BCD,EPS}.
In this context, the bosonic particle number operator can be seen as a \emph{polynomial degree operator}. That is, if a polynomial is a linear combination of homogeneous polynomials with the same degree, its coefficient vector is an eigenvector of the polynomial degree operator with the eigenvalue equal to the degree. In $V_n^K$ the degree is precisely the cardinality of the multi-index $\alpha$.
	 
Another interesting example is the case $n_k = n$ and $\lambda_{k,\alpha_k} = 1$, $\alpha_k > 0$, which in the context of polynomials with $V_n^K$ as above measures the number of variables in a polynomial. This means eigenvectors of this operator are associated with coefficient vectors where the multi-index $\alpha$ is nonzero only for a fixed number (the associated eigenvalue) of variables.
\end{example}	

We define the block sizes $\rho_{k,\mu} := \# \cS_{k,\mu} $, where $\sum_{\mu \in \cK_{\lambda,k}} \rho_{k,\mu} = r_k$, and derive the following upper bounds.
\begin{lemma}\label{lmm:blockranks}
	Let $\ten{x} \in \mathbb{R}^{n_1\times\cdots\times n_K}$, $\ten{x}\neq 0$, have the representation $\ten{x} = \rmap{\rep{X}}$ with minimal ranks $\rk{r} = (r_1,\ldots, r_{K-1})$ and $\ten{L} \ten{x} = \lambda \ten{x}$. Furthermore, let $E_{\mu,k}^\sst{\leq}$ and $E_{\lambda-\mu,k}^\sst{>}$ be the eigenspaces of $\opleq{\ten{L}}{k}$ and $\opgtr{\ten{L}}{k}$  of eigenvalues $\mu$ and $\lambda-\mu$, respectively. Then for $k=1,\ldots, K$ and $\mu \in \cK_{\lambda,k}$, we have
	\[
	\rho_{k,\mu} \leq \min\bigl\{ \dim E_{\mu,k}^\sst{\leq},\dim E_{\lambda-\mu,k}^\sst{>} \bigr\} \,.
	\]
\end{lemma}
\begin{proof}
	With Theorem~\ref{thm:generalblockstructure}, for $j \in \cS_{k,\mu}$, we obtain $\opleq{\ten{L}}{k}  \rmapleq{k}{j}{\rep{X}}   =   \mu   \rmapleq{k}{j}{\rep{X}}$ and $\opgtr{\ten{L}}{k} \rmapgtr{k}{j}{\rep{X}} = (\lambda-\mu) \rmapgtr{k}{j}{\rep{X}}$.
	The partial tensors $\rmapleq{k}{j}{\rep{X}}$ are linearly independent because if they were not, we could reduce the rank $r_k$, which is assumed to be minimal. Therefore $\rho_{k,\mu}$ has to be smaller or equal to the dimension of the eigenspace of the operator $\opleq{\ten{L}}{k}$ to the eigenvalue $\mu$. %
	Analogously, we look at the eigenspace of the operator $\opgtr{\ten{L}}{k}$ to the eigenvalue $\lambda-\mu$.%
\end{proof}
Note that for the particle number operator $\ten{P}$, where $\rho_{k,n} = \# \cS_{k,n} $ for $n \in \cK_k$,
\[
\dim E_{n,k}^\sst{\leq} = {k \choose n}, \quad n \in\{0\ldots,N\}.
\]

As a final result in this chapter, we show that Laplace-like operators commute with the tangent space projection and thus, they share the same eigenvectors.
  \begin{cor}\label{cor:tangprok}
 	For $\ten{x}\in\mathbb{R}^{n_1\times\cdots\times n_K}$ and $\ten{L}\ten{x} = \lambda \ten{x}$,
 	we have $\ten{Q}_{\ten{x}}\ten{L} = \ten{L}\ten{Q}_{\ten{x}}$.
 \end{cor}
 \begin{proof}
 	We apply Theorem~\ref{thm:generalblockstructure}. For two arbitrary but fixed multi-indices $\alpha,\beta \in \cN$ let $\rep{e}^\alpha = (e^{\alpha_1},\ldots,e^{\alpha_K})$ and $\rep{e}^\beta = (e^{\beta_1},\ldots,e^{\beta_K})$ be two representations such that $\rmap{\rep{e}^\alpha}$ and $\rmap{\rep{e}^\beta}$ are unit vectors in $\mathbb{R}^{n_1\times\cdots\times n_K}$. Then it suffices to show that 
 	\[
 	 \langle \ten{Q}_{\ten{x}}\ten{L}\rmap{\rep{e}^\beta}, \rmap{\rep{e}^\alpha} \rangle  = \langle \ten{L}\ten{Q}_{\ten{x}}\rmap{\rep{e}^\beta}, \rmap{\rep{e}^\alpha} \rangle ,
 	\]
 	and since $\rmap{\rep{e}^\alpha}$ and $\rmap{\rep{e}^\beta}$ are eigenvectors of $\ten{L}$ with eigenvalue $\lambda_\alpha$ and $\lambda_\beta$, this simplifies to showing that for $\lambda_\alpha \neq \lambda_\beta$ we have
$
 	 \langle \ten{Q}_{\ten{x}}\rmap{\rep{e}^\beta}, \rmap{\rep{e}^\alpha} \rangle = 0
$.
 	Now for $k=1,\ldots,K$, we show that 
 	\[
 	 \langle \ten Q_{\ten x}^{k,i} 
 	\rmap{\rep{e}^\beta},  \rmap{\rep{e}^\alpha} \rangle = 0, \qquad i = 1,2.
 	\] 
 	We give the proof for $i=1$, the case $i=2$ can be treated analogously.
 	Note that if $\lambda^\alpha \neq \lambda^\beta$, we can assume without loss of generality that $\rmapless{k}{1}{\rep{e}^\alpha}$ and $\rmapless{k}{1}{\rep{e}^\beta}$ are eigenvectors of different eigenvalues $\lambda_{k-1,\alpha}^\sst{\leq}\neq \lambda_{k-1,\beta}^\sst{\leq}$ of $\opleq{\ten{L}}{k-1}$.
 	But since $\ten{x}$ is also an eigenvalue  of $\ten{L}$ it has the properties shown in Theorem~\ref{thm:generalblockstructure}, and consequently,
 	\[
 	\langle \rmapless{k}{j_1}{\rep{U}}, \rmapless{k}{1}{\rep{e}^\alpha} \rangle = 0 \text{ for }j_1\notin\cS_{k-1,\lambda_{k-1,\alpha}^{\leq}} \text{ and } \langle \rmapless{k}{j_2}{\rep{U}}, \rmapless{k}{1}{\rep{e}^\beta}  \rangle = 0 \text{ for }j_2\notin\cS_{k-1,\lambda_{k-1,\beta}^{\leq}}.
 	\]
 	As $\cS_{k-1,\lambda_{k-1,\alpha}^\sst{\leq}}\cap\cS_{k-1,\lambda_{k-1,\beta}^\sst{\leq}} = \emptyset$, this implies
 	\[
 	\langle \rmapless{k}{j}{\rep{U}}, \rmapless{k}{1}{\rep{e}^\alpha}\rangle \langle \rmapless{k}{j}{\rep{U}},\rmapless{k}{1}{\rep{e}^\beta}\rangle  = 0\quad \text{for all } j\in\{1,\ldots, r_{k-1}\},
 	\]
 	concluding the proof. 
 \end{proof}

\section{Basic Operations on Block-Structured Matrix Products States}\label{sec:blockops}

In the remainder of this article, we restrict ourselves again to the case $\ten L = \ten P$ and $\mathcal N = \{0,1\}^K$.
For fixed $N \leq K$, we denote the space of all tensors $\ten x \in \cF^K$ with $\ten P \ten x = N \ten x$ as 
\[ \cF^K_N := \{ \ten x \in \cF^K : \ten P \ten x = N \ten x \}\]
and we represent them in the block-sparse MPS format. The block structure of the matrix product states leads to more efficient storage and computation, if exploited correctly. Furthermore, a restriction to one of the eigenspaces of the particle number operator eliminates redundancies in iterative minimization schemes. 

We simplify notation by first noting that the sets $\cS_{k,n}$ are disjoint and that they can be ordered arbitrarily due to the invariance of the components. This means that the matrices $X^{\{0\}}_k$ and $X^{\{1\}}_k$ are either block-diagonal or they have blocks only just above or just below the diagonal. We denote the blocks representing an unoccupied $k$-th orbital by 
\begin{equation*}
\unocc{X}_{k,n} = X^{\{0\}}_{k}\big|_{\cS_{k-1,n} \times \cS_{k,n}} \in \R^{\rho_{k-1,n} \times \rho_{k,n}} \qquad \text{for $n \in \cK_{k-1}\cap \cK_{k}$},
\end{equation*}
and those representing an occupied orbital by 
\begin{equation*}
\occ{X}_{k,n} = X^{\{1\}}_{k}\big|_{\cS_{k-1,n} \times \cS_{k,n+1}} \in \R^{\rho_{k-1,n} \times \rho_{k,n+1}} \qquad \text{for $n \in \cK_{k-1} \cap (\cK_{k}-1)$}.
\end{equation*}
For $k$ such that $N < k < K-N+1$, which we refer to as the \emph{generic case}, we have $\cK_{k-1} = \cK_k = \{0,\ldots,N\}$; otherwise, the number of particles to the right and to the left of orbital $k$, and hence the elements of $\cK_{k-1}$ and $\cK_k$, are restricted according to \eqref{eq:Kkdef}. The corresponding block structure according to Corollary \ref{cor:blockstructure} has the form
\begin{equation*}
\setlength\arraycolsep{3pt}
X^{\{0\}}_k = \begin{pmatrix}
\unocc{X}_{k,0} & 0 & \cdots & 0 \\
0 & \unocc{X}_{k,1} & \cdots & 0 \\
\vdots & \vdots & \ddots & \vdots \\
0 & 0 & \cdots & \unocc{X}_{k,N}
\end{pmatrix}	
\qquad \text{and} \qquad
X^{\{1\}}_k = \begin{pmatrix}
0 & \occ{X}_{k,0} & \cdots & 0 \\
\vdots & \vdots & \ddots & \vdots \\
0 & 0 & \cdots & \occ{X}_{k,N-1} \\
0 & 0 & \cdots & 0 
\end{pmatrix}.	
\end{equation*}
Since nonzero blocks for the unoccupied orbital never occur in the same position as the ones for the occupied orbital, the two layers $\alpha = 0$ and $\alpha = 1$ can be summarized in the core representation
\begin{equation*}
\setlength\arraycolsep{3pt}
X_k = \begin{bmatrix}
\unocc{X}_{k,0}^{\uparrow} & \occ{X}_{k,0}^{\uparrow} & \cdots & 0 \\[-3pt] 
0 & \unocc{X}_{k,1}^{\uparrow} & \ddots & \vdots \\[-2pt]
\vdots & \vdots & \ddots & \hspace{4pt} \occ{X}_{k,N-1}^{\uparrow} \\[3pt]
0 & 0 & \cdots & \hspace{-7pt}\unocc{X}_{k,N}^{\uparrow}
\end{bmatrix} =: \bidiag\left(\begin{bmatrix} \unocc{X}_{k,n}^{\uparrow} & \occ{X}_{k,n}^{\uparrow} \end{bmatrix}_{n \in \cK_{k-1}} \right),
\end{equation*}
where each block is composed of vectors, and where we define
\begin{equation*}
\unocc{X}_{k,n}^{\uparrow} = \unocc{X}_{k,n} \uparrow e^0 \in \R^{\rho_{k-1,n} \times 2 \times \rho_{k,n}} \qquad \text{for $n \in \cK_{k-1}\cap \cK_{k}$}
\end{equation*}
and
\begin{equation*}
\occ{X}_{k,n}^{\uparrow} = \occ{X}_{k,n} \uparrow e^1 \in \R^{\rho_{k-1,n} \times 2 \times \rho_{k,n+1}} \qquad \text{for $n \in \cK_{k-1} \cap (\cK_{k}-1)$}.
\end{equation*}
The cases where either $k \leq N$ or $k \geq K-N+1$ have the last rows or first columns (and zero rows and columns) removed, respectively, as illustrated in the following example.
\begin{example}
A tensor $\ten x \in \cF^5_2$ of order $K=5$ and particle number $N=2$, representing $5$ orbitals and $2$ particles, has the form\small
\begin{equation*}
\setlength\arraycolsep{3pt}
\ten x = \begin{bmatrix}
\unocc{X}_{1,0}^{\uparrow} & \occ{X}_{1,0}^{\uparrow}
\end{bmatrix}
\SKP
\begin{bmatrix}
\unocc{X}_{2,0}^{\uparrow} & \occ{X}_{2,0}^{\uparrow} & 0 \\[3pt]
0 & \unocc{X}_{2,1}^{\uparrow} & \occ{X}_{2,1}^{\uparrow}
\end{bmatrix}
\SKP
\begin{bmatrix}
\unocc{X}_{3,0}^{\uparrow} & \occ{X}_{3,0}^{\uparrow} & 0 \\[3pt]
0 & \unocc{X}_{3,1}^{\uparrow} & \occ{X}_{3,1}^{\uparrow} \\[3pt]
0 & 0 & \unocc{X}_{3,2}^{\uparrow} \\
\end{bmatrix}
\SKP
\begin{bmatrix}
\occ{X}_{4,0}^{\uparrow} & 0 \\[3pt]
\unocc{X}_{4,1}^{\uparrow} & \occ{X}_{4,1}^{\uparrow} \\[3pt]
0 & \unocc{X}_{4,2}^{\uparrow}
\end{bmatrix}
\SKP
\begin{bmatrix}
\occ{X}_{5,1}^{\uparrow} \\[3pt]
\unocc{X}_{5,2}^{\uparrow} 
\end{bmatrix}.
\end{equation*}\normalsize
\end{example}

As an eigenspace of a linear operator, $\cF^K_N$ is a linear subspace of $\cF^K$. Addition and scalar multiplication of MPS in this subspace correspondingly  work equivalently to those of regular MPS: Addition of two tensors in block-sparse MPS format is the concatenation of corresponding blocks, scalar multiplication is the multiplication of all blocks in one of its components. We will now explicitly describe some further more involved operations, as well as left- and right-orthogonalization and rank truncation procedures.

Many operations on MPS can be performed either left-to-right or right-to-left, and in general, both versions are required. Here we state right-to-left procedures, as their description is notationally more compact. Apart from the notation, however, all left-to-right procedures are performed analogously.

\begin{algorithm}[t!]
\label{alg:innerprod}
\SetKwInOut{Input}{input}\SetKwInOut{Output}{output}
\Input{Tensors $\ten x= \rmap{\rep X}, \,\ten y = \rmap{\rep Y} \in \cF^K_N$.}
\Output{$\langle \ten x , \ten y \rangle \in \R$.}
\BlankLine
Set $R_N = \unocc{X}_{K,N} \unocc{Y}_{K,N}^{\mathsf T}$ and $R_{N-1} = \occ{X}_{K,N-1} \occ{Y}_{K,N-1}^{\mathsf T}$

\For{$k = K-1,K-2,\ldots,1$}{
\For{$n \in \cK_{k-1}$}{
\If{$n \in \cK_k \cap (\cK_k - 1)$}{
Set $L_n = \unocc{X}_{k,n} R_n \unocc{Y}_{k,n}^{\mathsf T} + \occ{X}_{k,n} R_{n+1} \occ{Y}_{k,n}^{\mathsf T}$
}

\If{$n \in \cK_k \setminus (\cK_k - 1)$}{
Set $L_n = \unocc{X}_{k,n} R_n \unocc{Y}_{k,n}^{\mathsf T}$
}

\If{$n \in (\cK_k - 1) \setminus \cK_k$}{
Set $L_n = \occ{X}_{k,n} R_{n+1} \occ{Y}_{k,n}^{\mathsf T}$
}}
\For{$n \in \cK_{k-1}$}{
Set $R_n \leftarrow L_n$
}}
\Return{$R_0$}
\caption{Inner product in block-sparse MPS format.}
\end{algorithm}

Alg.~\ref{alg:innerprod} describes the scheme for computing the inner product of two elements of $\cF^K_N$ in block-sparse MPS format, each of which may be of arbitrary ranks. The scheme consists of a right-to-left procedure that successively contracts the blocks of corresponding size. {Note that in each step, one ultimately constructs the partial representation mapping $\rmapgtr{k}{j}{\rep{X}}$ where $j \in \cS_{k,n}, n \in \cK_k$. This, and its left-to-right counterpart $\rmapless{k}{j}{\rep{X}}$, is then given in a block-diagonal form, which can be exploited in the construction of subproblems in the DMRG algorithm discussed in Sec.~\ref{sec:itermethods}.}

\begin{algorithm}[t!]
\label{alg:QR}
\SetKwInOut{Input}{input}\SetKwInOut{Output}{output}
\Input{Tensor $\ten x = \rmap{\rep X} \in \cF^K_N$.}
\Output{Right-orthogonal representation $\rep X$.}
\BlankLine
\For{$k = K,K-1,\ldots,2$}{
\For{$n \in \cK_{k-1}$}{
\If{$n \in \cK_k \cap (\cK_k - 1)$}{
Compute QR-factorization $\begin{pmatrix}
\unocc{X}_{k,n}^\mathsf{T} \\ \occ{X}_{k,n}^\mathsf{T}
\end{pmatrix} = Q_n R_n$ with $Q_n = \begin{pmatrix}
Q_{0,n} \\ Q_{1,n}
\end{pmatrix}$

Set $\unocc{X}_{k,n} \leftarrow Q_{0,n}^\mathsf{T}$ and $\occ{X}_{k,n} \leftarrow Q_{1,n}^\mathsf{T}$

Set $\unocc{X}_{k-1,n} \leftarrow \unocc{X}_{k-1,n} R_n^\mathsf{T}$
 
\If{$n-1 \in \cK_{k-2}$}{
Set $\occ{X}_{k-1,n-1} \leftarrow \occ{X}_{k-1,n-1} R_n^\mathsf{T}$
}
}
\If{$n \in \cK_k \setminus (\cK_k - 1)$}{
Compute QR-factorization $\unocc{X}_{k,n}^\mathsf{T} = Q_n R_n$

Set $\unocc{X}_{k,n} \leftarrow Q_n^\mathsf{T}$

Set $\unocc{X}_{k-1,n} \leftarrow \unocc{X}_{k-1,n} R_n^\mathsf{T}$

\If{$n-1 \in \cK_{k-2}$}{
Set $\occ{X}_{k-1,n-1} \leftarrow \occ{X}_{k-1,n-1} R_n^\mathsf{T}$
}
}

\If{$n \in (\cK_k - 1) \setminus \cK_k$}{
Compute QR-factorization $\occ{X}_{k,n}^\mathsf{T} = Q_n R_n$

Set $\occ{X}_{k,n} \gets Q_n^\mathsf{T}$

Set $\unocc{X}_{k-1,n} \leftarrow \unocc{X}_{k-1,n} R_n^\mathsf{T}$

\If{$n-1 \in \cK_{k-2}$}{
Set $\occ{X}_{k-1,n-1} \leftarrow \occ{X}_{k-1,n-1} R_n^\mathsf{T}$
}
}
}
Set $X_k \gets \bidiag\left(\begin{bmatrix} \unocc{X}_{k,n}^{\uparrow} & \occ{X}_{k,n}^{\uparrow} \end{bmatrix}_{n \in \cK_k}\right)$
}
Set $X_1 \gets \bidiag\left(\begin{bmatrix} \unocc{X}_{1,n}^{\uparrow} & \occ{X}_{1,n}^{\uparrow} \end{bmatrix}_{n \in \cK_1}\right)$

\Return{$\rep X = (X_1, \ldots, X_K)$}
\caption{Right-orthogonalization in block-sparse MPS format.}
\end{algorithm}

In Alg.~\ref{alg:QR}, we demonstrate the procedure for orthogonalizing from right to left, resulting in a right-orthogonal tensor. 
The method for bringing the tensor into its right-orthogonal TT-SVD representation, as described in Alg.~\ref{alg:SVD}, follows a similar pattern. Here, the input tensor needs to be given in left-orthogonal format (to which one transforms analogously to Alg.~\ref{alg:QR}), singular value decompositions of joined blocks are computed from right to left, and the singular values in each step are stored. These singular values are used to truncate the ranks of the tensor based on the the estimates in \cite{Oseledets:2011:TT,Grasedyck:2010:HierarchicalSVD}. Alg.~\ref{alg:svdthresh} summarizes this procedure, where the smallest singular values are selected such that they do not exceed the upper bound on the truncation error $\varepsilon$. The rows and columns of the corresponding blocks are then deleted. A similar procedure can be used for truncation to given ranks. Note that if the error threshold $\varepsilon$ is chosen too large, it is possible that the whole tensor is truncated to zero. This means that zero is actually the best low rank approximation to the given tensor. In the present context, we want to avoid this anomaly, since the zero tensor is physically meaningless (we emphasize that it does not represent the vacuum state). As long as $\varepsilon < \norm{\ten x}$ in Alg.~\ref{alg:svdthresh}, this cannot occur.

\begin{algorithm}[t!]
\label{alg:SVD}
\SetKwInOut{Input}{input}\SetKwInOut{Output}{output}
\Input{Left-orthogonal representation block-sparse $\ten x= \rmap{\rep X} \in \cF^K_N$.}
\Output{$\ten x = \rmap{\rep X}$ with block-sparse $\rep X$ in right-orthogonal TT-SVD form.}
\BlankLine
\For{$k = K,K-1,\ldots,2$}{
\For{$n \in \cK_{k-1}$}{
\If{$n \in \cK_k \cap (\cK_k - 1)$}{
Compute SVD $\begin{pmatrix}
\unocc{X}_{k,n} & \occ{X}_{k,n}
\end{pmatrix} = U_n \Sigma_{k-1,n} V_n^\mathsf{T}$ with $V_n = \begin{pmatrix}
V_{0,n} \\ V_{1,n}
\end{pmatrix}$

Set $\unocc{X}_{k,n} \leftarrow V_{0,n}^\mathsf{T}$ and $\occ{X}_{k,n} \leftarrow V_{1,n}^\mathsf{T}$

Set $\unocc{X}_{k-1,n} \leftarrow \unocc{X}_{k-1,n} U_n \Sigma_{k-1,n}$

\If{$n-1 \in \cK_{k-2}$}{
Set $\occ{X}_{k-1,n-1} \leftarrow \occ{X}_{k-1,n-1} U_n \Sigma_{k-1,n}$
}
}
\If{$n \in \cK_k \setminus (\cK_k - 1)$}{
Compute SVD $\unocc{X}_{k,n} = U_n \Sigma_{k-1,n} V_n^\mathsf{T}$

Set $\unocc{X}_{k,n} \leftarrow V_n^\mathsf{T}$

Set $\unocc{X}_{k-1,n} \leftarrow \unocc{X}_{k-1,n} U_n \Sigma_{k-1,n}$

\If{$n-1 \in \cK_{k-2}$}{
Set $\occ{X}_{k-1,n-1} \leftarrow \occ{X}_{k-1,n-1} U_n \Sigma_{k-1,n}$
}
}

\If{$n \in (\cK_k - 1) \setminus \cK_k$}{
Compute SVD $\occ{X}_{k,n}^\mathsf{T} = U_n \Sigma_{k-1,n} V_n^\mathsf{T}$

Set $\occ{X}_{k,n} \gets V_n^\mathsf{T}$

Set $\unocc{X}_{k-1,n} \leftarrow \unocc{X}_{k-1,n} U_n \Sigma_{k-1,n}$

\If{$n-1 \in \cK_{k-2}$}{
Set $\occ{X}_{k-1,n-1} \leftarrow \occ{X}_{k-1,n-1} U_n \Sigma_{k-1,n}$
}
}
}
Set $X_k \gets \bidiag\left(\begin{bmatrix} \unocc{X}_{k,n}^{\uparrow} & \occ{X}_{k,n}^{\uparrow} \end{bmatrix}_{n \in \cK_k}\right)$
}
Set $X_1 \gets \bidiag\left(\begin{bmatrix} \unocc{X}_{1,n}^{\uparrow} & \occ{X}_{1,n}^{\uparrow} \end{bmatrix}_{n \in \cK_1}\right)$

\Return{$\rep X = (X_1,\ldots, X_K)$ and $\left(\Sigma_{k,n}\right)_{\substack{k = 1,\ldots,K-1 \\ n \in \cK_k}}$}
\caption{Right-orthogonal TT-SVD in block-sparse MPS format.}
\end{algorithm}

 \begin{algorithm}[t!]
\label{alg:svdthresh}
\SetKwInOut{Input}{input}\SetKwInOut{Output}{output}
\SetKwInOut{Input}{input} 
\Input{SVD-normalized tensor $\ten x \in \cF^K_N$, singular value matrices $\left(\Sigma_{k,n}\right)_{\substack{k = 1,\ldots,K-1 \\ n \in \cK_k}}$, \vspace{-5pt} \\ error threshold $\varepsilon \in \R_+$. \vspace{5pt}}
\Output{Truncated $\rep Y$ with $\| \ten x - \rmap{{\rep Y}} \| \leq \varepsilon$.}
\BlankLine
Find {$(k_i,n_i)_{i\in I}$ for the} $I$ smallest singular values $\sigma_{k_i,n_i}$ in $\bigl(\diag(\Sigma_{k,n})\bigr)_{\substack{k = 1,\ldots,K-1 \\ n \in \cK_k}}$ {s.t.} \vspace{-5pt}
\[ \sum_{i = 1}^I \sigma_{k_i,n_i}^2 \leq \varepsilon^2 \]
\\[-5pt]
Set $\rep Y = \rep X$

\For{$i = 1,\ldots,I$}{
\If{$n_i \in \cK_{{k_i}-1}$}{
Delete last column of $\unocc{Y}_{k_i,n_i}$
}
\If{$n_i-1 \in \cK_{{k_i}-1}$}{
Delete last column of $\occ{Y}_{k_i,n_i-1}$
}
\If{$n_i \in \cK_{{k_i}+1}$}{
Delete last row of $\unocc{Y}_{k_i+1,n_i}$
}
\If{$n_i \in (\cK_{{k_i}+1} - 1)$}{
Delete last row of $\occ{Y}_{k_i+1,n_i}$
}
}
\For{$k = 1,\ldots,K$}{
Set $Y_k \gets \bidiag\left(\begin{bmatrix} \unocc{Y}_{k,n}^{\uparrow} & \occ{Y}_{k,n}^{\uparrow} \end{bmatrix}_{n \in \cK_k}\right)$
}
\Return{$\rep Y = (Y_1, \ldots, Y_K)$}
\caption{TT-SVD truncation in block-sparse MPS format.}
\end{algorithm}

\begin{remark}[Optimality]
If the standard TT-SVD is unique, for each $k$, it differs from the block-sparse TT-SVD representation produced by Algorithm \ref{alg:SVD} only by the ordering of singular values. In particular, when the entries of the diagonal matrices $\Sigma_{k,n}$, $n \in \cK_k$ are ordered by size, the optimality properties \eqref{eq:ttsvdquasiopt} of the TT-SVD truncation hold also in this setting.
\end{remark}

\begin{remark}[Particle number conservation]\label{rem:ttsvd}
The above remark implies that rank truncation of the TT-SVD is a particle number preserving operation. Non-uniqueness of the TT-SVD can occur only if a matricization has a multiple singular value; in such a case, an arbitrary choice of the TT-SVD in general is not block-sparse, and truncation of the TT-SVD may change the particle number. As illustrated in Section \ref{ssc:rounding}, in cases with singular values that are distinct but close, the associated numerical instability of singular vectors can lead to deviations in the particle number when the numerically computed TT-SVD is truncated, unless the block structure is enforced explicitly.
\end{remark}

\begin{remark}[Number of operations]
Depending on the relation between block sizes and total MPS ranks, the block-sparse representation can allow for a substantial reduction of computational costs. As an example, we consider the TT-SVD procedure as in Algorithm~\ref{alg:SVD} with $K\gg N$.
For convenience, let $\rho_{k,n} = \#{ \cS_{k,n} }$ if $n \in \cK_k$ and $\rho_{k,n} = 0$ otherwise. Then the number of operations of this algorithm are dominated by the SVDs for joined blocks for each $k$ and $n$, and thus in total of order \[\mathcal{O}\biggl( \sum_{k=2}^K \sum_{n\in \cK_k} \min\{ \rho_{k-1,n}, \rho_{k,n} + \rho_{k,n+1}\}^2 \max\{ \rho_{k-1,n}, \rho_{k,n} + \rho_{k,n+1}\} \biggr).\] 

If $\rho_{k,n} = \bar\rho$ for all $k,n$, the corresponding total ranks are $r_k = (N+1)\bar\rho$, with exception of the $2N$ lowest and highest values of $k$, and thus
the total operation costs scale as $\mathcal{O}(K N \bar\rho^3)$. In comparison, the TT-SVD of a full MPS representation costs \[ \mathcal{O}\biggl(\sum_{k=2}^K \min\{r_{k-1},r_k\}^2 \max\{r_{k-1},r_k\} \biggr),\] where $\sum_{k=2}^K \min\{r_{k-1},r_k\}^2 \max\{r_{k-1},r_k\} \approx K N^3\bar\rho^3$. In such a case, the exploitation of block sparsity thus leads to a reduction of the costs approximately by a factor $N^2$. However, if for each $k$, one has $r_k  = \rho_{k,n}$ for some $n$, corresponding to only a single block being nonzero for each $k$, there is no reduction of storage or operations costs compared to the full MPS representation.
\end{remark}

Finally, we want to show how the tangent space of fixed-rank MPS manifolds (see \cite{holtz_manifolds_2012}) can be handled algorithmically with the block-sparse MPS format.  To this end, we state the parametrization in block-sparse MPS format of an arbitrary element $\ten y$ of the tangent space at a given $\ten{x} \in\mathbb{R}^{n_1\times\cdots\times n_K}$ of ranks $r = (r_1,\cdots,r_{K-1})$. 
For such $\ten y$, there exist cores $\delta Y_k\in\mathbb{R}^{r_{k-1}\times n_k\times r_{k}}$ such that
 \begin{align}\label{eq:tangentambientspace}
 \ten{y} = \sum_{k=1}^K \tau(U_1,\cdots,U_{k-1},\delta Y_k,V_{k+1},\cdots,V_K),
 \end{align}
 where as in Sec.~\ref{ssc:tangentspace}, we assume that $\ten x = \rmap{\rep{U}} = \rmap{\rep V}$ with $\rep{U} = (U_1,\cdots,U_K)$ in left- and $\rep{V} = (V_1,\cdots,V_K)$ in right-orthogonal form. Note that the cores $\delta Y_k$ necessarily have the same block-sparse structure as $U_k$ and $V_k$.
 
The projection of $\ten z = \rmap{\rep Z} \in \cF^K_N$ to the tangent space at $\ten x$ can be obtained by first computing the components $\delta Y_k$ of the projection $\ten y$ and then assembling the MPS representation of $\ten y$ using \eqref{eq:tangentambientspace}. The computation of the $\delta Y_k$ can be performed in a similar fashion as the inner product in Algorithm~\ref{alg:innerprod}, once from left to right and once from right to left. This means that for $k=1,\ldots,K-1$, one recursively evaluates 
 \[
 \left(\langle \rmapleq{k}{j}{\rep{U}},\rmapleq{k}{j'}{\rep{Z}}\rangle\right)_{j,j'} \text{ and } \left(\langle\rmapgtr{k}{j}{\rep{V}},\rmapgtr{k}{j'}{\rep{Z}}\rangle\right)_{j,j'}.
 \]
 With these quantities at hand, one obtains $\delta Y_k$ as described in \cite{steinlechner_riemannian_2016}.
Then one has the representation
 \begin{align*}
 \ten{y} = \begin{bmatrix}
 U_1 & \delta Y_1
 \end{bmatrix}\SKP \begin{bmatrix}
 U_2 & \delta Y_2\\ 0 & V_2
 \end{bmatrix}\SKP\cdots\SKP
 \begin{bmatrix}
 \delta Y_K \\ V_K
 \end{bmatrix},
 \end{align*}
where the rank indices in each core can be reordered to yield the same block-sparse structure as in $\rep U$ and $\rep V$ with at most doubled rank parameters.

\section{Matrix Product Operators}\label{sec:mpo}

It is well known that the Hamiltonian~\eqref{eq:hamil} commutes with the particle number operator $\ten P$ \cite[\S1.3.2]{Helgaker:00}:
\begin{lemma}\label{lemma:hamilpncommute}
	We have that the Hamiltonian and the particle number operator commute, that is, $\ten H\ten P = \ten P \ten H$. Furthermore, all eigenvectors of $\ten H$ are eigenvectors of $\ten P$.
\end{lemma} 
Thus $\ten H$ preserves the particle number of a state as well as its block structure. In fact, we can show that every particle number-preserving operator can be written as a sum of \emph{rank-one particle number preserving operators} of the form $\ten{a}_{D^+}^*\ten{a}_{D^-}$ with subsets ${D^-,D^+\subseteq \{1,\dots,K\}}$ such that $\#{D^-}=\#{D^+}$, where $\ten{a}_D = \prod_{i\in D} \ten{a}_i$. Note that we define $\ten{a}_{\emptyset}$ to be the identity mapping. Furthermore, we associate  each $D \subseteq \{ 1,\ldots, K\}$ with a unit vector $\ten{e}_D = \ten{a}_D^*\vac$, where $\vac$ is the \emph{vacuum state}
\[
 \vac = \left(e^0\right)^{\otimes K} \,.
\]
We then have the following result, which is shown in Appendix \ref{app:oprankoneterms}.

\begin{lemma}\label{lmm:oprankoneterms}
	Let $\ten{B}: \cF^K \rightarrow  \cF^K$ be a particle number-preserving operator, that is, $\ten B$ maps each eigenspace of $\ten{P}$ to itself. Then there exist coefficients $v_{D^+,D^-}\in\mathbb{R}$ such that
\begin{equation}\label{eq:partop}
\ten{B} = \sum_{\substack{ D^+,D^- \subseteq \{1,\dots,K\} \\  \# D^+ = \# D^-} } v_{D^+,D^-} \ten{a}_{D^+}^*\ten{a}_{D^-},
\end{equation}
in other words, $\ten{B}$ can be written as a sum of rank-one particle number-preserving operators. 
\end{lemma}

Linear operators on matrix product states can be in the MPO format \eqref{eq:mpo} with cores of order four. We will now investigate the ranks in the MPO representation of Hamiltonians of the form~\eqref{eq:hamil}, that is, particle number-preserving operators with one- and two-particle terms. As shown below, the MPO ranks of such operators grow at most quadratically with the order $K$ of the tensor. 
Furthermore, since both of these operators preserve the particle number, their effect on a block-sparse MPS can be expressed only in terms of the blocks. At the end of this section, we will show that each of the summands in~\eqref{eq:partop} describes nothing but a shift and scalar multiplication of some of the blocks. This means that the application of the Hamiltonian to a block-sparse MPS can be expressed in a matrix-free way, leading to an elegant and efficient algorithmic treatment.

\subsection{Compact Forms of Operators}\label{ssc:compactform}

We now turn to the ranks of Hamiltonians as in \eqref{eq:hamil} in second quantization in MPO format.
  As shown in this section, compared to the number of rank-one terms in the representation \eqref{eq:hamil}, one can obtain substantially reduced ranks in MPO representations. The basic mechanism behind this rank reduction is described in \cite{chan_matrix_2016} for projected Hamiltonians in the context of DMRG solvers {and, in an MPO form for full Hamiltonians similar to the one given here, in \cite{crosswhite2008finite,keller2015efficient}.} 
  {For one-particle operators, MPO representations are also given in the mathematical literature \cite{dolgov2013two,kazeev2013low}.}
  Here we use similar considerations to construct an explicit MPO representation of the full Hamiltonian with near-minimal ranks {and with a unified treatment of one- and two-particle operators}. To avoid technicalities, in what follows we assume $K$ to be even, but this is not essential for the construction.

In preparation for the MPO representations of the one- and two-particle operators, for illustrative purposes, start with the case of Laplace-like operators 
\begin{equation}\label{eq:laplacelikeF}
\ten{F} = \sum_{i=1}^K\lambda_i\ten{a}_i^*\ten{a}_i.
\end{equation}
We denote this operator by $\ten{F}$ since the Fock operator is of this form when its eigenfunctions are used as orbitals. 
From \eqref{eq:laplacelikeF} and \eqref{def:anniloperator}, we immediately obtain a representation of $\ten{F}$ of MPO rank $K$.
However, using the components in \eqref{def:anniloperator}, we can write $\ten{F}$ in MPO format with rank $2$. To this end, we define
\begin{align*}
F_1 = \begin{bmatrix}
I & \lambda_1 A^*A
\end{bmatrix},  \qquad
F_k = \begin{bmatrix}
I & \lambda_k A^*A\\ 0 & I
\end{bmatrix}, \;\;  k =2,\ldots, K-1,
\qquad
F_K = \begin{bmatrix}
\lambda_K A^*A \\ I
\end{bmatrix}.
\end{align*}
With these blocks, as in \cite{kazeev_low-rank_2012} one immediately verifies that the following representation holds and a linear scaling with respect to $K$ in the termwise representation \eqref{eq:laplacelikeF} can be reduced to a constant rank in the MPO format. 
\begin{lemma}
	We have $\ten{F} = F_1 \SKP F_2 \SKP \cdots \SKP F_K$,
	that is, $\ten{F}$ has an MPO representation of rank $2$.
\end{lemma}

Before we turn to the one- and two-particle operators, we define some notation that will be needed in both cases. For abbreviating blocks with repeated components, we introduce the abbreviations
\begin{equation*}
\rep{I}_k = I_k \uparrow I =\begin{bmatrix}
I & 0 & \cdots & 0\\
0 & I & \cdots & 0\\
\vdots &\vdots & \ddots & \vdots\\
0 & 0 & \cdots & I
\end{bmatrix}\in\mathbb{R}^{k\times 2\times 2\times k},\quad
\rep{S}_k = I_k \uparrow S = \begin{bmatrix}
S & 0 & \cdots & 0\\
0 & S & \cdots & 0\\
\vdots &\vdots & \ddots & \vdots\\
0 & 0 & \cdots & S
\end{bmatrix}\in\mathbb{R}^{k\times 2\times 2\times k},
\end{equation*}
and analogously $\rep{A}_k = I_k \uparrow A$ and $\rep{A}^*_k = I_k \uparrow A^*$, where we write $I_k$ for the identity matrix of size $k$.
 
With this, we can turn to the one-particle operator $\ten{S}$ given by
\[
\ten{S} =\sum_{i,j=1}^Kt_{ij}\ten{a}_i^*\ten{a}_j,
\]
with symmetric coefficient matrix $(t_{ij})_{i,j=1,\ldots, K}$.
Naively, $\ten{S}$ can be written in MPO format with rank $K^2$, but again, one can do better: we now show that $\ten{S}$ in fact can be written in MPO format with rank $K+2$.

For each $k \in \{1,\ldots,K\}$, we define some slices of the coefficient matrix $T = (t_{ij})_{i=1,\ldots, K}^{j=1,\ldots, K}$, where the subscript indices correspond to rows and the superscript indices to columns:
\begin{align*}
W_{T,k}^1 &= (t_{ik})_{i=1,\ldots, k-1},\quad & W_{T,k}^2 &= (t_{kj})_{j=1,\ldots, k-1} , \\
W_{T,k}^3 &= (t_{kj})^{j=K,\ldots, k+1},\quad & W_{T,k}^4 &= (t_{ik})^{i=K,\ldots, k+1}.
\end{align*}
Furthermore, the top-right and bottom-left blocks of $T$ are given by
\begin{equation*}
W_{T}^5  = (t_{ij})_{i=1,\ldots, K/2}^{j=K,\ldots, K/2+1},\qquad \qquad \qquad \qquad  W_{T}^6 = (t_{ij})_{j=1,\ldots K/2}^{i=K,\ldots, K/2+1}.
\end{equation*}
We define the components
\[
 T_1 = \begin{bmatrix}
I & A^* & A & t_{1,1}A^*A
\end{bmatrix}\,,\qquad 
T_K = \begin{bmatrix}
I \\ A \\ A^* \\ t_{K,K}A^*A
\end{bmatrix},
\]
for $k = 2,\ldots,\frac{K}{2}$,
\begin{equation*}
T_k = \begin{bmatrix}
I & 0  &A & 0 &A^*& t_{k,k}A^*A\\
0 & \rep{S}_{k-1}  & 0 &  0 &0 & W_{T,k}^1 \uparrow A^*\\
0 & 0  &0 & \rep{S}_{k-1} &  0 & W_{T,k}^2 \uparrow A\\
0 & 0 &0 & 0&0& I\\
\end{bmatrix},
\end{equation*}
and for $k = \frac{K}{2} + 1, \ldots,K-1$,
\begin{equation*}
T_k = \begin{bmatrix}
I & 0  &0 &  0\\
0 & \rep{S}_{K-k}   & 0   & 0  \\
A & 0  &0 & 0\\
0 & 0 &  \rep{S}_{K-k} & 0\\
A^* & 0  &  0  & 0\\
t_{k,k}A^*A & W_{T,k}^3\uparrow A^* & W_{T,k}^4\uparrow A& I\\
\end{bmatrix}.
\end{equation*}
Finally, let
\begin{equation*}
M_T = \begin{pmatrix}
0 & 0 & 0 & 1\\
0 & 0 &W_{T}^6 & 0 \\
0 & W_{T}^5 & 0 & 0 \\
1 & 0 & 0 & 0 \\
\end{pmatrix}.
\end{equation*}
This allows us to state the one-particle operator explicitly and with (near-)minimal ranks. {The same rank bounds can also be extracted from the alternative MPO representation in \cite[Thm.~4.2]{dolgov2013two} (see also \cite[Lemma 3.2]{kazeev2013low}); the construction we describe here, however, also serves as a preparation for our similar approach to the two-particle case.}

\begin{theorem}\label{thm:oneelecasMPO}
	We have 
	\begin{align*}
	\ten{S} = T_1\SKP T_2\SKP \cdots \SKP T_{ K/2}   M_T \SKP T_{ K/2+1} \SKP \cdots \SKP T_K.
	\end{align*}
	Furthermore, the MPO rank of $\ten{S}$ is bounded by $K+2$. If for some $d \in \N_0$, we have $ t_{ij} = 0$ whenever $\abs{i-j} > d$,
	then the MPO ranks of $\ten S$ are bounded by $2d+2$.
\end{theorem}

{For the proof, we proceed as follows:} We divide the claim into two cases $k \leq \frac{K}{2}$ and $k > \frac{K}{2}$ and show by induction over $k$ that the rank of $T_k$ for $k\leq \frac{K}{2}$ can be bounded by $2+2k$. Consequently, for $k=\frac{K}{2}$ we find that the rank of $\ten{S}$ can be bounded by $K+2$. This is also the bound for the rank of the matrix $M_T$. For the sparse coefficient matrix we directly consider $M_T$, since the rank is maximized at the center of the representation, where the rank of $W_{T}^5$ and the rank of $W_{T}^6$ can be bounded by $d$ in both cases. Thus the rank of $M_T$ is bounded by $2d+2$. 
{The details of the proof are given Appendix~\ref{app:oneelecproof}.}

The case of two-electron operators $\ten{D}$ given by
\[
\ten{D} = \sum_{i_1,i_2,j_1,j_2=1}^K v_{i_1 i_2 j_1 j_2}\ten{a}_{i_1}^* \ten{a}_{i_2}^* \ten{a}_{j_1} \ten{a}_{j_2}
\]
is more involved but can be dealt with quite analogously. We briefly note that due to the anticommutation relations \eqref{eq:anticomm}, one only needs to do the sum over $i_1 < i_2$, $j_1 < j_2$, which reduces the number of terms: By grouping together
\begin{align*}
\tilde v_{i_1 i_2 j_1 j_2} = &\begin{cases}
v_{i_1 i_2 j_1 j_2} + v_{i_2 i_1 j_2 j_1} - v_{i_2 i_1 j_1 j_2} - v_{i_1 i_2 j_2 j_1}, & i_1 < i_2, j_1 < j_2, \\
0 ,& \text{otherwise,}
\end{cases}
\end{align*}
we obtain 
\[
\ten{D} =  \sum_{i_2,j_2=1}^K \sum_{i_1=1}^{i_2-1} \sum_{j_1=1}^{j_2-1} \tilde v_{i_1 i_2 j_1 j_2} \ten{a}_{i_1}^* \ten{a}_{i_2}^* \ten{a}_{j_1} \ten{a}_{j_2}.
\]
We denote the tensor grouping the coefficients of $V$ by $\tilde V=\left(\tilde v_{i_1 i_2 j_1 j_2}\right)_{i_1 i_2 j_1 j_2=1}^K$. Again, the Kronecker-rank of this operator is $\binom{K}{2}^2 = O(K^4)$, so naively $\ten D$ could be written with MPO rank $\binom{K}{2}^2$. But we can do better.
With the help of  the matrices in \ref{def:anniloperator} we can write $\ten{D}$ in MPO format with rank $\frac12 K^2 + \frac32{K} + 2$.

As before we need to extract different matrix slices from $\tilde V$, where again subscript indices correspond to rows and superscript indices to columns:
\begin{align*}
W_{V,k}^1 &= (\tilde v_{i_1kkj_2})_{i_1=1,\ldots, k-1}^{j_2=k+1,\ldots, K} \quad &  W_{V,k}^2 & = (\tilde v_{ki_2kj_2})_{j_2=1,\ldots, k-1}^{i_2=k+1,\ldots, K}\\
W_{V,k}^3 &= (\tilde v_{i_1kj_1j_2})_{i_1=1,\ldots, k-1;\,j_1=1,\ldots, k-1}^{j_2=k+1\cdots K} \quad & W_{V,k}^4 &= (\tilde v_{i_1i_2j_1k})_{i_1=1,\ldots, k-1;\,j_1=1,\ldots, k-1}^{i_2=k+1,\ldots, K}\\
W_{V,k}^5 &= (\tilde v_{i_1kj_1k})_{i_1=1,\ldots, k-1 ;\, j_1=1,\ldots, k-1} \quad & W_{V,k}^6 &= (\tilde v_{i_1i_2kj_2})_{i_1=1,\ldots, i_2-1;\,i_2=2,\ldots, k-1}^{j_2=k+1,\ldots, K}\\
W_{V,k}^7 &= (\tilde v_{ki_2j_1j_2})_{j_1=1,\ldots, j_2-1 ;\,j_2=2,\ldots, k-1}^{i_2=k+1,\ldots, K}\quad & W_{V,k}^8 &= (\tilde v_{ki_2j_1j_2})_{i_1=1,\ldots, K/2 ;\, j_1=1,\ldots, K/2}^{i_2=K,\ldots, K/2+1 ;\, j_2=K,\ldots, K/2+1}\\
W_{V,k}^9 &=  (\tilde v_{ki_2j_1j_2})_{j_1=2,\ldots, K/2;\, j_2=1,\ldots, j_1-1}^{i_1=K,\ldots, i_2+1 ;\, i_2=K-1,\ldots, K/2+1}\quad  & W_{V,k}^{10}  &= (\tilde v_{ki_2j_1j_2})_{i_1=1,\ldots, i_2-1 ;\, i_2=2,\ldots, K/2}^{j_1=K-1,\ldots, K/2+1; \, j_2=K,\ldots, j_1+1}
\end{align*}
For $k = 1,\ldots, \frac{K}{2}$ let us define the blocks  
\begin{align*}
V_k^{1,1} &= \begin{bmatrix}I & 0 & A^* & 0 & A&0&A^*A&0&0&0&0&0&0\\
0 & \rep{S}_{k-1} & 0 & 0 & 0&0&0&\rep{A}_{k-1}&0&0&\rep{A}_{k-1}^*&0&0\\
0 & 0 & 0 & \rep{S}_{k-1} & 0&0&0&0&\rep{A}_{k-1}^*&0&0&0&\rep{A}_{k-1}\\
0&0&0&0&0 & \rep{I}_{k^2}&0&0&0&0&0&0&0\\
0&0&0&0&0 &0&0&0&0&\rep{I}_{\binom{k-1}{2}}&0&0&0\\
0&0&0&0&0 &0&0&0&0&0&0&\rep{I}_{\binom{k-1}{2}}&0
\end{bmatrix}\\
\intertext{as well as}
V_k^{1,2} &= \begin{bmatrix}
0 &0 &0\\
W_{V,k}^1\uparrow A^*A  &0&0\\
0&W_{V,k}^2\uparrow A^*A & 0\\
W_{V,k}^3\uparrow A^*&W_{V,k}^4\uparrow A&W_{V,k}^5\uparrow A^*A\\
W_{V,k}^6\uparrow A&0&0\\
0&W_{V,k}^7\uparrow A^*&0\\
\end{bmatrix} , \qquad \qquad 
V_k^{2,2} =\begin{bmatrix}
0&0&A\\
\rep{S}_{K-k}&0&0 \\
0&0&A^*\\
0&\rep{S}_{K-k}&0\\
0&0&I
\end{bmatrix}\,.
\end{align*}
With these blocks we have
\begin{equation*}
V_1 = V_1^{1,1} \quad V_2 = \begin{bmatrix}
V_2^{1,1} & V_2^{1,2}\end{bmatrix} 
\quad \mathrm{and} \quad V_k = \begin{bmatrix}
V_k^{1,1} & V_k^{1,2}\\ 0 &V_k^{2,2}
\end{bmatrix}, \, k = 3,\ldots, \frac K2.
\end{equation*}
Furthermore, we set
\begin{align*}
M_V =& \begin{pmatrix}
0&\cdots&0&0&0&\cdots&1\\
\vdots&\ddots&\vdots&\vdots&\vdots&\iddots&\vdots\\
0&\cdots&0&0&1&\cdots&0\\
0&\cdots&0&V_{\mathrm{mid}} &0&\cdots&0\\
0&\cdots&1&0&0&\cdots&0\\
\vdots&\iddots&\vdots&\vdots&\vdots&\ddots&\vdots\\
1&\cdots&0&0&0&\cdots&0\\
\end{pmatrix}, \qquad \qquad V_{\mathrm{mid}} = \begin{pmatrix}
W_{V,k}^8&0&0\\
0&0&W_{V,k}^{10}\\
0&W_{V,k}^9&0\\
\end{pmatrix} \,,
\end{align*}
where there are $\frac{K}{2}+1$ ones on each of the two antidiagonals above and below $V_{\mathrm{mid}}$.
Finally, for $k = \frac{K}{2}+1,\ldots, K$, we analogously obtain 
\begin{equation*}
V_K = V_K^{1,1} \quad V_{K-1} = \begin{bmatrix}
V_{K-1}^{1,1} \\ V_{K-1}^{2,1}\end{bmatrix} 
\quad \mathrm{and} \quad V_k = \begin{bmatrix}
V_k^{1,1} & 0\\ V_k^{2,1} &V_k^{2,2}
\end{bmatrix}, \, k = \frac K2 +1,\ldots,  K-2. 
\end{equation*}
where $\left(V_k^{2,1}\right)^T$ is similar to $V_k^{1,2}$ with modified coefficients.  

With the necessary notation out of the way, we immediately state the MPO representation of the two-particle operator with near-minimal ranks.
\begin{theorem}\label{thm:twoelecasMPO}
We have
\begin{align*}
\ten{D} = V_1\SKP V_2\SKP \cdots \SKP V_{K/2}  M_V \SKP V_{ K/2 + 1}\SKP \cdots \SKP V_K,
\end{align*}
implying that $\ten D$ has an MPO representation of rank $\frac12 {K^2}+ \frac32 {K}+2$. If there exists $d\in\N_0$ such that $v_{i_1 i_2 j_1 j_2} = 0$ whenever  \begin{equation}\label{eq:localityassumption}  \max\{ \abs{i_1-i_2},\abs{i_1-j_1},\abs{i_1-j_2},\abs{i_2-j_1},\abs{i_2-j_2},\abs{j_1-j_2}  \} > d,\end{equation} then the MPO ranks of $\ten{D}$ are bounded by $d^2+3d-1$ if $d$ is odd and by $d^2+3d-2$ if $d$ is even.
\end{theorem}
\begin{proof}
We can proceed as in the proof of Theorem~\ref{thm:oneelecasMPO}. Again, a detailed proof can be found in Appendix~\ref{app:twoelecproof}.
\end{proof}

\begin{remark}
It is possible to incorporate the one-electron operator into the two-electron operator, such that the MPO ranks of $\ten{S}+\ten{D}$ are also bounded by $\frac12 {K^2} + \frac32 {K} +2$. Note that the stated rank bounds are not sharp for the leading and trailing cores, where further reductions are possible with additional technical effort (see also Sec.~\ref{ssc:opranktests}).
\end{remark}

\subsection{Matrix-Free Operations on Block Structures}
\label{ssc:matfree}
{Corollary}~\ref{cor:blockstructure} implies that particle number-preserving operators must also preserve the block structure. In other words, if $\ten{x} \in \cF^K_N$ and if $\ten{B}$ is any particle number-preserving operator, then $\ten{y} := \ten{B} \ten{x} \in \cF^K_N$ has a representation with block structure according to {Corollary}~\ref{cor:blockstructure}. However, if $\rep{B}$ is an MPO representation of $\ten{B}$ as derived in Section \ref{ssc:compactform} and $\ten{x} = \tau(\rep{X})$ with block-sparse $\rep{X}$, this leaves open the questions whether we can directly obtain the block-sparse representation of $\ten{y}$ from the standard representation $\rep{Y} := \rep{B} \bullet \rep{X}$ of the matrix-vector product, or whether additional transformations of $\rep{Y}$ are required to extract the block structure. We now describe how the blocks of $\ten{y}$ can be obtained directly by replacing the component matrices $I$, $S$, $A$, $A^*$, $A^* A$ in $\rep{B}$ by certain matrix-free operations on the blocks of $\rep{X}$. {Furthermore, these operations are performed entirely componentwise, that is, each component can be computed separately from the others and even partial evaluations are possible.}

It turns out that each summand in~\eqref{eq:partop} acts on the block-sparse MPS by shifting and deleting some of the blocks. This can be visualized by considering the one-particle part of the Hamiltonian, specifically the case where $i = j$:
\begin{equation*}
\ten a_i^* \ten a_i = I \otimes \dots \otimes A^* A \otimes \dots \otimes I. 
\end{equation*}
Since this operator has Kronecker rank one, each matrix in the product acts only on the corresponding component in the block MPS. Clearly, the identity matrices leave their components and their respective block structure unchanged. Only the matrix 
\[
A^*A = \begin{pmatrix} 0 & 0 \\ 0 & 1 \end{pmatrix}
\] has the immediate effect of assigning zero to all unoccupied blocks,
\begin{equation}\label{eq:AsAop}
\unocc{X}_{i,n} \xleftarrow[]{A^* A} 0 \quad \text{ for } n \in \cK_{i-1} \cap \cK_i.
\end{equation}
Thus in this case, the particle number is preserved \emph{locally} at orbital $i$ and therefore, the block structure remains otherwise unchanged. 

Additional difficulties appear, however, when $i\neq j$, since the particle number is then conserved only by the combination of operations on different modes. Let us first consider $i<j$,
\begin{equation*}
\ten a_i^* \ten a_j = I \otimes \dots \otimes A^* \otimes S \otimes \dots \otimes S \otimes A \otimes \dots \otimes I. 
\end{equation*}
To avoid technicalities, for the moment we assume $N < i < j < K - N + 1$, corresponding to the generic case where all blocks appear in each core.
Again, the identity matrices leave everything unchanged. 
The creation matrix $A^*$
replaces the occupied layer $X_i^{\{1\}}$ by the unoccupied layer $X_i^{\{0\}}$.
However, this clearly violates the block structure, because occupied blocks should only be located on the off-diagonal and because $N \notin (\cK_i - 1)$. This inconsistency can only be resolved by noting that the added particle will be removed further down in the $j$-th position of the tensor.
Additionally, we have to take into account that a particle was added to the left of all following components, thus increasing the particle count $n$ by one in each block. The solution to the block structure violation therefore lies in shifting the corresponding blocks and deleting the ones that violate particle number counts.
We summarize the case $N < i < j < K - N + 1$ as follows:
\[
A^* \bullet X_i = \begin{bmatrix}
0 & \occ{X}_{i,0}^{\uparrow} & \cdots & 0 \\[-3pt] 
0 & 0 & \ddots & \vdots \\[-2pt]
\vdots & \vdots & \ddots & \occ{X}_{i,N-1}^{\uparrow} \\[3pt]
0 & 0 & \cdots & 0
\end{bmatrix}, \quad
A \bullet X_j = \begin{bmatrix}
0 & 0 & \cdots & 0 \\[-3pt] 
0 & \unocc{X}_{j,1}^{\uparrow} & \ddots & \vdots \\[-2pt]
\vdots & \vdots & \ddots & 0\\[3pt]
0 & 0 & \cdots & \unocc{X}_{j,N}^{\uparrow}
\end{bmatrix},
\]
where application of $A^*$ corresponds to the block operations
\begin{equation}\label{eq:Asop}
\begin{aligned}
\unocc{X}_{i,n} \hspace{11pt} &\xleftarrow[]{A^*} 0 && \quad \text{ for } n \in \cK_{i-1} \cap \cK_i, \\
\occ{X}_{i,n} \hspace{11pt} &\xleftarrow[]{A^*} \unocc{X}_{i,n} && \quad \text{ for } n \in \cK_{i-1} \cap (\cK_i - 1), 
\end{aligned}
\end{equation}
application of $A$ to the block operations
\begin{equation}\label{eq:Aop}
\begin{aligned}
\unocc{X}_{j,0} \hspace{12pt} &\xleftarrow[]{A} 0, && \\
\unocc{X}_{j,n+1} &\xleftarrow[]{A} \occ{X}_{j,n} && \quad \text{ for } n \in \cK_{j-1} \cap (\cK_j - 1), \\
\occ{X}_{i,n} \hspace{11pt} &\xleftarrow[]{A} 0 && \quad \text{ for } n \in \cK_{j-1} \cap (\cK_j - 1),
\end{aligned}
\end{equation}
and for $i < k < j$, applying $S$ amounts to the block operations
\begin{equation}\label{eq:Sop}
\begin{aligned}
\unocc{X}_{k,0} \hspace{12pt} &\xleftarrow[]{S} 0, && \\
\occ{X}_{k,0} \hspace{12pt} &\xleftarrow[]{S} 0, && \\
\unocc{X}_{k,n+1} &\xleftarrow[]{S} \unocc{X}_{k,n} && \quad \text{ for } n \in (\cK_{k-1} \cap \cK_k) \setminus \{0\}, \\
\occ{X}_{k,n+1} &\xleftarrow[]{S} -\occ{X}_{k,n} && \quad \text{ for } n \in (\cK_{k-1} \cap (\cK_k - 1)) \setminus \{0\}.
\end{aligned}
\end{equation}
The case where $j < i$ can be dealt with analogously, but with the opposite shift because a particle gets removed on the left and all particle counts $n$ have to be decreased by one until a particle gets added again in component $i$. This means that we have to distinguish the different cases in the implementation of, for instance, the action of the matrix $S$.

\begin{remark}\label{rem:blockoptechnicalities}
Some further technicalities need to be taken into account in an implementation:
\begin{enumerate}[{\rm(i)}]
\item The sizes of blocks that are set to zero in the above operations is dictated by the consistency of the representation: zero blocks of nontrivial size need to be kept, whereas redundancies due to zero blocks with a vanishing dimension need to be removed.
\item For the border cases, such as $i < N$ or $j > K-N+1$, certain blocks do not occur. The accordingly modified sets $\cK_k$ and the corresponding differences in the blocks that are present lead to modifications in the above operations.
\end{enumerate}
\end{remark}

Figure~\ref{fig:oneparticleop} shows the different cases in the implementation, including the blocks that need to be deleted in order to avoid irregularities.
Here $A^*A$ corresponds to the block operation \eqref{eq:AsAop}; $A^*_\ell$, $A_r$ and $S^+$ correspond to \eqref{eq:Asop}, \eqref{eq:Aop}, and \eqref{eq:Sop}, respectively; and $A^*_r$, $A_\ell$ and $S^-$ are the analogous operations with opposite shifts.
The figure assumes the generic block structure for $i,j \in \{N+1,\ldots,K-N\}$ and thus needs to be modified for the border cases mentioned in Remark \ref{rem:blockoptechnicalities}(ii).
\begin{figure}[t]
\begin{center}
\includegraphics[width=\textwidth]{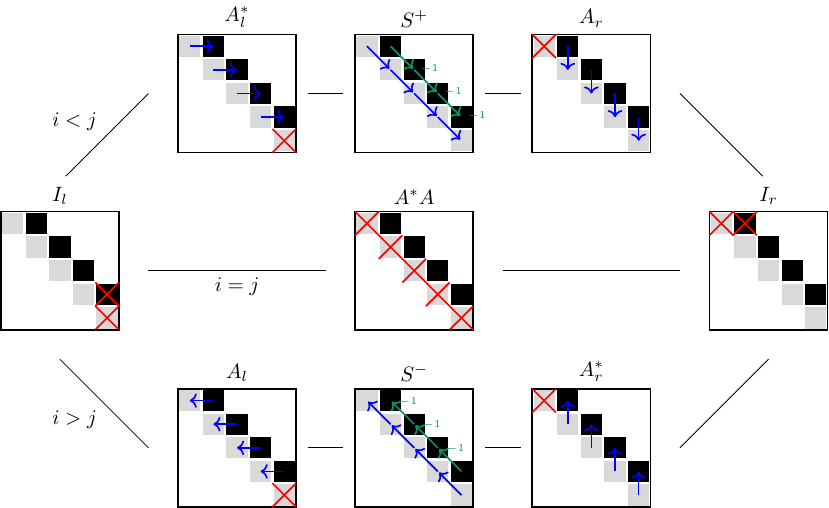}
\end{center}
\caption{An illustration of the matrix free operations for the one-particle operator.}
\label{fig:oneparticleop}
\end{figure}

In order to apply tensor representations of general operators of the form $\ten a_i^* \ten a_j$ with a result given in the same block structure, the components need to be replaced by block operations with particle number semantics as in Figure \ref{fig:oneparticleop}, depending on their position $k = 1,\ldots,K$:
\begin{align*}
A &\to \begin{cases}
A_l &\quad \text{if } k = j < i, \\
A_r &\quad \text{if } k = j > i,
\end{cases} & 
A^* &\to \begin{cases}
A_l^* &\quad \text{if } k = i < j, \\
A_r^* &\quad \text{if } k = i > j,
\end{cases} \\
I &\;\to\; \begin{cases}
I_l &\quad \text{if } k < \min(i,j), \\
I_r &\quad \text{if } k > \max(i,j),
\end{cases} & 
S &\to \begin{cases}
S^+ &\quad \text{if } i < k < j, \\
S^- &\quad \text{if } j < k < i.
\end{cases} 
\end{align*}
In addition, we have the $k$-independent replacement of $A^*A$ by the operation \eqref{eq:AsAop} and replace zero components by the operation $Z$ that assigns zero to all blocks. 

These operations can be performed efficiently by exchange, removal or sign changes of blocks in the tensor representation. One can proceed analogously for two-particle operators $\ten{a}_{i_1}^* \ten{a}_{i_2}^* \ten{a}_{j_1} \ten{a}_{j_2}$, as shown in Figure~\ref{fig:twoparticleop}, where again one needs to make appropriate adjustments for the border cases. In a similar manner, this can be generalized to interactions of three or more particles. 

\begin{figure}[t]
\begin{center}
\includegraphics[height=\textwidth,angle=90]{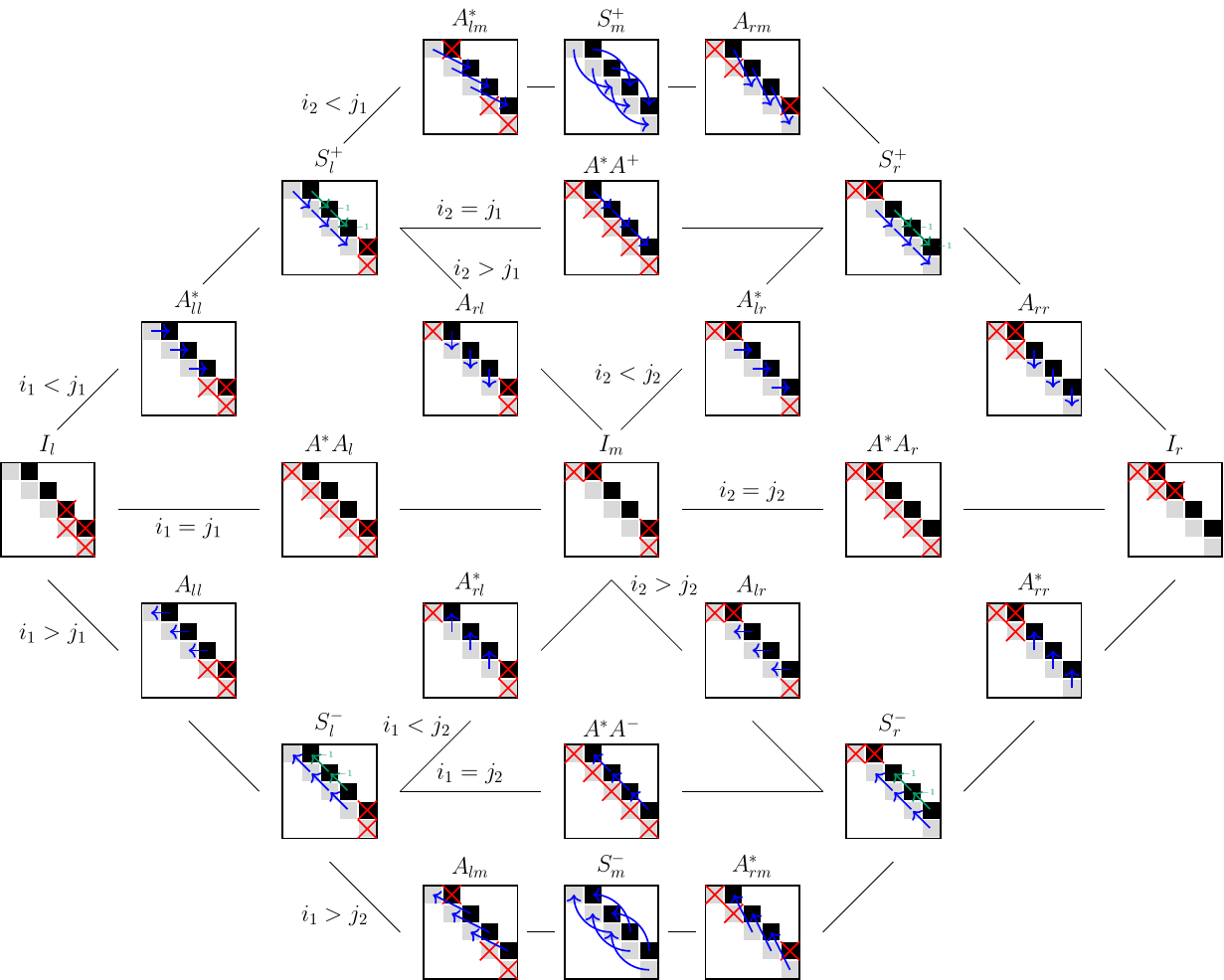}
\end{center}
\caption{An illustration of the matrix-free operations for the two-particle operator, assuming $i_1<i_2$ and $j_1 < j_2$ without loss of generality.}
\label{fig:twoparticleop}
\end{figure}

\subsection{{Automatic Rank Reduction}}

The particle number semantics of the block operations according to Figure \ref{fig:oneparticleop} and \ref{fig:twoparticleop} are compatible with forming linear combinations of operators. 
In particular, the full one- and two-particle operator representations constructed in Section \ref{ssc:compactform} can in the same manner be applied entirely in terms of block operations: replacing $A^* A$, $A^*$, $A$, $I$, $S$, and $0$ by the respective $k$-dependent block operations {according to Figures \ref{fig:oneparticleop}, \ref{fig:twoparticleop} again leads to a consistent representation, and its application directly produces the correct block structure. For this, let $\rep S_{\mathrm{b}}$ and $\rep D_{\mathrm{b}}$ be the resulting representations of $\ten S$ and $\ten D$ with particle number semantics.
\begin{prop}
Let $\ten x = \rmap{\rep X}$ be a block-sparse MPS representation, then $\rep U = \rep S_{\mathrm{b}} \bullet \rep X$ and $\rep V = \rep D_{\mathrm{b}} \bullet \rep X$ are block-sparse MPS representations with $\ten S \ten x = \rmap{\rep U}$, $\ten D \ten x = \rmap{\rep V}$.
\end{prop}
\begin{proof}
	By inspection of the proofs of Thm.\ \ref{thm:oneelecasMPO} and \ref{thm:twoelecasMPO}, one finds that each rank index in the MPO representations constructed there corresponds to precisely one case in Fig.~\ref{fig:oneparticleop} or \ref{fig:twoparticleop}, respectively. The  resulting representations operating on blocks thus directly yield a consistent block-sparse representation of the matrix-vector products.
\end{proof}}

{The addition of such symbolic MPO representations can be done analogously to the addition of MPS and block-sparse MPS. In each core, these symbolic representations are composed of scalar multiples of the elementary matrix-free block operations discussed above, where the corresponding scalars can be collected in a separate matrix of coefficients. We say that a collection of columns of cores in rank-wise representation of matrix-free operators are {\em linearly dependent} if they contain the same symbols but the coefficient matrix is rank-deficient. In this case, the operator ranks can be reduced, provided that the corresponding rows in the next component are {\em compatible}, meaning that entrywise, their symbols can only differ if all but one of them are the zero symbol $Z$, which in turn means that they can be added. The resulting algorithm can be performed from left to right and the procedure can be repeated from right to left, where linearly dependent rows can be merged, see Algorithm~\ref{alg:rankredop}. This automatically reduces the ranks of the sums of operators in the above symbolic representation. It can be very useful if the rank-reduced format for an operator is not known. In fact, as we show experimentally in Section \ref{ssc:opranktests}, automatic rank reduction can even improve upon the operator representation derived in Sec.~\ref{ssc:compactform}. 
\begin{example}\label{ex:symop}
Let $K = 5$. The operators $\ten a_2^* \ten a_2$, $\ten a_2^* \ten a_4$ and $\ten a_4^* \ten a_2$ can be represented with particle number semantics by
\begin{align*}
\ten a_2^* \ten a_2 &= [ I_l ] \SKP [ A^* A ] \SKP [ I_r ] \SKP [ I_r ] \SKP [ I_r ], \\
\ten a_2^* \ten a_4 &= [ I_l ] \SKP [ A^*_l ] \SKP [ S^+ ] \SKP [ A_r ] \SKP [ I_r ], \\
\ten a_4^* \ten a_2 &= [ I_l ] \SKP [ A_l ] \SKP [ S^- ] \SKP [ A^*_r ] \SKP [ I_r ].
\end{align*}
The sum of the operators is given by
\begin{multline*}
\ten a_2^* \ten a_2 + \ten a_2^* \ten a_4 + \ten a_4^* \ten a_2  =  \\ \begin{bmatrix}
I_l & I_l & I_l 
\end{bmatrix} \SKP \begin{bmatrix}
A^*A & Z & Z \\
Z & A^*_l & Z \\
Z & Z & A_r
\end{bmatrix} \SKP \begin{bmatrix}
I_r & Z & Z \\
Z & S^+ & Z \\
Z & Z & S^-
\end{bmatrix} \SKP \begin{bmatrix}
I_r & Z & Z \\
Z & A_r & Z \\
Z & Z & A^*_r
\end{bmatrix} \SKP \begin{bmatrix}
I_r \\
I_r \\
I_r
\end{bmatrix}.
\end{multline*}
After the rank reduction, the operator has the form
\begin{equation*}
\ten a_2^* \ten a_2 + \ten a_2^* \ten a_4 + \ten a_4^* \ten a_2 = \begin{bmatrix}
I_l 
\end{bmatrix} \SKP \begin{bmatrix}
A^*A & A^*_l & A_l
\end{bmatrix} \SKP \begin{bmatrix}
I_r & Z & Z \\
Z & S^+ & Z \\
Z & Z & S^-
\end{bmatrix} \SKP \begin{bmatrix}
I_r \\
A_r \\
A^*_r
\end{bmatrix} \SKP \begin{bmatrix}
I_r 
\end{bmatrix}.
\end{equation*}
\end{example}
}

\begin{algorithm}[t!]
\label{alg:rankredop}
\SetKwInOut{Input}{input}\SetKwInOut{Output}{output}
\Input{Matrix-free operator $\rep H$ with cores $\rep H_k$ that have coefficent matrices $G_k = (g^k_{i,j})_{i \in \{1,\ldots,r_k\}}^{j \in \{1,\ldots,r_{k+1}\}}, k = 1,\ldots,K$, with operator ranks $r_1,\ldots,r_{K-1}$.}
\Output{Rank-reduced matrix-free operator representation $\rep H$.}
\BlankLine

\For{$k = 1,2,\ldots,K-1$}{

Find $L$ sets $\mathcal L_\ell \subset \{1,\ldots,r_{k-1}\}$ s.t.\ the respective columns of $\rep H_k$ are linearly dependent and the respective rows of $\rep H_{k+1}$ are compatible\;

\For{$\ell = 1,\ldots,L$}{
Compute reduced QR factorization $(g^k_{i,j})_{i \in \{1,\ldots,r_k\}}^{j \in \mathcal L_\ell} = QR$\;
Set $(g^k_{i,j})_{i \in \{1,\ldots,r_k\}}^{j \in \mathcal L_\ell} \leftarrow Q$\;
Set $(g^{k+1}_{i,j})_{i \in \mathcal L_\ell}^{j \in \{1,\ldots,r_{k+2}\}} \leftarrow R \cdot (g^{k+1}_{i,j})_{i \in \mathcal L_\ell}^{j \in \{1,\ldots,r_{k+2}\}}$\;
Reduce columns of $\rep H_k$ and add columns of $\rep H_{k+1}$ accordingly\;
Update $r_k$\;
}
}
\Return{$\ten H = \rmap{\rep H}$}
\caption{Left-to-right rank reduction algorithm for matrix-free operators.}
\end{algorithm}

\section{Numerical Aspects}\label{sec:numer}

This chapter serves as an outlook on numerical solvers for the eigenvalue problem 
$
\ten H \ten x = \lambda \ten x
$ with the additional constraint $\ten P \ten x = N \ten x$ implemented by keeping $\ten x$ in block-sparse format.
We comment on standard iterative solvers and discuss their relation to this representation format. Furthermore, we give an example on the effect of enforcing block sparsity on the numerical stability of particle numbers with respect to TT-SVD truncation. Finally, we show that the ranks of the one- and two-particle operators, as discussed in Sec.~\ref{ssc:compactform}, are indeed near-optimal.

\subsection{Iterative Methods with Fixed and Variable Ranks}\label{sec:itermethods}
A standard method for the computation with MPS is the DMRG algorithm. All modern implementations of this method (see, for instance, \cite{itensor,tensornetwork,tenpy,pytenet}) exploit the block sparsity in some form. For the sake of completeness, we give a brief overview of both the one-site and the two-site DMRG. We then turn to methods using global eigenvalue residuals. These methods are nonstandard in physical computations, but may become competitive when block sparsity is taken into account. A detailed numerical comparison of the methods will be subject of further research.

\subsubsection{One-site DMRG/ALS}
The one-site DMRG or ALS algorithm \cite{Holtz:12} optimizes one component of the MPS $\ten x$ at a time. With the appropriate orthogonalization, each subiteration consists of an optimization step on the linear part of the fixed-rank manifold, which coincides with its own tangent space. As such, the one-site DMRG can be formulated as a tangent space prodedure: Let $\ten x_{k,\ell}$ be the current iterate after $\ell$ sweeps and the $k$-th subiteration. That is, we have previously optimized the $k$-th component and orthogononalized accordingly. Now, we optimize the $(k+1)$-st component by minimizing the energy
\[
\ten E_{k,\ell}(\ten x_{k+1,\ell}) = {\frac{\langle \ten x_{k+1,\ell},\ten Q_{\ten x_{k,\ell}}^{k+1,1}\ten H \ten Q_{\ten x_{k,\ell}}^{k+1,1}\ten x_{k+1,\ell} \rangle}{\langle \ten x_{k+1,\ell}, \ten x_{k+1,\ell}\rangle}}.
\] 
If $k = K$, we can go back to $k = 1$ or do the sweep in reverse. We note that $\ten Q_{\ten x_{k,\ell}}^{k+1,1}$ is exactly the projection onto the part of the tangent space at $\ten x_{k,\ell}$ that corresponds to the $(k+1)$-st component. If $\ten x_{k,\ell}$ is an eigenvector of the particle number operator $\ten P$, then by Corollary~\ref{cor:tangprok}, $\ten Q_{\ten x_{k,\ell}}^{k+1,1}$ commutes with $\ten P$. By Lemma~\ref{lemma:hamilpncommute}, so does the Hamiltonian $\ten H$. Thus, the next iterate $\ten x_{k+1,\ell}$ will be in the same eigenspace of $\ten P$. Therefore, if one initializes the one-site DMRG algorithm with a block-sparse MPS of fixed particle number, then the block sparsity will be preserved for each iterate and the algorithm can be performed by operating only on the nonzero blocks.

\subsubsection{Two-site DMRG}
The classical (two-site) DMRG \cite{white,Holtz:12} optimizes two neighboring components at once. This allows for a certain rank-adaptivity in between these components. While this gives the algorithm more flexibility, it also means that the subiterates can leave the fixed-rank manifold and even the tangent space. Nevertheless, we can show that the particle number will be preserved. To this end, we define the operation $\tilde{ \ten Q}_{\ten x}^{k,1}$ for $k=1,\ldots,K-1$ similarly  to $ \ten Q_{\ten x}^{k,1}$ by 
\begin{equation*}
\tilde{\ten Q}_{\ten x}^{k,1} = \biggl(\sum_{j_{k-1}=1}^{r_{k-1}}\rmapless{k}{j_{k-1}}{\rep{U}}\, \langle \rmapless{k}{j_{k-1}}{\rep{U}},\, \cdot\, \rangle \biggr)
\otimes I \otimes I \otimes \biggl(\sum_{j_{k+1}=1}^{r_{k+1}}\rmapgtr{k+1}{j_{k+1}}{\rep{V}}\,\langle\rmapgtr{k+1}{j_{k+1}}{\rep{V}},\,\cdot\,\rangle \biggr).
\end{equation*}
As in Corollary~\ref{cor:tangprok}, it can be shown that $\tilde{\ten Q}_{\ten x}^{k,1}$ and $\ten P$ commute. Thus, with the same argument as above, if the first iterate is an eigenvector of $\ten P$, then all iterates are in the same eigenspace.

\subsubsection{(Preconditioned) Gradient Descent}
An alternative to the DMRG algorithm are methods operating \emph{globally} on the MPS representation, such as (preconditioned) gradient descent or more involved variants such as LOBPCG \cite{Kressner:11}. For basic gradient descent, one can control the ranks by defining a threshold $\epsilon > 0$ and performing the update scheme
\[
\ten x_{\ell+1} = \operatorname{trunc}_\epsilon\left(\ten x_\ell - \alpha_\ell \left(\ten H\ten x_m - \frac{\langle \ten x_\ell, \ten H \ten x_\ell \rangle}{\langle \ten x_\ell,\ten x_\ell \rangle}\ten x_\ell\right)\right).
\]
Since all involved steps preserve the particle number, this scheme produces a sequence $\ten x_\ell$ with the same particle number if the initial value $\ten x_0$ has a fixed particle number. Convergence can be accelerated by using an optimized step size $\alpha_\ell$ or by preconditioning the system \cite{Rohwedder:11}. 

\subsubsection{Riemannian Gradient Descent}
One could also consider Riemannian methods, where the gradient is projected first onto the tangent space and the step is performed on the fixed-rank manifold \cite{Kressner_SV_2013}. Generalizations are possible that allow for rank adaptivity. This method is often used because the ranks can be fixed and because the projected gradient in the tangent space can be stated explicitly and compactly, thus reducing computational overhead. We construct a sequence $\ten x_\ell$ from an initial value $\ten x_0$ with initial rank $r$. If $\ten x_0$ has a fixed particle number, then so does the entire sequence
\[
\ten x_{\ell+1} = \operatorname{trunc}_r \left(\ten x_\ell - \alpha_\ell \ten Q_{\ten x_\ell}\left(\ten H\ten x_\ell - \frac{\langle \ten x_\ell, \ten H \ten x_\ell\rangle}{\langle \ten x_\ell,\ten x_\ell\rangle}\ten x_\ell\right)\right),
\]
where $\alpha_\ell$ is the step size.  In \cite{steinlechner_riemannian_2016}, it is shown that the truncation to fixed rank is a retraction, and thus the stated scheme can be regarded as a Riemannian optimization method. These methods can be accelerated by typical techniques for gradient descent, such as nonlinear conjugate gradient descent, see \cite{Absil2008}.

\subsection{Blocks of Zero Size}
We usually assume a tensor $\ten x$ to be represented with minimal ranks; otherwise, we can perform a TT-SVD truncation with a given error threshold or to a fixed multilinear rank as in Algorithm~\ref{alg:svdthresh}. This means that in the block-sparse format, all blocks that contain only zeros will be actually set to size zero, which has several implications.

First of all, as already mentioned in Sec.~\ref{sec:blockstructure}, we stress that truncating the ranks of a tensor $\ten x$ to a fixed multilinear rank $r_1,\ldots,r_{K-1}$ can lead to the tensor being set to zero, that is, $\rmap{\rep{\operatorname{trunc}_{r_1,\ldots,r_K}(\rep X)}} = 0$. The block-sparse format allows for a deeper understanding of this fact: Setting a block to zero can lead to the tensor as a whole being set to zero, if all other nonzero blocks depend on it. 

Furthermore, the question arises whether the above iterative methods can recover a block after it has been momentarily set to zero during an iteration step. We know that the one-site DMRG is not rank-adaptive and the block sizes are fixed during the iteration. In the other three methods it is possible to increase the rank based on some threshold (in the Riemannian case this can be achieved by modification of the retraction onto the manifold).  

If $\rho_{k,n} = 0$ for some $k$ and $n$, then there exists a basis element $\ten e^\alpha$ of the eigenspace of the particle number operator with eigenvalue $N$ (that is, $\ten P \ten{e}^\alpha = N \ten{e}^\alpha$), such that $\langle \ten{e}^\alpha, \ten x\rangle = 0$. Then we have $\ten Q_{\ten x}^{k,1} \ten{e}^\alpha = 0$, since $\rmapless{k}{j}{\rep U}$ is not present for $j\in\cS_{k,n} = \emptyset$. A similar argument can be made for $\ten Q_{\ten x}^{k,2}$. This means that in Riemannian gradient descent, once $\rho_{k,n} = 0$ in some iterative step, then also $\rho_{k,n} = 0$ for all subsequent steps. One can overcome this problem by choosing the initial point $\ten x_0$ in such a way that all $\rho_{k,n}$ are at least 1. However, when retracting back onto the manifold, care needs to be taken that blocks are not set to zero. 

For the two-site DMRG, a similar argument implies that some blocks can be created in each substep, depending on neighboring block sizes. A thorough analysis shows that the rank adaptivity of the two-site DMRG is always local and thus not all points can necessarily be reached from a given starting point $\ten x_0$. However, if we start from a generic point (with some block sizes possibly zero), we can expect a favorable behavior of the method.

The general gradient case is in this regard the most versatile as it has the fewest restrictions on the update step. Starting in $\ten x_0 \neq 0$ will allow us to optimize on the whole linear space of fixed particle number $N$ throughout the procedure. 
A more detailed investigation will be given elsewhere.

\subsection{Numerical Stability of Rounding}\label{ssc:rounding}
\pgfplotsset{compat=1.5}
\begin{figure}[t]\label{fig:svdrounding}
	\centering
\begin{tikzpicture}

\begin{loglogaxis}[
log basis x={10},
tick align=outside,
tick pos=left,
grid=both,
xlabel={$| \sigma_6 - \sigma_7 |$},
xmin=9e-16, xmax=1.8,
xmode=log,
xtick style={color=black},
xtick={1e-1,1e-3,1e-5,1e-7,1e-9,1e-11,1e-13},
ylabel={$| \langle \boldsymbol{x}_{6,\epsilon}, \boldsymbol{P} \boldsymbol{x}_{6,\epsilon} \rangle / \langle \boldsymbol{x}_{6,\epsilon}, \boldsymbol{x}_{6,\epsilon} \rangle - N |$},
ymode=log,
xlabel near ticks,
ylabel near ticks,
ytick style={color=black},
ytick={1e-2,1e-8,1e-14},
width=.95\textwidth,
height=5cm
]
\addplot [semithick, red, mark size=2, mark options={solid}, only marks]
table {%
0.9999999999999996 1.7763568394002505e-15
0.49999999999999867 1e-16
0.24999999999999623 8.881784197001252e-16
0.12499999999999911 1e-16
0.06250000000000111 8.881784197001252e-16
0.03125000000000111 8.881784197001252e-16
0.015625000000000666 8.881784197001252e-16
0.007812500000000666 8.881784197001252e-16
0.00390625000000111 1e-16
0.001953124999995115 1e-16
0.0009765625000022204 1.7763568394002505e-15
0.0004882812499991118 8.881784197001252e-16
0.00024414062499666933 8.881784197001252e-16
0.0001220703125024425 8.881784197001252e-16
6.103515625177636e-5 8.881784197001252e-16
3.051757812544409e-5 1e-16
1.5258789061611822e-5 1e-16
7.629394530805911e-6 8.881784197001252e-16
3.814697263848643e-6 1e-16
1.907348631702277e-6 1e-16
9.536743210691867e-7 8.881784197001252e-16
4.768371577590358e-7 8.881784197001252e-16
2.384185788795179e-7 8.881784197001252e-16
1.1920928799646902e-7 8.881784197001252e-16
5.9604641666766156e-8 8.881784197001252e-16
2.9802319723160053e-8 8.881784197001252e-16
1.4901162526115286e-8 8.881784197001252e-16
7.4505797087454084e-9 8.881784197001252e-16
3.7252894102834944e-9 7.460698725481052e-14
1.862645593320167e-9 3.552713678800501e-15
9.313201321248243e-10 1.1546319456101628e-14
4.656588448170851e-10 8.881784197001252e-16
2.3282997752005485e-10 1.2434497875801753e-13
1.1641643204995944e-10 3.099742684753437e-13
5.820544046741816e-11 1.3001155707570433e-11
2.9104718635153404e-11 1.6564527527407336e-12
1.4552137272971777e-11 1.2997691811733603e-10
7.277067837208051e-12 2.984279490192421e-13
3.639311074721263e-12 1.3044683555563097e-8
1.815658734471981e-12 3.887793198487088e-9
9.097167463778533e-13 1.9279582375020254e-9
4.549693954913892e-13 1.7877441749192258e-9
2.2737367544323206e-13 1.2554517425655831e-9
1.1080025785759062e-13 2.9866455442117967e-7
5.906386491005833e-14 1.072299599069737e-5
2.731148640577885e-14 4.7347222054128224e-7
1.2212453270876722e-14 0.001637055353477912
9.547918011776346e-15 0.0009129640952592055
2.4424906541753444e-15 0.001607962374475136
2.6645352591003757e-15 0.00288903503933291
};
\addplot [semithick, blue, mark size=2, mark options={solid}, only marks]
table {%
0.9999999999999996 8.881784197001252e-16
0.49999999999999956 8.881784197001252e-16
0.24999999999999933 1e-16
0.12499999999999956 1e-16
0.062499999999999334 1.7763568394002505e-15
0.031249999999999334 1e-16
0.015625 8.881784197001252e-16
0.0078125 8.881784197001252e-16
0.003906249999999334 1.7763568394002505e-15
0.0019531249999980016 8.881784197001252e-16
0.0009765624999984457 1e-16
0.0004882812499995559 1e-16
0.00024414062500066613 8.881784197001252e-16
0.00012207031249844569 8.881784197001252e-16
6.103515625066613e-5 8.881784197001252e-16
3.0517578123445688e-5 1e-16
1.5258789061167732e-5 8.881784197001252e-16
7.629394530583866e-6 8.881784197001252e-16
3.8146972654029554e-6 1.7763568394002505e-15
1.9073486314802324e-6 1e-16
9.536743161842054e-7 1.7763568394002505e-15
4.768371573149466e-7 8.881784197001252e-16
2.3841857843542869e-7 1.7763568394002505e-15
1.1920928866260283e-7 1.7763568394002505e-15
5.960464477539063e-8 1e-16
2.980232260973992e-8 8.881784197001252e-16
1.4901160749758446e-8 8.881784197001252e-16
7.4505794867008035e-9 3.552713678800501e-15
3.725290298461914e-9 1e-16
1.862645593320167e-9 8.881784197001252e-16
9.31321020303244e-10 8.881784197001252e-16
4.656606211739245e-10 1.7763568394002505e-15
2.32829533430845e-10 8.881784197001252e-16
1.1641487773772496e-10 1e-16
5.820655069044278e-11 1e-16
2.910272023370908e-11 1e-16
1.4551693183761927e-11 8.881784197001252e-16
7.274625346553876e-12 8.881784197001252e-16
3.637534717881863e-12 8.881784197001252e-16
1.8174350913113813e-12 1e-16
9.08162434143378e-13 1e-16
4.556355293061642e-13 1e-16
2.262634524186069e-13 1e-16
1.13464793116691e-13 1e-16
5.639932965095795e-14 1.7763568394002505e-15
2.842170943040401e-14 8.881784197001252e-16
1.354472090042691e-14 1e-16
6.661338147750939e-15 1e-16
3.774758283725532e-15 1.7763568394002505e-15
1.3322676295501878e-15 8.881784197001252e-16
};
\legend{full MPS,block-sparse MPS}
\end{loglogaxis}
\end{tikzpicture}
\caption{Rayleigh quotient of ${\ten x}_{6,\epsilon}$ after rounding from rank $7$ to rank $6$. Plots the difference of the smallest two singular values over the error in the Rayleigh quotient.}
\end{figure}
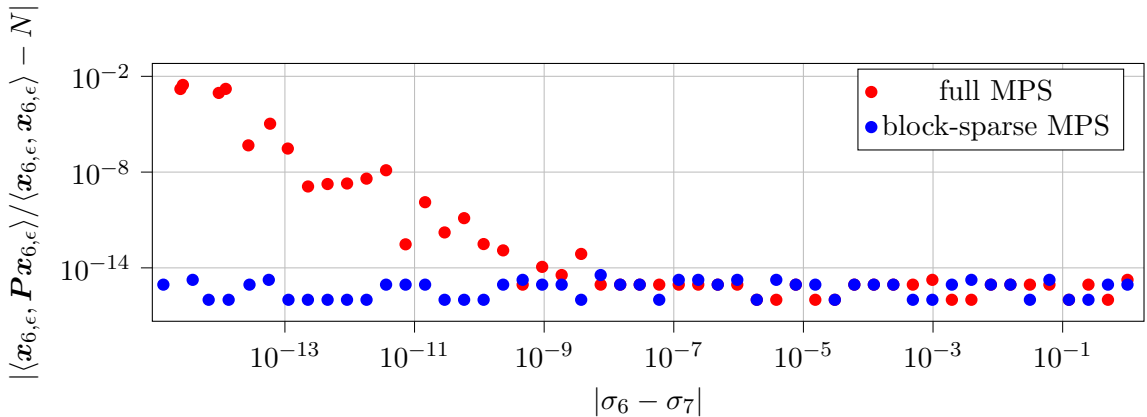

As mentioned in Remark~\ref{rem:ttsvd}, if the singular values in each matricization are distinct, the TT-SVD in Algorithm~\ref{alg:svdthresh} is unique up to the signs of singular vectors. Therefore, performing a TT-SVD on a tensor with fixed particle number will automatically result in a (reordered) block-sparse format. However, when two singular values coincide, there is an additional rotational freedom in the corresponding subspaces that needs to be factored out in order to enforce block sparsity.

This fact has an important implication on the SVD truncation of a tensor of fixed particle number that is {\em not} in block-sparse format. If the TT-SVD is unique, the SVD truncation will result in a tensor of the same fixed particle number (or the zero tensor, which is still an element of the same eigenspace). If two or more singular values are equal, a general SVD truncation can destroy the natural block structure, as it could remove parts of two blocks simultaneously. Numerically, this already occurs when two singular values are close to each other, since singular vectors become increasingly ill-conditioned with decreasing difference of the corresponding singular values, resulting in numerical errors in the particle number upon truncation. We emphasize that if the block-sparse structure is enforced, this is ruled out.

To illustrate this numerical issue, we conduct the following artificial experiment: Let $K=20$ and $N=6$. We pick a tensor with blocks of size $1$, left-orthogonal components $U_1,\ldots,U_{10}$, and right-orthogonal components $V_{11},\ldots,V_{20}$. This tensor has rank at most $7$ because there can be no more than $7$ blocks of size $1$ on the unoccupied or occupied layer, respectively. For $\epsilon \geq 0$, we choose a diagonal matrix of singular values
\[
\Sigma = \diag(
\sigma_1, \sigma_2, \ldots, \sigma_7) = 
\diag(6,5,4,3,2,1,1 - \epsilon)
\]
and construct the tensor
\[
\ten x_\epsilon = U_1 \SKP \dots \SKP U_{10} \Sigma \SKP V_{11} \SKP \ldots \SKP V_{20},
\]
which has the singular values $\sigma_1,\ldots,\sigma_7$ between its middle components $U_{10}$ and $V_{11}$. When $\epsilon = 0$, truncating the last singular value in full MPS format will therefore in general lead to a deviation from the particle number $N=6$. This effect can also be observed when the smallest singular values are only roughly equal. We consider the Rayleigh quotient of $\ten x_{6,\epsilon} = \operatorname{trunc}_{6,\ldots,6}(\ten x_\epsilon)$ with the particle number operator for $\epsilon \rightarrow 0$, which directly translates to the difference of the smallest singular values $| \sigma_6 - \sigma_7 | = \epsilon$. This is shown in Fig.~\ref{fig:svdrounding}.

Ideally, this Rayleigh quotient should be constant $N=6$. This is the case for $\epsilon > 10^{-8}$. However, as we can see if we keep the MPS in its full format, for small differences in the singular values, the Rayleigh quotient deviates and the natural block structure is destroyed. This is due to the ill-conditioning of singular vectors in the TT-SVD when the smallest singular values of $\ten x_\epsilon$ are close. If we keep the tensor in block-sparse format, as expected, this problem cannot occur.

\subsection{Operator Ranks}
\label{ssc:opranktests}

\begin{table}[t]\label{tab:opranks}
\centering
\begin{tabu}{c|c|c}
\toprule
$K$ & {\bf rank reduced representation} & {\bf extra compression} \\
\midrule
\multicolumn{2}{l}{\bf Ranks of one-particle operator} \\
\midrule
8 & $(4,6,8,10,8,6,4)$ & $(4,6,8,10,8,6,4)$ \\
16 & $(4,6,8,\ldots,14,16,18,16,14,\ldots,8,6,4)$ & $(4,6,8,\ldots,14,16,18,16,14,\ldots,8,6,4)$\\
24 & $(4,6,8,\ldots,22,24,26,24,22,\ldots,8,6,4)$ & $(4,6,8,\ldots,22,24,26,24,22,\ldots,8,6,4)$\\
32 & $(4,6,8,\ldots,30,32,34,32,30,\ldots,8,6,4)$ & $(4,6,8,\ldots,30,32,34,32,30,\ldots,8,6,4)$\\
\midrule
\multicolumn{2}{l}{\bf Ranks of two-particle operator} \\
\midrule
8 & $(4,24,33,46,33,24,4)$ & $(4,\mathbf{16},33,46,33,\mathbf{16},4)$ \\
16 & $(4, 40, 49, \ldots, 125, 154, 125, \ldots, 49, 40, 4)$ & $(4, \mathbf{16}, \mathbf{41}, \ldots, 125, 154, 125, \ldots, \mathbf{41}, \mathbf{16}, 4)$ \\
24 & $(4,56,65,\ldots,281,326,281,\ldots,65,56,4)$ & $(4,\mathbf{16},\mathbf{41},\ldots,281,326,281,\ldots,\mathbf{41},\mathbf{16},4)$ \\
32 & $(4,72,81,94,\ldots,562,\ldots,94,81,72,4)$ & $(4,\mathbf{16},\mathbf{41},\mathbf{86},\ldots,562,\ldots,\mathbf{86},\mathbf{41},\mathbf{16},4)$\\
\bottomrule 
\end{tabu}
\vspace{4pt}
\caption{Ranks of the one- and two-particle operators for different numbers of orbitals $K$. The left column shows the ranks of the rank reduced representation, the right column shows the ranks after an extra rank compression sweep back and forth.}
\end{table}

We have discussed in Sec.~\ref{ssc:compactform} that the one- and two-particle operators $\ten S$ and $\ten D$ can be explicitly stated in rank-compressed format. Here, we numerically show that these representations indeed have near-optimal ranks, in the sense that the rank compression procedure for operators that is outlined in Sec.~\ref{ssc:matfree} can reduce the ranks only for the border cases at the beginning and end of the MPO chain.

Table~\ref{tab:opranks} shows the ranks of the two operators for different numbers of orbitals $K$ and normally distributed random values for $T$ and $V$. In the left column, we see the ranks of these operators when they are represented in the rank-reduced form discussed in Sec.~\ref{ssc:compactform}. In the right column, we have performed an extra rank compression from left to right and one from right to left, as discussed in Sec.~\ref{ssc:matfree}. One can see that the ranks in the left column are almost optimal. Only for the border cases of the two-particle operator, the ranks can be reduced further.

Finally, for $K=32$, we apply these operators $\ten S$ and $\ten D$ to a random MPS of rank 1 that is not in the block-sparse format. The entries in the components of this tensor are chosen to be $\mathcal N(0,1)$-distributed and the tensor is then normalized. The operators are in the rank-reduced format but they are applied explicitly and not as their matrix-free versions as in Sec.~\ref{ssc:matfree}, since this is applicable only to tensors in block-sparse format.

Fig.~\ref{fig:solranks} shows the ranks of the output tensor for the two operators (top and bottom) and different truncation parameters $\varepsilon$. One can see that with minimal truncation, the ranks of the output are about the same as the ranks of the operators. However, the output ranks can be further reduced by about a factor $2$ if one is willing to accept an error threshold of $\varepsilon = 10^{-12}$. A more substantial truncation does not reduce the ranks further, indicating that the ranks of the two operators are indeed linear and quadradic in $K$, respectively.

\begin{figure}[t]\label{fig:solranks}
\centering
\begin{tikzpicture}
\begin{axis}[
footnotesize,
ybar={.5pt},
bar width = 1pt,
tick align=outside,
tick pos=left,
grid=both,
xlabel={$k$},
xmin=0, xmax=32,
ymin=0, ymax=40,
xtick style={color=black},
xtick={1,11,21,31},
xlabel near ticks,
ylabel near ticks,
ytick={0,10,20,30,40},
ylabel={$r_k$},
ytick style={color=black},
width=\textwidth,
height=4.7cm,
legend columns=2, 
legend style={/tikz/column 2/.style={column sep=5pt,},},
]
\addplot [draw=red,fill=red]
table {%
  1   2
  2   4
  3   8
  4  10
  5  12
  6  14
  7  16
  8  17
  9  20
 10  22
 11  24
 12  26
 13  28
 14  30
 15  32
 16  34
 17  32
 18  30
 19  28
 20  26
 21  23
 22  22
 23  20
 24  18
 25  16
 26  14
 27  12
 28   9
 29   7
 30   4
 31   2
};
\addplot [draw=red!80!blue,fill=red!80!blue]
table {%
  1   2
  2   4
  3   5
  4   7
  5   9
  6  10
  7  11
  8  12
  9  13
 10  15
 11  19
 12  26
 13  28
 14  23
 15  19
 16  18
 17  19
 18  20
 19  23
 20  26
 21  16
 22  16
 23  14
 24  12
 25  11
 26   8
 27   7
 28   6
 29   5
 30   4
 31   2
};
\addplot [draw=red!60!blue,fill=red!60!blue]
table {%
  1   2
  2   4
  3   5
  4   6
  5   7
  6   8
  7   9
  8  10
  9  11
 10  12
 11  13
 12  14
 13  15
 14  16
 15  17
 16  18
 17  17
 18  16
 19  15
 20  14
 21  13
 22  12
 23  11
 24  10
 25   9
 26   8
 27   7
 28   6
 29   5
 30   4
 31   2
};
\addplot [draw=red!40!blue,fill=red!40!blue]
table {%
  1   2
  2   4
  3   5
  4   6
  5   7
  6   8
  7   9
  8  10
  9  11
 10  12
 11  13
 12  14
 13  15
 14  16
 15  17
 16  18
 17  17
 18  16
 19  15
 20  14
 21  13
 22  12
 23  11
 24  10
 25   9
 26   8
 27   7
 28   6
 29   5
 30   4
 31   2
};
\addplot [draw=red!20!blue,fill=red!20!blue]
table {%
  1   2
  2   4
  3   5
  4   6
  5   7
  6   8
  7   9
  8  10
  9  11
 10  12
 11  13
 12  14
 13  15
 14  16
 15  17
 16  18
 17  17
 18  16
 19  15
 20  14
 21  13
 22  12
 23  11
 24  10
 25   9
 26   8
 27   7
 28   6
 29   5
 30   4
 31   2
};
\addplot [draw=blue,fill=blue]
table {%
  1   2
  2   4
  3   5
  4   6
  5   7
  6   8
  7   9
  8  10
  9  11
 10  12
 11  13
 12  14
 13  15
 14  16
 15  17
 16  18
 17  17
 18  16
 19  15
 20  14
 21  13
 22  12
 23  11
 24  10
 25   9
 26   8
 27   7
 28   6
 29   5
 30   4
 31   2
};
\legend{$\varepsilon=10^{-16}$,$\varepsilon=10^{-14}$,$\varepsilon=10^{-12}$,$\varepsilon=10^{-10}$,$\varepsilon=10^{-8}$,$\varepsilon=10^{-6}$}
\end{axis}
\end{tikzpicture}

\begin{tikzpicture}
\begin{axis}[
footnotesize,
ybar={.5pt},
bar width = 1pt,
tick align=outside,
tick pos=left,
grid=both,
xlabel={$k$},
xmin=0, xmax=32,
ymin=0, ymax=600,
xtick style={color=black},
xtick={1,11,21,31},
xlabel near ticks,
ylabel near ticks,
ytick={0,200,400,600},
ylabel={$r_k$},
ytick style={color=black},
width=\textwidth,
height=4.7cm,
legend columns=2, 
legend style={/tikz/column 2/.style={column sep=5pt,},},
]
\addplot [draw=red,fill=red]
table {%
  1    2
  2    4
  3    8
  4   16
  5   32
  6   64
  7  128
  8  186
  9  219
 10  256
 11  297
 12  342
 13  391
 14  444
 15  501
 16  562
 17  501
 18  444
 19  391
 20  342
 21  297
 22  256
 23  219
 24  186
 25  128
 26   64
 27   32
 28   16
 29    8
 30    4
 31    2
};
\addplot [draw=red!80!blue,fill=red!80!blue]
table {%
  1    2
  2    4
  3    8
  4   16
  5   32
  6   64
  7  128
  8  186
  9  219
 10  255
 11  296
 12  341
 13  389
 14  443
 15  500
 16  561
 17  498
 18  438
 19  391
 20  342
 21  297
 22  256
 23  219
 24  186
 25  128
 26   64
 27   32
 28   16
 29    8
 30    4
 31    2
};
\addplot [draw=red!60!blue,fill=red!60!blue]
table {%
  1    2
  2    4
  3    8
  4   16
  5   31
  6   53
  7   70
  8   83
  9   98
 10  115
 11  134
 12  155
 13  178
 14  203
 15  230
 16  259
 17  230
 18  203
 19  178
 20  155
 21  134
 22  115
 23   98
 24   83
 25   70
 26   53
 27   31
 28   16
 29    8
 30    4
 31    2
};
\addplot [draw=red!40!blue,fill=red!40!blue]
table {%
  1    2
  2    4
  3    8
  4   16
  5   31
  6   53
  7   70
  8   83
  9   98
 10  115
 11  134
 12  155
 13  178
 14  203
 15  230
 16  259
 17  230
 18  203
 19  178
 20  155
 21  134
 22  115
 23   98
 24   83
 25   70
 26   53
 27   31
 28   16
 29    8
 30    4
 31    2
};
\addplot [draw=red!20!blue,fill=red!20!blue]
table {%
  1    2
  2    4
  3    8
  4   16
  5   31
  6   53
  7   70
  8   83
  9   98
 10  115
 11  134
 12  155
 13  178
 14  203
 15  230
 16  259
 17  230
 18  203
 19  178
 20  155
 21  134
 22  115
 23   98
 24   83
 25   70
 26   53
 27   31
 28   16
 29    8
 30    4
 31    2
};
\addplot [draw=blue,fill=blue]
table {%
  1    2
  2    4
  3    8
  4   16
  5   31
  6   53
  7   70
  8   83
  9   98
 10  115
 11  134
 12  155
 13  178
 14  203
 15  230
 16  259
 17  230
 18  203
 19  178
 20  155
 21  134
 22  115
 23   98
 24   83
 25   70
 26   53
 27   31
 28   16
 29    8
 30    4
 31    2
};
\legend{$\varepsilon=10^{-16}$,$\varepsilon=10^{-14}$,$\varepsilon=10^{-12}$,$\varepsilon=10^{-10}$,$\varepsilon=10^{-8}$,$\varepsilon=10^{-6}$}
\end{axis}
\end{tikzpicture}
\caption{Ranks of the output after the one-particle operator (top) and the two-particle operator (bottom) have been applied to a random MPS of rank 1 in full format, $K = 32$. The ranks are shown after a TT-SVD trunctation with various values for $\varepsilon$. The gradient from red to blue indicates a more substantial truncation, resulting in lower ranks.}
\end{figure}
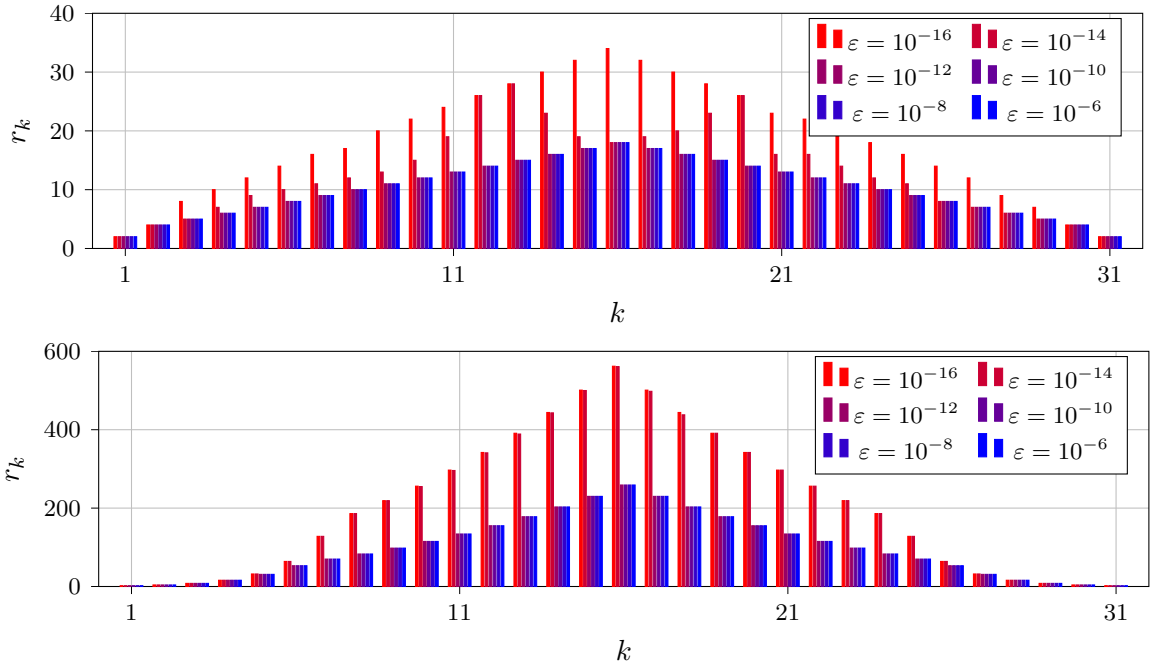

\bibliographystyle{amsplain}
\bibliography{BGPblockmps_main}

\providecommand{\bysame}{\leavevmode\hbox to3em{\hrulefill}\thinspace}
\providecommand{\MR}{\relax\ifhmode\unskip\space\fi MR }
\providecommand{\MRhref}[2]{%
  \href{http://www.ams.org/mathscinet-getitem?mr=#1}{#2}
}
\providecommand{\href}[2]{#2}
\begin{thebibliography}{10}

\bibitem{Absil2008}
Pierre-Antoine Absil, Robert Mahony, and Rodolphe Sepulchre, \emph{Optimization
  algorithms on matrix manifolds}, Princeton University Press, Princeton, NJ,
  2008.

\bibitem{BCD}
Markus Bachmayr, Albert Cohen, and Wolfgang Dahmen, \emph{Parametric {PDE}s:
  {S}parse or low-rank approximations?}, IMA Journal of Numerical Analysis
  \textbf{38} (2018), 1661--1708.

\bibitem{BK:20}
Markus Bachmayr and Vladimir Kazeev, \emph{Stability of low-rank tensor
  representations and structured multilevel preconditioning for elliptic
  {PDE}s}, Foundations of Computational Mathematics \textbf{20} (2020),
  1175--1236.

\bibitem{BSU:16}
Markus Bachmayr, Reinhold Schneider, and Andr{\'e} Uschmajew, \emph{Tensor
  networks and hierarchical tensors for the solution of high-dimensional
  partial differential equations}, Foundations of Computational Mathematics
  \textbf{16} (2016), no.~6, 1423--1472.

\bibitem{bauer2011implementing}
Bela Bauer, Philippe Corboz, Rom{\'a}n Or{\'u}s, and Matthias Troyer,
  \emph{Implementing global {A}belian symmetries in projected entangled-pair
  state algorithms}, Physical Review B \textbf{83} (2011), no.~12, 125106.

\bibitem{chan_matrix_2016}
Garnet Kin-Lic Chan, Anna Keselman, Naoki Nakatani, Zhendong Li, and Steven~R.
  White, \emph{Matrix product operators, matrix product states, and ab initio
  density matrix renormalization group algorithms}, The Journal of Chemical
  Physics \textbf{145} (2016), no.~1, 014102 (en).

\bibitem{crosswhite2008finite}
Gregory~M Crosswhite and Dave Bacon, \emph{Finite automata for caching in
  matrix product algorithms}, Physical Review A \textbf{78} (2008), no.~1,
  012356.

\bibitem{Daley:04}
Andrew~John Daley, Corinna Kollath, Ulrich Schollw{\"o}ck, and Guifr{\'e}
  Vidal, \emph{Time-dependent density-matrix renormalization-group using
  adaptive effective {H}ilbert spaces}, Journal of Statistical Mechanics:
  Theory and Experiment \textbf{2004} (2004), no.~04, P04005.

\bibitem{Dolfietal:12}
Michele Dolfi, Bela Bauer, Matthias Troyer, and Zoran Ristivojevic,
  \emph{Multigrid algorithms for tensor network states}, Physical review
  letters \textbf{109} (2012), no.~2, 020604.

\bibitem{dolgov_tensor_nodate}
Sergey Dolgov, Dante Kalise, and Karl Kunisch, \emph{Tensor decompositions for
  high-dimensional {H}amilton-{J}acobi-{B}ellman equations}, 24.

\bibitem{dolgov2013two}
Sergey Dolgov and Boris Khoromskij, \emph{Two-level {QTT}-{T}ucker format for
  optimized tensor calculus}, SIAM Journal on Matrix Analysis and Applications
  \textbf{34} (2013), no.~2, 593--623.

\bibitem{EPS}
Martin Eigel, Max Pfeffer, and Reinhold Schneider, \emph{Adaptive stochastic
  {G}alerkin {FEM} with hierarchical tensor representations}, Numer. Math.
  \textbf{136} (2017), no.~3, 765--803.

\bibitem{itensor}
Matthew Fishman, Steven~R. White, and E.~Miles Stoudenmire, \emph{The
  \mbox{ITensor} software library for tensor network calculations},
  arXiv:2007.14822, 2020.

\bibitem{Grasedyck:2010:HierarchicalSVD}
Lars Grasedyck, \emph{Hierarchical singular value decomposition of tensors},
  SIAM Journal on Matrix Analysis and Applications \textbf{31} (2010), no.~4,
  2029--2054.

\bibitem{Hackbusch:18}
Wolfgang Hackbusch, \emph{On the representation of symmetric and antisymmetric
  tensors}, Contemporary Computational Mathematics-A Celebration of the 80th
  Birthday of Ian Sloan, Springer, 2018, pp.~483--515.

\bibitem{Hackbusch:09}
Wolfgang Hackbusch and Stefan K{\"u}hn, \emph{A new scheme for the tensor
  representation}, Journal of Fourier Analysis and Applications \textbf{15}
  (2009), no.~5, 706--722.

\bibitem{tenpy}
Johannes Hauschild and Frank Pollmann, \emph{Efficient numerical simulations
  with tensor networks: Tensor network python (tenpy)}, SciPost Physics Lecture
  Notes (2018).

\bibitem{Helgaker:00}
Trygve Helgaker, Poul Jorgensen, and Jeppe Olsen, \emph{Molecular
  electronic-structure theory}, John Wiley \& Sons, 2000.

\bibitem{Holtz:12}
Sebastian Holtz, Thorsten Rohwedder, and Reinhold Schneider, \emph{The
  alternating linear scheme for tensor optimization in the tensor train
  format}, SIAM Journal on Scientific Computing \textbf{34} (2012), no.~2,
  A683--A713.

\bibitem{holtz_manifolds_2012}
\bysame, \emph{On manifolds of tensors of fixed {TT}-rank}, Numerische
  Mathematik \textbf{120} (2012), no.~4, 701--731.

\bibitem{kazeev2013low}
Vladimir Kazeev, Oleg Reichmann, and Christoph Schwab, \emph{Low-rank tensor
  structure of linear diffusion operators in the {TT} and {QTT} formats},
  Linear Algebra and its Applications \textbf{438} (2013), no.~11, 4204--4221.

\bibitem{kazeev_low-rank_2012}
Vladimir~A. Kazeev and Boris~N. Khoromskij, \emph{Low-{Rank} {Explicit} {QTT}
  {Representation} of the {Laplace} {Operator} and {Its} {Inverse}}, SIAM J.
  Matrix Anal. \& Appl. \textbf{33} (2012), no.~3, 742--758.

\bibitem{keller2015efficient}
Sebastian Keller, Michele Dolfi, Matthias Troyer, and Markus Reiher, \emph{An
  efficient matrix product operator representation of the quantum chemical
  hamiltonian}, The Journal of chemical physics \textbf{143} (2015), no.~24,
  244118.

\bibitem{Kressner_SV_2013}
Daniel Kressner, Michael Steinlechner, and Bart Vandereycken, \emph{Low-rank
  tensor completion by {R}iemannian optimization}, BIT Numer. Math. \textbf{54}
  (2014), no.~2, 447--468.

\bibitem{Kressner:11}
Daniel Kressner and Christine Tobler, \emph{Preconditioned low-rank methods for
  high-dimensional elliptic pde eigenvalue problems}, Computational Methods in
  Applied Mathematics \textbf{11} (2011), no.~3, 363--381.

\bibitem{McCulloch:07}
Ian~P. McCulloch, \emph{From density-matrix renormalization group to matrix
  product states}, Journal of Statistical Mechanics: Theory and Experiment
  \textbf{2007} (2007), no.~10, P10014.

\bibitem{pytenet}
Christian~B. Mendl, \emph{Pytenet: A concise python implementation of quantum
  tensor network algorithms}, Journal of Open Source Software \textbf{3}
  (2018), no.~30, 948.

\bibitem{Oseledets:2011:TT}
Ivan~V. Oseledets, \emph{{T}ensor {T}rain decomposition}, SIAM Journal on
  Scientific Computing \textbf{33} (2011), no.~5, 2295--2317.

\bibitem{oster_approximating_2020}
Mathias Oster, Leon Sallandt, and Reinhold Schneider, \emph{Approximating the
  stationary {Hamilton}-{Jacobi}-{Bellman} equation by hierarchical tensor
  products}, arXiv:1911.00279 [math] (2020), arXiv: 1911.00279.

\bibitem{OR95}
Stellan {\"O}stlund and Stefan Rommer, \emph{Thermodynamic limit of density
  matrix renormalization}, Physical review letters \textbf{75} (1995), no.~19,
  3537.

\bibitem{tensornetwork}
Chase Roberts, Ashley Milsted, Martin Ganahl, Adam Zalcman, Bruce Fontaine,
  Yijian Zou, Jack Hidary, Guifre Vidal, and Stefan Leichenauer,
  \emph{Tensornetwork: A library for physics and machine learning},
  arXiv:1905.01330, 2019.

\bibitem{Rohwedder:11}
Thorsten Rohwedder, Reinhold Schneider, and Andreas Zeiser, \emph{Perturbed
  preconditioned inverse iteration for operator eigenvalue problems with
  applications to adaptive wavelet discretization}, Advances in Computational
  Mathematics \textbf{34} (2011), no.~1, 43--66.

\bibitem{Schollwoeck:11}
Ulrich Schollw{\"o}ck, \emph{The density-matrix renormalization group in the
  age of matrix product states}, Annals of physics \textbf{326} (2011), no.~1,
  96--192.

\bibitem{SPV11}
Sukhwinder Singh, Robert N.~C. Pfeifer, and Guifre Vidal, \emph{Tensor network
  states and algorithms in the presence of a global {$U(1)$} symmetry}, Phys.
  Rev. B \textbf{83} (2011), 115125.

\bibitem{steinlechner_riemannian_2016}
Michael Steinlechner, \emph{Riemannian optimization for high-dimensional tensor
  completion}, SIAM Journal on Scientific Computing \textbf{38} (2016), no.~5,
  S461--S484.

\bibitem{Szalay2015}
Szil{\'a}rd Szalay, Max Pfeffer, Valentin Murg, Gergely Barcza, Frank
  Verstraete, Reinhold Schneider, and {\"O}rs Legeza, \emph{Tensor product
  methods and entanglement optimization for ab initio quantum chemistry},
  International Journal of Quantum Chemistry \textbf{115} (2015), no.~19,
  1342--1391.

\bibitem{verstraete_renormalization_2004}
Frank Verstraete and Juan~Ignacio Cirac, \emph{Renormalization algorithms for
  quantum-many body systems in two and higher dimensions}, arXiv:
  cond-mat/0407066, 2004.

\bibitem{vidal03}
Guifr{\'e} Vidal, \emph{Efficient classical simulation of slightly entangled
  quantum computations}, Physical review letters \textbf{91} (2003), no.~14,
  147902.

\bibitem{VidalMERA}
\bysame, \emph{Entanglement renormalization}, Phys. Rev. Lett. \textbf{99}
  (2007), 220405.

\bibitem{white}
Steven~R. White, \emph{Density matrix formulation for quantum renormalization
  groups}, Phys. Rev. Lett. \textbf{69} (1992), 2863--2866.

\end{thebibliography}

\appendix

\section{Proof of Lemma \ref{lmm:oprankoneterms}}\label{app:oprankoneterms}

\begin{proof}
	We use induction over the eigenvalues of $\ten{P}$, which are $\{0,1,\ldots,K\}$. For the corresponding eigenspace for eigenvalue $N$, we write $U_N$. Let 
	\[
	\ten{B}_N =  \sum_{\substack{ D^+,D^- \subseteq \{1,\dots,K\} \\  \# D^+  = \# D^-  \leq N} } v_{D^+,D^-} \ten{a}_{D^+}^*\ten{a}_{D^-}.
	\]
	We aim to show that $\ten{B}_K = \ten{B}$ for appropriate coefficients $v_{D^+,D^-}$. For $\ten{B}_0$, we set $v_{\emptyset,\emptyset} = \langle \ten{e}_{\emptyset}, \ten{B} \ten{e}_{\emptyset} \rangle = \langle \vac, \ten{B} \vac \rangle$. Thus, $\ten{B}_0$ and $\ten{B}$ coincide on $U_0$. 
	
	Now assume we have $\ten{B}_N$ such that $\ten{B}_N$ and $\ten{B}$ coincide on $U_L$ with $L\leq N$. Then for ${\# D^+ = \# D^- =N+1}$ we set 
	\[
	v_{D^+,D^-} = \langle \ten{e}_{D^+}, \ten{a}_{D^+}^*\ten{a}_{D^-}\ten{e}_{D^-} \rangle \, \langle \ten{e}_{D^+},(\ten{B} - \ten{B}_N)\ten{e}_{D^-} \rangle,
	\]
	where $\ten{e}_{D^+}^T\ten{a}_{D^+}^*\ten{a}_{D^-}\ten{e}_{D^-}$, in view of~\eqref{def:anniloperator}, is $\pm 1$.
	We now show that $\ten{B}_{N+1}$ and $\ten{B}$ coincide on $U_L$ for $L\leq N+1$. Since for all $E^+, E^- \subset \{ 1,\ldots, K\}$ with ${\# E^+ = \# E^- \leq N}$ we have $\langle \ten{e}_{E^+},\ten{a}_{D^+}^*\ten{a}_{D^-}\ten{e}_{E^-}\rangle=0$, by the induction hypothesis, we obtain
	\[
	\langle \ten{e}_{E^+}, \ten{B}_{N+1}\ten{e}_{E^-} \rangle = \langle \ten{e}_{E^+}, \ten{B}_N\ten{e}_{E^-} \rangle = \langle \ten{e}_{E^+},\ten{B}\ten{e}_{E^-} \rangle.
	\]
Furthermore, we also have $\langle \ten{e}_{E^+},\ten{a}_{D^+}^*\ten{a}_{D^-}\ten{e}_{E^-}\rangle=0$ for ${\# E^+ =\# E^- = N+1}$ with  $E^+\neq D^+$ or $E^-\neq D^-$. Consequently, 
\begin{align*}
\langle \ten{e}_{D^+},\ten{B}_{N+1}\ten{e}_{D^-} \rangle &= \langle \ten{e}_{D^+}, \ten{B}_{N}\ten{e}_{D^-} \rangle + v_{D^+,D^-} \langle \ten{e}_{D^+},\ten{a}_{D^+}^*\ten{a}_{D^-}\ten{e}_{D^-} \rangle \\
&= \langle \ten{e}_{D^+}, \ten{B}_{N}\ten{e}_{D^-} \rangle + \langle \ten{e}_{D^+}, \ten{a}_{D^+}^*\ten{a}_{D^-}\ten{e}_{D^-}\rangle^2 \, \langle \ten{e}_{D^+}, (\ten{B} - \ten{B}_N) \ten{e}_{D^-} \rangle \\
&= \langle \ten{e}_{D^+}, \ten{B}\ten{e}_{D^-} \rangle. \qedhere
\end{align*}
\end{proof}

\section{Proof of Thm.~\ref{thm:oneelecasMPO}}\label{app:oneelecproof}

\begin{proof}
For $k=1,\ldots, K$, we have the two cases $k\leq \frac{K}{2}$ and $ k > \frac{K}{2}$. We begin with $k\leq \frac{K}{2}$ and define the first $k$ factors of $\ten a_i^*$ and $\ten a_j$
by
\begin{equation*}
\ten a_{k,i}^* = \biggl( \bigotimes_{\ell=1}^{i-1} S \biggr) \otimes A^* \otimes \biggl( \bigotimes_{\ell=i+1}^{k} I \biggr), \quad \ten a_{k,j} = \biggl( \bigotimes_{\ell=1}^{j-1} S \biggr) \otimes A \otimes \biggl( \bigotimes_{\ell=j+1}^{k} I \biggr).
\end{equation*}
Clearly, if $\max\{i, j\} > k$, the actual values of $i,j$ are irrelevant. Thus, we write
\begin{equation*}
\ten a_{k,>}^* = \ten a_{k,k+1}^* = \dots = \ten a_{k,K}^*, \quad \ten a_{k,>} = \ten a_{k,k+1} = \dots = \ten a_{k,K}
\end{equation*}
and in particular
\begin{equation*}
\ten{I}_k = \ten a_{k,>}^* \ten a_{k,>} = \bigotimes_{\ell=1}^k I.
\end{equation*}
We further define
\begin{equation*}
\rep{a}_k^* = 
\begin{bmatrix} 
\ten a_{k,i}^* \ten a_{k,>}
\end{bmatrix}^{i=1,\ldots,k}, \quad
\rep{a}_k = 
\begin{bmatrix} 
\ten a_{k,>}^* \ten a_{k,j}
\end{bmatrix}^{j=1,\ldots,k}, \quad
\ten{S}_k = \sum_{i,j=1}^k t_{ij} \ten{a}_{k,i}^* \ten{a}_{k,j}.
\end{equation*}
With this, for $k \leq \frac{K}{2}$, we can show by induction that 
\[
T_1\SKP T_2\SKP \cdots \SKP T_{k} = 
\begin{bmatrix} 
\ten{I}_k &\rep{a}_k & \rep{a}_k^* & \ten{S}_k
\end{bmatrix}.
\]
This holds true by definition for $k=1$. For $k>1$ we calculate
\begin{align*}
&\begin{bmatrix} \ten{I}_{k-1}
&\rep{a}_{k-1}& \rep{a}_{k-1}^* & \ten{S}_{k-1}
\end{bmatrix} \SKP \begin{bmatrix}
I & 0  &A & 0 &A^*& t_{k,k}A^*A\\
0 & \rep{S}_{k-1} & 0 &  0 & 0 & W_{T,k}^1 \uparrow A^*\\
0 & 0 & 0 & \rep{S}_{k-1} & 0 & W_{T,k}^2 \uparrow A \\
0 & 0 & 0 & 0 & 0 & I \\
\end{bmatrix} \\
&= \begin{bmatrix} \ten{I}_{k}
& \rep{a}_{k}& \rep{a}_{k}^* & \left(\ten{I}_{k-1} \otimes t_{k,k}A^*A + \sum_{i=1}^{k-1} t_{i,k} \ten{a}_{k,i}^* \ten{a}_{k,k} + \sum_{i=1}^{k-1} t_{k,i} \ten{a}_{k,k}^* \ten{a}_{k,i} + \ten{S}_{k-1} \otimes I\right)
\end{bmatrix}\\
&=\begin{bmatrix} \ten{I}_k
&\rep{a}_k & \rep{a}_k^* & \ten{S}_k
\end{bmatrix}.
\end{align*}
Similarly, it can be shown for $k > \frac{K}{2}$ that 
\[
T_{k} \SKP \cdots \SKP T_K = \begin{bmatrix}
\tilde{\ten{I}}_{k} \\ \tilde{ \rep{a}}_k \\ \tilde{\rep{a}}_k^*\\ \tilde{\ten{S}}_k
\end{bmatrix},
\]
where
\begin{equation*}
\tilde{\ten{I}}_k = \bigotimes_{\ell=k}^K I, \quad
\tilde{\rep{a}}_k^* = \begin{bmatrix}
\tilde{\ten a}_{k,i}^* \tilde{\ten a}_{k,<}
\end{bmatrix}_{i=k,\ldots,K}, \quad
\tilde{\rep{a}}_k = \begin{bmatrix}
\tilde{\ten a}_{k,<}^* \tilde{\ten a}_{k,j}
\end{bmatrix}_{j=k,\ldots,K}, \quad
\tilde{\ten{S}}_k = \sum_{i,j=k}^K t_{ij} \tilde{\ten{a}}_{k,i}^* \tilde{\ten{a}}_{k,j}
\end{equation*}
and $\tilde{\ten a}_{k,i}^*, \tilde{\ten a}_{k,j}, \tilde{\ten a}_{k,<}^*$ and $\tilde{\ten a}_{k,<}$ are defined analogously to the above for the last $k$ factors of $\ten a_i^*$ and $\ten a_j$ respectively. With this we calculate
\begin{align*}
&\begin{bmatrix} 
\ten{I}_{K/2} & \rep{a}_{K/2} & \rep{a}_{K/2}^* & \ten{S}_{K/2}
\end{bmatrix} M_T \SKP \begin{bmatrix}
\ten{I}_{K/2+1} \\ \tilde{\rep{a}}_{K/2+1} \\ \tilde{\rep{a}}_{K/2+1}^* \\ \tilde{\ten{S}}_{K/2+1}
\end{bmatrix} \\
&= \sum_{i,j=1}^{K/2} t_{ij} \ten{a}_i^* \ten{a}_j + \rep{a}_{K/2}^*W_{T}^5 \tilde{\rep{a}}_{K/2+1} + \rep{a}_{K/2} W_{T}^6 \tilde{\rep{a}}_{K/2+1}^* + \sum_{i,j=K/2+1}^K t_{ij} \tilde{\ten{a}}_i^* \tilde{\ten{a}}_j
= \ten{S},
\end{align*}
because 
\begin{align*}
\rep{a}_{K/2}^*W_{T}^5\tilde{\rep{a}}_{K/2+1} &= \begin{bmatrix}
\displaystyle \sum_{i=1}^{K/2} t_{i,K} \ten a_{k,i}^* \ten a_{k,>} & \cdots & \displaystyle \sum_{i=1}^{K/2}t_{i,K/2+1} \ten a_{k,i}^* \ten a_{k,>} \end{bmatrix} \tilde{\rep{a}}_{K/2+1} = \sum_{i=1}^{K/2} \sum_{j=K/2+1}^K t_{ij} \ten{a}_i^* \ten{a}_j, \\
\rep{a}_{K/2}W_{T}^6\tilde{\rep{a}}_{K/2+1}^* &= \begin{bmatrix}
\displaystyle \sum_{j=1}^{K/2} t_{K,j} \ten a_{k,>}^* \ten a_{k,j} & \cdots & \displaystyle \sum_{j=1}^{K/2} t_{K/2+1,j} \ten a_{k,>}^* \ten a_{k,j} \end{bmatrix} \tilde{\rep{a}}_{K/2+1}^* = \sum_{i=K/2+1}^K \sum_{j=1}^{K/2}  t_{ij} \ten{a}_i^* \ten{a}_j.
\end{align*}
The rank of $T_k$ for $k \leq \frac{K}{2}$ can be bounded by $2+2k$. Consequently, for $k=\frac{K}{2}$ we find that the rank of $\ten{S}$ can be bounded by $K+2$. This is also the bound for the rank of the matrix $M_T$. For the sparse coefficient matrix we directly consider $M_T$, since the rank is maximized at the center of the representation, where the rank of $W_{T}^5$ and the rank of $W_{T}^6$ can be bounded by $d$ in both cases. Thus the rank of $M_T$ is bounded by $2d+2$.
\end{proof}

\section{Proof of Thm.~\ref{thm:twoelecasMPO}}\label{app:twoelecproof}

\begin{proof}
We use the same notation as in the proof of Theorem~\ref{thm:oneelecasMPO}.
With this, we define
\begin{align*}
\ten{I}_k &= \ten a_{k,>}^* \ten a_{k,>}^* \ten a_{k,>} \ten a_{k,>} = \bigotimes_{\ell=1}^k I, \\
\rep{a}_k^* &= \begin{bmatrix} \ten a_{k,i_1}^* \ten a_{k,>}^* \ten a_{k,>} \ten a_{k,>} \end{bmatrix}^{i_1=1,\ldots,k}, \\
\rep{a}_k &= \begin{bmatrix} \ten a_{k,>}^* \ten a_{k,>}^* \ten a_{k,j_1} \ten a_{k,>} \end{bmatrix}^{j_1=1,\ldots,k}, \\
\rep{b}_k &= \begin{bmatrix} \ten a_{k,i_1}^* \ten a_{k,>}^* \ten a_{k,j_1} \ten a_{k,>} \end{bmatrix}^{i_1=1,\ldots,k; j_1=1,\ldots,k}, & &\\
\rep{c}_k^* &= \begin{bmatrix} \ten a_{k,i_1}^* \ten a_{k,i_1}^* \ten a_{k,>} \ten a_{k,>} \end{bmatrix}^{i_1=1,\ldots,k; i_2=i_1,\ldots,k}, \\
\intertext{as well as}
\rep{c}_k &= \begin{bmatrix} \ten a_{k,>}^* \ten a_{k,>}^* \ten a_{k,j_1} \ten a_{k,j_2} \end{bmatrix}^{j_1=1,\ldots,k; j_2=j_1,\ldots,k}, \\
\rep{e}_k^* &= \begin{bmatrix} \displaystyle \sum_{\substack{ i_1,j_1,j_2=1 \\ j_1<j_2 } }^k \tilde{v}_{i_1 i_2 j_1 j_2} \ten a_{k,i_1}^* \ten a_{k,i_2}^* \ten a_{k,j_1} \ten a_{k,j_2} \end{bmatrix}^{i_2=k+1,\ldots,K}, \\
\rep{e}_k &= \begin{bmatrix} \displaystyle \sum_{\substack{ i_1,i_2,j_1=1 \\ i_1<i_2 } }^k \tilde{v}_{i_1 i_2 j_1 j_2} \ten a_{k,i_1}^* \ten a_{k,i_2}^* \ten a_{k,j_1} \ten a_{k,j_2} \end{bmatrix}^{j_2=k+1,\ldots,K},\\
\ten{D}_n &= \sum_{\substack{ i_1,i_2,j_1,j_2=1 \\ i_1<i_2,\,j_1<j_2}}^k \tilde{v}_{i_1 i_2 j_1 j_2} \ten{a}_{k,i_1}^*\ten{a}_{k,i_2}^*\ten{a}_{k,j_1}\ten{a}_{k,j_2}.
\end{align*}
We again proceed by induction and show that 
\begin{equation*}
V_1 \SKP \dots \SKP V_k = \begin{bmatrix}
\ten{I}_{k} & \rep{a}_{k}^* & \rep{a}_{k} & \rep{b}_{k} & \rep{c}_{k}^* & \rep{c}_{k} & \rep{e}_{k} & \rep{e}_{k}^* & \ten{D}_{k}
\end{bmatrix}.
\end{equation*}
This holds for $k=1$ and for $k=2,\ldots, \frac{K}{2}$, we want to show
\begin{multline*}
\begin{bmatrix}
\ten{I}_{k-1} & \rep{a}_{k-1}^* & \rep{a}_{k-1} & \rep{b}_{k-1} &\rep{c}_{k-1}^* & \rep{c}_{k-1} & \rep{e}_{k-1} & \rep{e}_{k-1}^* & \ten{D}_{k-1}
\end{bmatrix} \SKP \begin{bmatrix}
V_k^{1,1} & V_k^{1,2}\\ 0 &V_k^{2,2}
\end{bmatrix} \\
= \begin{bmatrix}
\ten{I}_{k} & \rep{a}_{k}^* & \rep{a}_{k} & \rep{b}_{k} & \rep{c}_{k}^* & \rep{c}_{k} & \rep{e}_{k} & \rep{e}_{k}^* & \ten{D}_{k}
\end{bmatrix}.
\end{multline*}
The product with $V_k^{1,1}$ is rather straightforward. For the products with $V_k^{1,2}$ and $V_k^{2,2}$, we calculate
\begin{align*}
\rep{a}_{k-1}^* \SKP \, (W_{V,k}^1\uparrow A^*A) &= \begin{bmatrix} \displaystyle \sum_{i_1=1}^{k-1} \tilde v_{i_1kkj_2} \ten a_{k,i_1}^* \ten a_{k,>}^* \ten a_{k,>} \ten a_{k,>} \otimes A^*A \end{bmatrix}^{j_2=k+1,\ldots, K} \\
&= \begin{bmatrix} \displaystyle \sum_{i_1=1}^{k-1} \tilde v_{i_1kkj_2} \ten a_{k,i_1}^* \ten a_{k,k}^* \ten a_{k,k} \ten a_{k,>} \end{bmatrix}^{j_2=k+1,\ldots, K},\\
\rep{b}_{k-1} \SKP \, (W_{V,k}^3\uparrow A^*) &= \begin{bmatrix} \displaystyle \sum_{i_1,j_1=1}^{k-1} \tilde v_{i_1kj_1j_2} \ten a_{k,i_1}^* \ten a_{k,>}^* \ten a_{k,j_1} \ten a_{k,>} \otimes A^* \end{bmatrix}^{j_2=k+1\cdots K} \\
&= \begin{bmatrix} \displaystyle \sum_{i_1,j_1=1}^{k-1} \tilde v_{i_1kj_1j_2} \ten a_{k,i_1}^* \ten a_{k,k}^* \ten a_{k,j_1} \ten a_{k,>} \end{bmatrix}^{j_2=k+1\cdots K} \\
\intertext{and}
\rep{c}_{k-1}^* \SKP \, (W_{V,k}^6\uparrow A) &= \begin{bmatrix} \displaystyle \sum_{\substack{i_1,i_2=1 \\ i_1 < i_2}}^{k-1} \tilde v_{i_1i_2kj_2} \ten a_{k,i_1}^* \ten a_{k,i_1}^* \ten a_{k,>} \ten a_{k,>} \otimes A \end{bmatrix}^{j_2=k+1,\ldots, K} \\
&= \begin{bmatrix} \displaystyle \sum_{\substack{i_1,i_2=1 \\ i_1 < i_2}}^{k-1} \tilde v_{i_1i_2kj_2} \ten a_{k,i_1}^* \ten a_{k,i_1}^* \ten a_{k,k} \ten a_{k,>} \end{bmatrix}^{j_2=k+1,\ldots, K} \\
\intertext{as well as}
\rep{e}_{k-1} \SKP \rep{S}_{K-k} &= \begin{bmatrix} \displaystyle \sum_{\substack{i_1,i_2,j_1=1 \\ i_1 < i_2}}^{k-1} \tilde v_{i_1i_2j_1j_2} \ten a_{k,i_1}^* \ten a_{k,i_2}^* \ten a_{k,j_1} \ten a_{k,j_2} \otimes S \end{bmatrix}^{j_2=k+1,\ldots, K} \\
&= \begin{bmatrix} \displaystyle \sum_{\substack{i_1,i_2,j_1=1 \\ i_1 < i_2}}^{k-1} \tilde v_{i_1i_2j_1j_2} \ten a_{k,i_1}^* \ten a_{k,i_2}^* \ten a_{k,j_1} \ten a_{k,j_2} \end{bmatrix}^{j_2=k+1,\ldots, K} , \\
\end{align*}
and with this 
\begin{equation*}
\rep{a}_{k-1}^* \SKP \,(W_{V,k}^1\uparrow A^*A) + \rep{b}_{k-1} \SKP \,(W_{V,k}^3\uparrow A^*) + \rep{c}_{k-1}^* \SKP \, (W_{V,k}^6\uparrow A) + \rep{e}_{k-1} \SKP \rep{S}_{K-k}= \rep{e}_k.
\end{equation*}
Similarly, we show
\begin{equation*}
\rep{a}_{k-1} \SKP \,(W_{V,k}^2\uparrow A^*A) + \rep{b}_{k-1} \SKP \,(W_{V,k}^4\uparrow A) + \rep{c}_{k-1} \SKP \,(W_{V,k}^7\uparrow A^*) + \rep{e}_{k-1}^* \SKP \rep{S}_{K-k} = \rep{e}_k^* 
\end{equation*}
and
\begin{multline*}
\rep{b}_{k-1} \SKP \,(W_{V,k}^5\uparrow A^*A) + \sum_{\substack{ i_1,i_2,j_1=1 \\ i_1<i_2 } }^k \tilde{v}_{i_1 i_2 j_1 j_2} \ten a_{k,i_1}^* \ten a_{k,i_2}^* \ten a_{k,j_1} \ten a_{k,k} \otimes A + \\ \sum_{\substack{ i_1,j_1,j_2=1 \\ j_1<j_2 } }^k \tilde{v}_{i_1 i_2 j_1 j_2} \ten a_{k,i_1}^* \ten a_{k,k}^* \ten a_{k,j_1} \ten a_{k,j_2} \otimes A^* + \sum_{\substack{ i_1,i_2,j_1,j_2=1 \\ i_1<i_2,\,j_1<j_2}}^k \tilde{v}_{i_1 i_2 j_1 j_2} \ten{a}_{k,i_1}^*\ten{a}_{k,i_2}^*\ten{a}_{k,j_1}\ten{a}_{k,j_2} \otimes I = \ten{D}_k.  
\end{multline*}

We proceed in the same fashion for $k = \frac{K}{2} +1 ,\ldots, K$. We define 
$\ten{\tilde{I}}_{k}$, $\rep{\tilde{a}}_{k}^*$, $\rep{\tilde{a}}_{k}$, $\rep{\tilde{b}}_{k}$, $\rep{\tilde{c}}_{k}^*$, $\rep{\tilde{c}}_{k}$, $\rep{\tilde{e}}_{k}$, $\rep{\tilde{e}}_{k}^*$, and $\ten{\tilde{D}}_{k}$
accordingly, such that
\begin{align*}
V_k \SKP \begin{bmatrix}
	\ten{\tilde{I}}_{k+1} \\ \rep{\tilde{a}}_{k+1}^* \\ \rep{\tilde{a}}_{k+1} \\ \rep{\tilde{b}}_{k+1}  \\ \rep{\tilde{c}}_{k+1}^* \\ \rep{\tilde{c}}_{k+1} \\ \rep{\tilde{e}}_{k+1} \\ \rep{\tilde{e}}_{k+1}^* \\ \ten{\tilde{D}}_{k+1}
\end{bmatrix}
=& \begin{bmatrix}
	\ten{\tilde{I}}_{k} \\ \rep{\tilde{a}}_{k}^* \\\rep{\tilde{a}}_{k} \\ \rep{\tilde{b}}_{k} \\ \rep{\tilde c}_{k}^* \\ \ten{\tilde c}_{k} \\ \rep{\tilde{e}}_{k} \\\rep{\tilde{e}}_{k}^* \\\ten{\tilde{D}}_{k}
\end{bmatrix}.
\end{align*}
It remains to show that 
\begin{equation}\label{eq:fullD}
\ten{D}=\begin{bmatrix}
\ten{I}_{K/2} & \rep{a}_{K/2}^* & \rep{a}_{K/2} & \rep{b}_{K/2} & \rep{c}_{K/2}^* & \rep{c}_{K/2} & \rep{e}_{K/2} & \rep{e}_{K/2}^* & \ten{D}_{K/2}
\end{bmatrix}M_V\SKP  \begin{bmatrix}
\ten{\tilde{I}}_{K/2+1} \\ \rep{\tilde{a}}_{K/2+1}^* \\\rep{\tilde{a}}_{K/2+1} \\ \rep{\tilde{b}}_{K/2+1} \\ \rep{c}_{K/2+1}^* \\ \rep{c}_{K/2+1} \\ \rep{\tilde{e}}_{K/2+1} \\\rep{\tilde{e}}_{K/2+1}^* \\\ten{\tilde{D}}_{K/2+1}\end{bmatrix}.
\end{equation}
This holds due to the anti-diagonal structure of $M_V$ and with an explicit calculation of $\rep{b}_{K/2} W_{V,k}^8 \rep{b}_{K/2 + 1}$, $\rep{c}_{K/2}^* W_{V,k}^9 \rep{c}_{K/2 + 1}$, and $\rep{c}_{K/2} W_{V,k}^{10} \rep{c}_{K/2 + 1}^*$.

Finally, we consider representation ranks. For non-sparse $\tilde V$, the number of columns in $V_k^{1,1}$ and $V_k^{1,2}$ are given by
\[
1+3k+4\binom{k}{2}\quad\text{and} \quad 2(K-k)+1,
\]
respectively. So for $k = \frac{K}{2}$ we get
\[
2+5\frac{K}{2}+4\binom{\frac{K}{2}}{2} = 2+5\frac{K}{2}+\frac{4}{2}\frac{K^2}{4}-\frac{4}{2} \frac{K}{2} = 2+\frac{3}{2}K+\frac{1}{2}K^2
\]
which is an upper bound for the MPO-rank of $\ten D$. 

For the sparse case \eqref{eq:localityassumption}, we only consider at the matrix $M_V$, since it determines the highest rank in the representation of $\ten D$. In this case,
 the contributions to the rank of the left and right components in~\eqref{eq:fullD} are as follows: Each of the two combinations of components $\ten{I}_{K/2}$ and $\ten{\tilde{D}}_{K/2+1}$ as well as $\ten{D}_{K/2}$ and $\ten{\tilde{I}}_{K/2+1}$ contributes one to the rank. Each of the four combinations of $\rep{a}_{K/2}^*$ and $\rep{\tilde{e}}^*_{K/2+1}$; $\rep{a}_{K/2}$ and $\rep{\tilde{e}}_{K/2+1}$; $\rep{e}_{K/2}$ and $\rep{\tilde{a}}_{K/2+1}$; $\rep{e}_{K/2}^*$ and $\rep{\tilde{a}}_{K/2+1}^*$ contributes $d-1$. In the case of $\rep{b}_{K/2}$ and $\rep{\tilde{b}}_{K/2+1}$, we obtain
 \[
   \sum_{\ell=1}^{d-1}\min\{2\ell-1,2(d-\ell)-1\}=\begin{cases}
\frac 12 (d-1)^2 +d & d \textrm{ odd, }\\
 \frac 12 d^2 & d \textrm{ even; }
\end{cases}
 \]
 in the case of $\rep{c}_{K/2}^*$ and $\rep{\tilde{c}}_{K/2+1}$,
 \[
 \sum_{\ell=2}^{d-1}\min\{\ell-1,d-\ell\}=\begin{cases}
2\binom{\frac 12 (d-1)}{2} +\frac 12 (d-1) & d \textrm{ odd, }\\
2\binom{\frac 12 d}{2} & d \textrm{ even; }
\end{cases}
 \]
 and in the case of $\rep{c}_{K/2}$ and $\rep{\tilde{c}}_{K/2+1}^*$,
 \[
   \sum_{\ell=2}^{d-1}\min\{\ell-1,d-\ell\}=\begin{cases}
2\binom{\frac 12 (d-1)}{2} +\frac 12 (d-1) & d \textrm{ odd, }\\
2\binom{\frac 12 d}{2} & d \textrm{ even, }
\end{cases}
 \]
 and summing up these contributions completes the proof.
\end{proof}

\end{document}